\documentclass[11pt]{article}
\pdfmapfile{+euler.map}
\usepackage[utf8]{inputenc}
\usepackage{cite}
\usepackage[colorlinks=true
]{hyperref}
\hypersetup{urlcolor=blue, citecolor=red, linkcolor=blue}
\usepackage{mleftright}
\usepackage{stmaryrd}
\SetSymbolFont{stmry}{bold}{U}{stmry}{m}{n}
\usepackage{microtype}

\usepackage{etoolbox}
\usepackage{bbm}
\usepackage{colortbl}

\usepackage{graphicx}
\usepackage[active]{srcltx}

\usepackage{amsmath,amssymb}

\numberwithin{equation}{section}

\usepackage{bm}
\usepackage{enumerate}
\usepackage{amsthm}
\usepackage{mathrsfs}
\usepackage{mathtools}
\mathtoolsset{showonlyrefs}

\usepackage{tikz}

\usepackage{xcolor}
\DeclareUnicodeCharacter{0301}{*************************************}

\hypersetup{urlcolor=blue, citecolor=red, linkcolor=blue}

\usepackage[a4paper, left=25mm, right=25mm, top=30mm, bottom=25mm]{geometry}

\usepackage{marginnote}

\usepackage{imakeidx}
\makeindex[intoc]

\usepackage[nottoc]{tocbibind}
\usepackage[colorinlistoftodos]{todonotes} 

\newcommand{\bomega}{{\boldsymbol{\omega}}}

\theoremstyle{plain}
\newtheorem{theorem}{Theorem}[section]
\newtheorem{proposition}[theorem]{Proposition}
\newtheorem{corollary}[theorem]{Corollary}

\theoremstyle{definition}
\newtheorem{definition}[theorem]{Definition}
\newtheorem{remark}[theorem]{Remark}
\newtheorem{example}[theorem]{Example}

\AtBeginEnvironment{remark}{%
  \pushQED{\qed}%
}
\AtEndEnvironment{remark}{\popQED\endexample}

\AtBeginEnvironment{example}{%
  \pushQED{\qed}%
}
\AtEndEnvironment{example}{\popQED\endexample}

\newcommand{\quotient}[2]{{\raisebox{.2em}{$#1$}\left/\raisebox{-.2em}{$#2$}\right.}}

\newcommand\restr[2]{{
  \left.\kern-\nulldelimiterspace #1 \right|_{#2} 
}}

\newcommand{\R}{\mathbb{R}}
\newcommand{\N}{\mathbb{N}}
\renewcommand{\d}{\mathrm{d}}

\newcommand{\Cinfty}{\mathscr{C}^\infty}
\newcommand{\T}{\mathrm{T}}

\newcommand{\cT}{\mathrm{T}^\ast}
\newcommand{\Id}{\mathrm{Id}}
\newcommand{\dd}{\mathop{}\!\mathrm{d}}
\newcommand{\rank}{\mathrm{rk}}

\newcommand*{\inn}[1]{\iota_{#1}}

\newcommand{\Lie}{\mathscr{L}}

\newcommand{\X}{\mathfrak{X}}

\newcommand{\D}{\mathcal{D}}
\newcommand{\V}{\mathcal{V}}

\newcommand{\abs}[1]{\left|#1\right|}

\newcommand{\bfX}{\mathbf{X}}

\newcommand{\qeddiamond}{\hfill $\diamond$}

\newcommand{\parder}[2]{\frac{\partial #1}{\partial #2}}
\newcommand{\dparder}[2]{\dfrac{\partial #1}{\partial #2}}
\newcommand{\tparder}[2]{\partial #1/\partial #2}

\newcommand{\boplus}{{\textstyle\bigoplus}}

\DeclareMathOperator{\rk}{rk}
\DeclareMathOperator{\cork}{co-rk}

\DeclareMathAlphabet{\mathpzc}{OT1}{pzc}{m}{it}
\def\d{\mathrm{d}}

\DeclareMathOperator{\pr}{pr}
 
\DeclareMathOperator{\corank}{co-rk}

\newcommand\finish{\xqed{ }}
\newcommand\xqed[1]{%
  \leavevmode\unskip\penalty9999 \hbox{}\nobreak\hfill
  \quad\hbox{#1}}

\usepackage{xcolor}

\usepackage{fancyhdr}

\begin{document}


\vspace{5em}

{\huge\sffamily\raggedright Foundations on $k$-contact geometry
}
\vspace{2em}

{\large\raggedright
    \today
}

\vspace{2em}

{\Large\raggedright\sffamily
    Javier de Lucas\footnote{Corresponding author.}
}\vspace{1mm}\newline
{\raggedright    
    Department of Mathematical Methods in Physics, University of Warsaw\\
    ul. Pasteura 5, 02-093, Warszawa, Poland\\
    e-mail: \href{mailto:javier.de.lucas@fuw.edu.pl}{javier.de.lucas@fuw.edu.pl} --- orcid: \href{https://orcid.org/0000-0001-8643-144X}{0000-0001-8643-144X}
}

\bigskip

{\Large\raggedright\sffamily
    Xavier Rivas
}\vspace{1mm}\newline
{\raggedright
    Department of Computer Engineering and Mathematics, Universitat Rovira i Virgili\\
    Avinguda Països Catalans 26, 43007,  Tarragona, Spain\\
    e-mail: \href{mailto:xavier.rivas@urv.cat}{xavier.rivas@urv.cat} --- orcid: \href{https://orcid.org/0000-0002-4175-5157}{0000-0002-4175-5157}
}

\bigskip

{\Large\raggedright\sffamily
    Tomasz Sobczak
}\vspace{1mm}\newline
{\raggedright    
    Department of Mathematical Methods in Physics, University of Warsaw\\
    ul. Pasteura 5, 02-093, Warszawa, Poland\\
    e-mail: \href{mailto:t.sobczak2@uw.edu.pl}{t.sobczak2@uw.edu.pl} --- orcid: \href{https://orcid.org/0009-0002-9577-0456}{0009-0002-9577-0456}
}

\vspace{2em}

{\large\bf\raggedright
    Abstract
}\vspace{1mm}\newline
{\raggedright
$k$-Contact geometry is a generalisation of contact geometry to analyse field theories. As a main goal, we develop an approach to $k$-contact geometry based on distributions that are distributionally maximally non-integrable and admit, locally, $k$ commuting supplementary Lie symmetries: the $k$-contact distributions. More specifically, we investigate the existence of Reeb vector fields for $k$-contact distributions, $k$-contact Lie groups are defined and studied, and, after analysing compact manifolds, a Weinstein conjecture for a product of $k$ compact contact manifolds is discussed. Polarisations for $k$-contact distributions are introduced and it is shown that a polarised $k$-contact distribution is locally diffeomorphic to the Cartan distribution of the first-order jet bundle over a fibre bundle of rank $k$. We relate $k$-contact manifolds to presymplectic and $k$-symplectic manifolds on fibre bundles of larger dimension and define types of submanifolds in $k$-contact geometry. We study Hamilton--De Donder--Weyl equations on $k$-contact Lie groups. A theory of $k$-contact Hamiltonian vector fields is developed, and we describe characteristics of Lie symmetries for first-order partial differential equations in a $k$-contact Hamiltonian manner. We use our techniques to analyse Hamilton--Jacobi and Dirac equations, types of Engel and Goursat distributions, physical problems, and control systems. The work also presents introductory results so as to make it more self-contained.}
\bigskip

\medskip

\medskip







\noindent {\bf Authors' contributions:} All authors contributed to the study conception and design. The manuscript was written  and revised by all authors. All authors read and approved the final version.
\medskip

\noindent {\bf Competing Interests:}  The authors have no competing interests to declare. 
\medskip

\noindent {\bf Data statement:}  No data were created or analysed during this study. 
\newpage

{\setcounter{tocdepth}{2}
\def\baselinestretch{1}
\small
\def\addvspace#1{\vskip 1pt}
\parskip 0pt plus 0.1mm
\tableofcontents
}

\pagestyle{fancy}

\fancyhead[L]{Foundations on $k$-contact geometry}    
\fancyhead[C]{}                  
\fancyhead[R]{J. de Lucas, X. Rivas and T. Sobczak}       

\fancyfoot[L]{}     
\fancyfoot[C]{\thepage}                  
\fancyfoot[R]{}            

\renewcommand{\headrulewidth}{0.1pt}  
\renewcommand{\footrulewidth}{0pt}    

\renewcommand{\headrule}{%
    \vspace{3pt}                
    \hrule width\headwidth height 0.4pt 
    \vspace{0pt}                
}

\setlength{\headsep}{30pt}  

\section{Introduction}

Contact geometry was initiated by Sophus Lie in his study of  differential equations via contact transformations \cite{Gei_01,Kay_23,LS_96}. Classical contact geometry is rooted in the contact distribution notion, namely a maximally non-integrable distribution \cite{AM_78,Etn_06,Gei_01,Gei_08}. In particular, Cartan distributions for first-order jet manifolds over a fibre bundle of rank one are contact distributions. Mathematically, contact geometry leads to the analysis of many interesting topological and differential geometric problems.  For example, classification problems exist concerning different types of three-dimensional contact manifolds, the examination of transverse and Legendrian knots, the exploration of the Reeb vector field dynamics and the Weinstein conjecture on their closed trajectories in compact manifolds, Floer homology methods, Marsden--Weinstein reductions, the theory of Hamiltonian contact systems, and so on \cite{Alb_89,BH_92,CF_09,CF_11,LL_19,EH_01,Gei_06,GG_23,Thu_97,Wei_79,Wil_02}.

Contact geometry has garnered consistent and robust interest since the close of the 20th century \cite{Gei_08}. Moreover, there has been a recent boost in the study of contact dynamics in relation to the study of dissipative mechanical systems \cite{BCT_17,BLMP_20,LLMR_21,LL_19,GGMRR_20a,Kho_13}. Contact geometry has even been investigated to understand models for the visual cortex of the brain \cite{Hof_89,LM_23}.

Contact distributions are frequently studied via a differential one-form whose kernel is equal, at least locally, to the contact distribution: a so-called {\it contact form}. Nowadays, attention is paid to contact distributions that are the kernel of a globally defined contact form, namely one defined on the whole manifold, giving rise to the so-called {\it co-oriented contact manifolds} \cite{BCT_17,LL_19} or {\it strict contact manifolds} \cite{HW_18}. When a contact distribution is equal to the kernel of a globally defined contact form, it is possible to define interesting associated notions, like {\it Reeb vector fields}, named after the French mathematician George Reeb and appearing in the study of periodic orbits \cite{Wei_79}. Even when an associated global contact form exists, certain properties of contact distributions, like their Marsden--Weinstein reductions, can be more appropriately described using contact distributions \cite{GG_22,Wil_02}. For instance, this is due to the fact that the symmetries of a contact distribution need not be symmetries of an associated contact form.

$k$-Contact geometry recently appeared as one of the possible generalisations of contact geometry to deal with field theories \cite{GGMRR_20,GGMRR_21,Riv_22,Riv_23}. The initial idea was to study an extension of the contact form notion to investigate classical field theories. This gave rise to the study of Hamilton--De Donder--Weyl equations via $k$-vector fields and a series of  $k$ differential one-forms giving rise to the so-called {\it $k$-contact structures}. These notions were used to study several physical field equations, Darboux theorems, and other related physical and mathematical topics \cite{GGMRR_20,GGMRR_21,RSS_24}.  

This paper begins by recalling the form-based $k$-contact formalism and its related concepts. We present a new approach based on {\it $k$-contact forms}, i.e. differential one-forms taking values in $\mathbb{R}^k$ defined on an open subset $U$ of a manifold $M$, let us say $\bm\eta$, whose kernel is a non-zero regular distribution of corank $k$ and satisfying that $\ker \bm \eta\oplus \ker \d\bm \eta=\T U$. The $k$-contact forms give a vector-valued formulation of $k$-contact structures \cite{GGMRR_20,GGMRR_21,Riv_22,Riv_23} that allows for describing Cartan distributions and jet manifold structures \cite{deLucas2026JetBundles}. A manifold endowed with a globally defined $k$-contact form is called a {\it co-oriented $k$-contact manifold}. Note that a one-contact form is just a contact form, and a co-oriented one-contact manifold is just a co-oriented contact one.  We stress in this work that $k$-contact forms are related to other constructions appearing in the literature \cite{Almeida2018,Almeida2024} or just as constructions, e.g. in certain non-Lagrangian coisotropic submanifolds of contact type \cite{Ginzburg2007Coisotropic,Ginzburg2011Maslov} or toric contact geometry \cite{ApostolovEtAl2020}. This work presents a brief description of the relations of our $k$-contact manifolds with other related notions in the literature \cite{Almeida2018,Almeida2024,LeokSardonZhao2026,Ginzburg2007Coisotropic,Ginzburg2011Maslov}. 

The main goal of this paper is to provide a distributional approach to $k$-contact geometry generalising some of the ideas from contact geometry.  In other words, instead of focusing on $k$-contact structures as in the pioneering literature \cite{GGMRR_20,GGMRR_21,Riv_22,Riv_23}, we investigate the properties of the kernels of $k$-contact forms.  Regular distributions on a manifold that are locally described by the kernel of a $k$-contact form are called {\it $k$-contact distributions}. A {\it $k$-contact manifold} is a manifold equipped with a $k$-contact distribution. This work suggests that $k$-contact distributions provide new intrinsic viewpoints for the analysis of known and new aspects of $k$-contact geometry.   

In the case of co-oriented $k$-contact manifolds, which admit a globally defined $k$-contact form $\bm \eta$, the distribution $\ker \d\bm\eta$, the so-called {\it Reeb distribution}, is integrable of rank $k$ and it admits a privileged basis of $k$ linearly commuting and globally defined Lie symmetries of the distribution $\ker \bm\eta$. The elements of such a basis are the referred to as  
{\it Reeb vector fields} of $\bm \eta$.

Theorem 3.14 proves that $k$-contact distributions amount to maximally non-integrable distributions admitting, on a neighbourhood of every point, a family of $k$ commuting Lie symmetries spanning a supplementary distribution to the $k$-contact distribution. This gives an intrinsic characterisation of $k$-contact distributions independent of the use of $k$-contact forms. Note that our definition of maximally non-integrability is in the distributional sense given, for instance, by L.\,Vitagliano in \cite{Vit_15}, which is based on intrinsic properties of the vector fields taking values in the distribution, and it recovers the usual meaning in contact geometry for regular non-zero  distributions of corank one. It is worth noting that maximally non-integrability is not enough to ensure that a regular distribution is locally the kernel of a $k$-contact form for $k>1$. This is illustrated by Counterexample \ref{thm:exis-max-nonintegr}, the use of Lie flags \cite{PR_01b}, and the analysis of the Lie symmetries of regular distributions. 

The properties of $k$-contact manifolds significantly differ from that of contact manifolds. For example, $k$-contact manifolds are not necessarily odd-dimensional, co-oriented $k$-contact manifolds do not require orientability, whereas co-oriented contact manifolds do. Moreover, given a contact distribution $\mathcal{D}$ on a manifold $M$, any one-form $\eta$ on $U$ such that $\ker \eta=\mathcal{D}|_U$ is a contact form, and its multiplication by a non-vanishing function gives rise to a new contact one-form for the same contact distribution. In the $k$-contact realm, this does not happen, and a $k$-contact distribution given locally by the kernel of a one-form, $\bm\zeta$, taking values in $\mathbb{R}^k$ does not need to be a $k$-contact form (see Example \ref{ex:DR-different-rank}). Additionally, we provide techniques to describe the properties of $\bm \zeta$.

This work shows that every first-order jet manifold $J^1(M,E)$ of a bundle $E$ over $M$ of rank $k$ is naturally a $k$-contact manifold associated with the $k$-contact distribution given by the Cartan distribution. We show that Cartan distributions on a first-order jet manifold are not a canonical example of $k$-contact distribution, as most  $k$-contact manifolds are not diffeomorphic to a first-order jet manifold, even locally. Moreover, it is proved that the Lie brackets of vector fields taking values in a $k$-contact distribution do not need to span the tangent bundle to the manifold, as happens for contact distributions. Nevertheless, such Lie brackets span a distribution that is strictly larger than the $k$-contact distribution at each point (see Proposition \ref{prop:ker_eta}). Examples of $k$-contact distributions appearing in non-holonomic systems, control theory, the theory of Engel and Goursat distributions, and other fields, are given \cite{PR_98,PR_01,PN_05}. 

We discuss some elementary topological obstructions to the
existence of a globally defined $k$-contact form. These follow
from the existence of $k$ globally defined, linearly independent
Reeb vector fields and are illustrated using Euler characteristics,
the Hairy Ball theorem \cite{EG_79}, and the Radon--Hurwitz
numbers \cite{Ada_62}.

Next, co-oriented $k$-contact compact manifolds are investigated. It is proven that two-contact four-dimensional manifolds are the lowest dimensional case of $k$-contact manifolds that are not contact manifolds. In this respect, the analysis of two-contact four-dimensional manifolds is the natural analogue in $k$-contact geometry of the study of three-dimensional compact manifolds \cite{Cal_11}. In particular, a two-contact four-dimensional compact manifold is studied in this work (see Example \ref{Ex:U2}).

We study the integral submanifolds of Reeb distributions for co-oriented $k$-contact compact manifolds paying special attention to two-contact and three--contact compact manifolds like in Examples \ref{Ex:SU3}, \ref{Ex:SU4}, and \ref{Ex:T*R^{4}}. In all these cases, the Reeb distribution has rank two or three and admits some integral submanifolds that are diffeomorphic to a two-dimensional or three-dimensional torus, respectively. This suggests us to discuss the existence of integral submanifolds of the Reeb distribution diffeomorphic to $k$-dimensional tori for every co-oriented $k$-contact compact manifold. For one-contact manifolds, this question retrieves the Weinstein conjecture for compact co-oriented contact manifolds \cite{Tau_07,Wei_79}, which appears here as a particular case. Nevertheless, one counterexample showing that the Reeb distribution on a compact co-oriented two-contact manifold of dimension four does not admit any leaf homeomorphic to a two-dimensional torus is given. This shows that the Weinstein conjecture cannot be straightforwardly extended to the $k$-contact realm for $k>1$. Instead, we discuss a product formulation of $k$-contact manifolds that is equivalent to the classical Weinstein conjecture. As a byproduct of our study of compact $k$-contact manifolds, we propose a definition of {\it $k$-contact Lie groups}, which are co-oriented $k$-contact manifolds defined on Lie groups admitting a global $k$-contact form that is left-invariant.

A $k$-contact manifold may not have Darboux-like coordinates \cite{GGMRR_20}, but the existence of a polarisation ensures their local existence. 
After recalling certain previous notions, $k$-contact Darboux-like theorems, and examples of polarisations for $k$-contact structures in the literature \cite{Riv_22}, we define in the middle of Section \ref{sec:polarised-k-contact-distributions} a new polarisation notion for $k$-contact distributions. This generalises the polarisation notion for $k$-contact structures in \cite{GGMRR_20}.  We then show that every $k$-contact manifold admitting a polarisation is locally diffeomorphic to a first-order jet manifold $J^1(M,E)$, where $E\rightarrow M$ is a fibre bundle of rank $k$ and the Cartan distribution becomes a $k$-contact distribution. This generalises to fibre bundles of arbitrary rank the typical result appearing in contact geometry. Our $k$-contact Darboux theorem for $k$-contact distributions (Theorem 5.6) is more general than its $k$-contact counterpart and has a distributional proof that is independent of the ideas used for the classical one \cite{Riv_22,GGMRR_20} based on reducing a $k$-contact form to a $k$-symplectic one. Indeed, Theorem 5.6 can be applied to $k$-contact manifolds that are not co-oriented.  Theorem~\ref{Th:PolCokcontaDar} also clarifies the role of
the prescribed polarisation: for $k>1$, it is unique and
the Darboux coordinates for the fixed form are adapted to it.
For $k=1$, simultaneous adaptation may fail, as shown in
Example~\ref{Ex:CounterDarboux}. Theorem~\ref{Th:DarkMan}
instead normalises the distribution and its prescribed
polarisation by choosing an associated local $k$-contact form, without
assuming global co-orientation. The independent proof of Theorem~\ref{Th:PolCokcontaDar}
also corrects the simultaneous-adaptation assertion and the proofs in
\cite[Theorem~3.10]{GGMRR_20} and
\cite[Theorem~7.1.11]{Riv_22}, distinguishing the cases
$k=1$ and $k>1$.

Next, in Section \ref{Sec:kSymplecticCover}, it is proved that every co-oriented $k$-contact $m$-dimensional manifold gives rise to a $k$-symplectic homogeneous $(m+1)$-dimensional total manifold of an $\mathbb{R}_\times$-principal bundle with certain additional properties (see Theorem \ref{thm:k-symplectic-bundle}). This is called a $k$-symplectic cover of the co-oriented $k$-contact manifold.  This generalises the approach given in \cite{GGMRR_20,GGMRR_21} for standard contact manifolds to the $k$-contact realm. Nevertheless, it is remarkable that a $k$-symplectic cover does not need to be well-defined globally for $k$-contact manifolds with $k>1$ that are not co-oriented. On the other hand, we study properties for $k$-symplectic homogeneous manifolds that allow us to understand them as $k$-symplectic covers.  Moreover, we define isotropic and Legendrian submanifolds for $k$-contact manifolds, which have been lately used and further analysed in \cite{HLM_25,DLRS_26}, and their relations to isotropic and Lagrangian submanifolds in $k$-symplectic covers.

Co-oriented $k$-contact manifolds are relevant cases of $k$-contact manifolds since, among other reasons, they admit a Hamiltonian theory for $k$-vector fields that can be used to analyse first-order partial differential equations. In this respect, after reviewing the standard theory of $k$-contact Hamiltonian $k$-vector fields, it is shown that every such a $k$-contact $k$-vector field gives rise to $k$-symplectic Hamiltonian $k$-vector fields in the $k$-symplectic cover of the co-oriented $k$-contact manifold. This is interesting as a link between the study of such $k$-contact Hamiltonian $k$-vector fields and $k$-symplectic geometry, which  has also been recently studied in  \cite{10.1007/978-3-032-03924-8_36}. It is worth noting that we review the theory of Hamilton--De Donder--Weyl equations. As a new result, we study the Hamilton--De Donder--Weyl equations on co-oriented $k$-contact Lie groups in Section \ref{Sec:HamkConLie}. Types of submanifolds in $k$-contact manifolds are described, and we highlight that their characterisation is indeed independent of using a $k$-contact form to describe the $k$-contact distribution. These submanifolds had lately applications in the theory of pseudo-gauge transformations and Hamilton--Jacobi equations \cite{DLRS_26,HLM_25}.

Next, we devise a {\it $k$-contact vector field} notion on $k$-contact manifolds. If $\mathcal{D}=\ker\bm \eta$ for a globally defined $k$-contact form $\bm\eta$, this amounts to 
a new $k$-contact Hamiltonian vector field notion on co-oriented $k$-contact manifolds, whose main features are analysed. This gives rise to the study of a space of $k$-contact Hamiltonian $k$-functions, which can be endowed with a natural Lie bracket. Every $k$-contact manifold is related to a presymplectic manifold of larger dimension, a so-called {\it presymplectic cover}, giving a new approach to the study of $k$-contact geometry via presymplectic manifolds. Moreover, every $k$-contact Hamiltonian vector field is related to a presymplectic Hamiltonian vector field relative to its presymplectic cover. The extended presymplectic Hamiltonian vector field is unique under very general conditions. In particular, they are always unique when the $k$-contact manifold is polarised.  
Our theory on $k$-contact Hamiltonian vector fields is used to study Lie symmetries of first-order partial differential equations and their characteristics \cite{Olv_93}, which are recovered as our newly defined {\it $k$-contact Hamiltonian $k$-functions}. Applications of our ideas to the characteristics of Lie symmetries for a Hamilton--Jacobi equation and a Dirac equation are presented. In particular, the spacetime translations of Dirac equations are described as eight-contact Hamiltonian vector fields on a first-order jet manifold $J^1(\mathbb{R}^4,\mathbb{R}^{4+8})$, which is endowed with a globally defined eight-contact form. More generally, the latter implies that every vector field preserving both the Cartan
distribution and the Dirac equation submanifold is
$\bm\eta$-Hamiltonian, since $\bm\eta$ is globally defined. 
Our approach can be employed to study first-order differential equations and their Lie symmetries \cite{Olv_93}, but just a brief $k$-contact reinterpretation of a known result is here given in Proposition \ref{Prop:Inv}. It is worth noting that our idea of a $k$-contact Hamiltonian vector field has been further used recently in \cite{LRS_24}. Finally, a brief description of other manners of relating $k$-contact manifolds to $k$-symplectic manifolds are given. 

The structure of the paper goes as follows. Section \ref{Sec:BasNot} recalls some basic notions on generalised subbundles, contact manifolds, $k$-vector fields, $k$-symplectic manifolds, and $k$-contact manifolds. Relations between different generalisations of contact geometry are analysed.  Section \ref{Sec:kConMan} introduces $k$-contact
distributions, proves their intrinsic
characterisation, and presents examples. 
Section \ref{Sec:Compact} analyses types of co-oriented $k$-contact compact manifolds and the existence of integral submanifolds of the Reeb distribution diffeomorphic to tori  of type $\mathbb{T}^k$. The relation of our results to a generalisation of the Weinstein conjecture to $k$-contact geometry is analysed. Next, Section \ref{sec:polarised-k-contact-distributions} defines polarisations for $k$-contact manifolds, their relation to Darboux theorems and first-order jet manifolds. $k$-Symplectic covers and diverse types of submanifolds of $k$-contact manifolds are introduced in Section \ref{Sec:kSymplecticCover}. Section \ref{Sec:Dynamics} studies $k$-contact Hamiltonian vector fields and $k$-contact vector fields, which is further used to study field equations and to recover the theory of characteristics for Lie symmetries of first-order partial differential equations as a $k$-contact Hamiltonian formalism. Moreover, applications to relevant first-order partial differential equations are given. Section \ref{Sec:kSym} describes another possible way of relating $k$-contact manifolds with $k$-symplectic manifolds. The compact examples, Lie-group dynamics, and alternative covers are complementary applications, rather than successive steps in the local recognition results. Finally, the conclusions of our work and prospects of future research are described in Section \ref{Sec:Conclusions}.

\section{Basic notions}
\label{Sec:BasNot}

This section provides an overview of the fundamental concepts that will be used in subsequent sections of the work. In particular, we will review the basics on generalised subbundles \cite{Lew_23}, contact geometry \cite{BH_16,Gei_08,Kho_13}, $k$-vector fields and their integral sections \cite{LSV_15} and $k$-symplectic geometry \cite{Awa_92,LSV_15,LV_13}. Next, a brief review on $k$-contact geometry using the here defined $k$-contact forms will be provided, and some new examples of $k$-contact forms are to be shown. Relevantly, we solve a minor technical problem about the original definition of a $k$-contact structure \cite{GGMRR_20,Riv_23} so as to recover properly contact geometry as a particular instance. 

First, let us fix some notation and basic notions to be used. Hereafter, $\{e_1,\ldots,e_k\}$ stands for a basis of $\mathbb{R}^k$. Natural numbers start at one. Manifolds are deemed to be finite-dimensional, non-zero dimensional, smooth, Hausdorff, and connected, unless otherwise stated. Moreover, $M$ is hereafter assumed to be an $m$-dimensional manifold. In the absence of other specifications, all mathematical entities are considered to be smooth.

Every differential form on a manifold $M$ taking values in $\mathbb{R}^k$ is represented by a Greek bold letter and $\Omega^p(M,\R^k)$ stands for the space of differential $p$-forms on $M$ taking values in $\R^k$. In addition, $\Omega^p(M)$ represents the space of differential $p$-forms on $M$. We will denote by $\X(M)$ the space of vector fields over $M$.

Moreover, if ${\bm\zeta}\in\Omega^p(M,\R^k)$, then ${\bm \zeta} = \sum_{\alpha=1}^k\zeta^\alpha\otimes e_\alpha$ for $k$ unique differential $p$-forms $\zeta^1,\ldots,\zeta^k$ on $M$. In what follows, the components of every ${\bm\zeta}$ are denoted by $\zeta^1,\ldots,\zeta^k\in\Omega^p(M)$. The \textit{inner contraction} of $\bm\zeta$ with a vector field $X$ on $M$ is
$$
    \inn{X} \bm\zeta = \langle \bm\zeta,X\rangle =  \sum_{\alpha=1}^k(\inn{X}\zeta^\alpha)\otimes e_\alpha\in\Omega^{p-1}(M,\R^k)\,,
$$
where $ \inn{X}\zeta^\alpha  $ stands for the inner contraction of $\zeta^\alpha\in \Omega^p(M)$ with the vector field $X$ for $\alpha=1,\ldots,k$. Let us define
$$
    \ker{\bm \zeta}_x = \left\{v_x\in \T_xM \mid \inn{v_x}\bm\zeta_x = \sum_{\alpha=1}^k(\inn{v_x}\zeta^\alpha_x)\otimes e_\alpha=0\right\}\,,\qquad \forall x\in M\,.
$$
We say that $\bm\zeta$ is {\it non-degenerate} if $\ker \bm\zeta=0$.

\subsection{Generalised subbundles}

Let us review the theory of generalised subbundles and other related notions of interest for our purposes \cite{Lew_23}.

\begin{definition}
    Let $E\rightarrow M$ be a vector bundle over $M$. Then, 
    \begin{itemize}
        \item A \textit{generalised subbundle} of $E\to M$ is a subset $D\subset E$ such that $D_x = D\cap E_x$ is a vector subspace of the fibre $E_x$ for every $x\in M$. We call {\it rank} of $D$ at $x\in M$ the dimension of the subspace $D_x\subset E_x$.
        \item A generalised subbundle $D$ is \textit{smooth} if it is locally spanned by a family of smooth sections of $E\rightarrow M$ taking values in $D$, i.e. if, for every $x\in M$, there exists an open neighbourhood $U$ of $x$ and a family of sections $e_1, \ldots, e_r$ defined on $U$ taking values on $D$ such that $D_{x'}=\langle e_1(x'), \ldots, e_r(x')\rangle$ for every $x'\in U$. It is worth noting that $e_1(x'),\ldots,e_r(x')$ do not need to be linearly independent at every $x'$ in $U$. 
        \item A generalised subbundle $D$ is \textit{regular} if it is smooth and has constant rank.
       \item A generalised subbundle $D$ is {\it parallelizable} if it admits a global basis of sections  taking values in $D$.
    \end{itemize}
\qeddiamond
\end{definition}

If $D$ is a regular generalised subbundle of constant rank $k$, we write $\rk D = k$. If $D$ has constant corank $p$, we denote $\cork D = p$. Finally, $\Gamma(D)$ stands for the space of sections of $E\rightarrow M$ taking values in $D$. We write $D|_U$ for the restriction of a generalised subbundle $D$ on $M$ to an open subset $U$ of $M$.

The {\it annihilator} of a generalised subbundle $D$ of the vector bundle $E\to M$ is the generalised subbundle $D^\circ=\bigsqcup_{x\in M} D^\circ_x$ of the dual vector bundle $E^{*} \rightarrow M$, where $D^\circ_x\subset E_x^*$ is the annihilator of $D_x$ at $x\in M$, and $\bigsqcup$ stands for the disjoint union of subsets. If $D$ is smooth and not regular, then $D^\circ$ is not smooth.  Using the usual identification $E^{\ast\ast} = E$ for a vector bundle $E$ over $M$ of finite rank, it follows that $(D^\circ)^\circ = D$.

A generalised subbundle of $\T M\to M$ is hereafter called a {\it distribution}, while a generalised subbundle of $\cT M\to M$ is called a {\it codistribution}. It is worth noting that our terminology may differ from the one in the  literature, where distributions are generally assumed to be regular, while distributions that are not regular are called {\it generalised distributions} or {\it Stefan--Sussmann distributions} \cite{Lav_18}. Nevertheless, we have simplified our terminology and skipped the term `generalised' whenever this does not lead to any misunderstanding. Distributions are hereafter denoted by calligraphic capital letters.

Given two distributions $\mathcal{D},\mathcal{D}'$ on $M$, we write $[\mathcal{D},\mathcal{D}']$ for the distribution on $M$ spanned by the Lie brackets of vector fields of $\Gamma(\mathcal{D})$  with vector fields of $\Gamma(\mathcal{D}')$. Given a vector field $X$ on $M$, we write $[X,\mathcal{D}]$ for the distribution generated by the Lie brackets of $X$ with elements of $\Gamma(\mathcal{D})$.

\begin{definition}
    Given a regular distribution $\mathcal{D}$ on $M$ of corank $k$, an {\it associated one-form} is a locally defined differential one-form on $M$ taking values in $\mathbb{R}^k$, i.e. an element $\bm\zeta\in \Omega^1(U,\mathbb{R}^k)$ for an open subset $U\subset M$, such that $\restr{\mathcal{D}}{U} = \ker\bm\zeta$. We say that all associated one-forms related to $\mathcal{D}$ on the same $U$ are {\it compatible}.
\qeddiamond\end{definition}

Every differential one-form $\zeta\in\Omega^1(M)$ spans a codistribution $C\subset \cT M$ such that $C_x= \langle\zeta_x\rangle$ for every $ x\in M$. Moreover, $C$ has rank one at every point $x\in M$ where $\zeta$ does not vanish. The annihilator of $C$ is the distribution $\ker \zeta\subset\T M$, namely $\ker\zeta = C^\circ$. The distribution $\ker\zeta$ has corank one at every point where $\zeta$ does not vanish and corank zero otherwise.

\subsection{Contact manifolds}

Let us provide a brief introduction to contact geometry (see \cite{BH_16,Gei_08,Kho_13} for details). Let us recall that a distribution $\mathcal{D}$ of corank one on a $(2n+1)$-dimensional manifold $M$, with $n\in \mathbb{N}$, can locally be, on an open neighbourhood $U$ of each point $x\in M$, described as the kernel of a differential one-form $\eta\in\Omega^1(U)$. Then, $\mathcal{D}$ is  {\it maximally non-integrable} if every associated form $\eta\in \Omega^1(U)$ is such that  $\eta\wedge(\d\eta)^n$ is a volume form on $U$\footnote{See \cite{Ada_21,CCF_24,Vit_15} for other, not always equivalent, definitions of maximally non-integrable distributions.}. In such a case, $\eta$ is called a \textit{contact form on $U$}. It is worth noting that the term {\it completely non-integrable} is used instead of {\it maximally non-integrable} in part of the literature \cite{BE_15}. 

A \textit{contact manifold} is a pair $(M,\mathcal{D})$ such that $\dim M=2n+1$ and $\mathcal{D}$ is a corank one maximally non-integrable distribution on $M$. We call $\mathcal{D}$ a \textit{contact distribution} on $M$. 

A {\it co-oriented contact manifold} is a pair $(M,\eta)$, where $\eta$ is a differential one-form on $M$ such that $(M,\ker\eta)$ is a contact manifold. Then, $\eta$ is called a {\it contact form}. To simplify the terminology, co-oriented contact manifolds are generally called contact manifolds in most modern literature on contact geometry and contact Hamiltonian systems \cite{BCT_17,LL_19,GGMRR_20a}, but this simplification will be avoided in our work, which frequently deals with general contact manifolds. If not otherwise stated, $(M,\eta)$ hereafter stands for a co-oriented contact manifold.

It is worth stressing that if $\eta$ is a contact form on $M$ associated with a contact distribution $\mathcal{D}$, then $f\eta$ is also a contact form on $M$ associated with $\mathcal{D}$ for every non-vanishing function $f\in\Cinfty(M)$. Moreover, a differential one-form $\eta$ on $M$ is such that $\eta\wedge(\d\eta)^n$ is a volume form on $M$ if, and only if, $\ker \eta$ is not zero and $\eta$ induces a decomposition $\T M = \ker\eta\oplus\ker\d\eta$.

A co-oriented contact manifold $(M,\eta)$ determines a unique vector field $R\in\X(M)$, called the {\it Reeb vector field}, such that $\iota_R\d\eta = 0$ and $\iota_R\eta = 1$. Then, $\Lie_R\eta = 0$ and, therefore, $\Lie_R\d\eta = 0$. Reeb vector fields on compact manifolds have garnered significant attention regarding the admission of closed integral curves and the so-called {\it Weinstein conjecture} \cite{ACH_19,CH_13,Tau_07,Wei_79}.

\begin{theorem}[Contact Darboux theorem \cite{AM_78,LM_87}]
    Given a co-oriented contact manifold $(M,\eta)$ such that $\dim M=2n+1$, there exist local coordinates $\{q^i, p_i, s\}$  with $i = 1,\dotsc,n$ around every point $x\in M$, called {\it Darboux coordinates},  such that
    $$ \eta = \d s - \sum_{i=1}^np_i\d q^i\,. $$
    In these coordinates, the Reeb vector field reads $R = \tparder{}{s}$.
\end{theorem}

Since every contact distribution is, locally, the kernel of a contact form, the above theorem also holds on the domain of definition of each associated contact form.

\begin{example}[Canonical co-oriented contact manifold]\label{ex:2.2}
    Consider the product manifold $M = \cT Q\times\R$, where $Q$ is any manifold. The cotangent bundle $\cT Q$ admits an adapted coordinate system $\{q^i,p_i\}$ and $\R$ has a natural coordinate $s$. These coordinates can be pulled-back to $M$ to give rise to a coordinate system $\{q^i, p_i, s\}$ on $\cT Q\times \R$. Then, $\eta_Q = \d s - \theta$, where $\theta$ is the pull-back of the Liouville one-form $\theta_\circ\in\Omega^1(\cT Q)$ relative to the canonical projection $\cT Q\times \R\to\cT Q$, is a contact form on $M$. In the local coordinates $\{q^i,p_i,s\}$ on $M$, one has that
    $$
        \eta_Q = \d s - \sum_{i=1}^n p_i\d q^i\,,\qquad R=\parder{}{s}\,.
    $$
    The coordinates $\{q^i,p_i,s\}$ are Darboux coordinates on $M$. Note that $\theta_\circ$, and thus $\eta_Q$, are independent of the chosen adapted coordinates $\{q^i,p_i\}$.
\end{example}

Example \ref{ex:2.2} is a particular case of {\it contactification} of an exact symplectic manifold. Consider an {\it exact symplectic manifold} $(N,\d\theta_N)$, namely a symplectic manifold $(N,\omega)$ whose symplectic form, $\omega$, is exact, i.e. $\omega = \d\theta_N$ for a differential one-form $\theta_N\in \Omega^1(N)$. Then, the product manifold $M = N\times\R$ is a contact manifold with the contact form $\eta = \d s + \theta$, where the variable $s$ stands for the canonical coordinate in $\R$ understood as a variable in $M$ in the natural manner and $\theta$ is the pull-back to $M$ of $\theta_N$ on $N$ via the projection $M=N\times\R\to N$.

Let $(M,\eta)$ be a co-oriented contact manifold. There exists a vector bundle isomorphism $\flat\colon\T M\to\cT M$ given by
$$
    \flat(v) = \inn{v}(\d\eta)_x + (\inn{v}\eta_x)\eta_x\,,\qquad \forall v\in \T_x M,\quad\forall x\in M\,. 
$$
This isomorphism can be extended to a $\Cinfty(M)$-module isomorphism $\flat:\X(M)\to\Omega^1(M)$ in the natural way. It is usual to denote both isomorphisms, of vector bundles and of $\Cinfty(M)$-modules, by $\flat$ as this does not lead to any misunderstanding. Meanwhile, the inverse of $\flat$ is denoted by $\sharp = \flat^{-1}$. Taking into account this last definition, $R = \sharp\eta$. It can be checked that a differential one-form $\eta$ on $M$ is a contact form if, and only if, its associated $\flat$ morphism is a vector bundle isomorphism.

A {\it contact Hamiltonian system} \cite{Bra_17,LL_19,GGMRR_20a,GG_22} is a triple $(M,\eta,h)$, where $(M,\eta)$ is a co-oriented contact manifold and $h\in\Cinfty(M)$. Given a contact Hamiltonian system $(M,\eta,h)$, there exists a unique vector field $X_h\in\X(M)$, called the {\it contact Hamiltonian vector field} of $h$, satisfying any of the following equivalent conditions:
\begin{enumerate}[(1)]
    \item $\iota_{X_h}\d\eta = \d h - (R h)\eta$ and $\iota_{X_h}\eta = -h$,
    \item $\ \Lie_{X_h}\eta = -(R h)\eta$ and $ \iota_{X_h}\eta = -h$,
    \item $\ \flat(X_h) = \d h - (R h + h)\eta$.
\end{enumerate}
A vector field $X\in\X(M)$ is said to be \textit{Hamiltonian} relative to a contact form $\eta$ if it is the contact Hamiltonian vector field of a function $h\in\Cinfty(M)$. It can be proved that $h$ is unique and it can be called the \textit{Hamiltonian function} of $X$. Moreover, every $h$ determines a unique contact Hamiltonian vector field $X_h$. 
Unlike the case of symplectic mechanics,  a Hamiltonian function $h$ may not be preserved along the integral curves of its contact Hamiltonian vector field $X_h$ (see \cite{AMR_88,GGMRR_20a} for details). More precisely,
$$
    X_h h = -(R h)h\,.
$$
A function $f\in\Cinfty(M)$ such that $X_h f = -(Rh)f$ is called a \textit{dissipated quantity} \cite{GGMRR_20a,GLR_23}. In Darboux coordinates, the contact Hamiltonian vector field $X_h$ reads
\begin{equation}\label{Eq:HamCoor}
    X_h = \sum_{i=1}^{n}\parder{h}{p_i}\parder{}{q^i} - \sum_{i=1}^{n}\left( \parder{h}{q^i} +  p_i\parder{h}{s} \right)\parder{}{p_i} + \left( \sum_{i=1}^{n}p_i\parder{h}{p_i} - h \right)\parder{}{s}\,.
\end{equation}
Its integral curves, let us say $\gamma(t) = (q^i(t), p_i(t), s(t))$, satisfy the system of ordinary differential equations
$$
    \frac{\d q^i}{\d t} = \parder{h}{p_i}\,,\qquad
    \frac{\d p_i}{\d t} = - \left( \parder{h}{q^i} + p_i\parder{h}{s} \right)\,,\qquad
    \frac{\d s}{\d t} = \sum_{j=1}^{n}p_j\parder{h}{p_j} - h\,,\qquad i = 1,\dotsc,n\,.
$$


\subsection{\texorpdfstring{$k$}--Vector fields and their integral sections}

$k$-Vector fields are of great use in the geometric study of systems of partial differential equations \cite{LSV_15,RRSV_11}. Consider the Whitney sum \footnote{From now on, the subindex $M$ in the Whitney sum will be skipped.}
\begin{equation}\label{eq:Whitney-sum}
    \boplus^k\T M := \T M\oplus_M\overset{(k)}{\dotsb}\oplus_M \T M\,,
\end{equation}
with the natural projections
$$
    \tau^\alpha\colon\boplus^k\T M\to\T M\,,\qquad \tau^k_M\colon\boplus^k\T M\to M\,, \qquad \alpha=1,\ldots,k\,,
$$
where $\tau^\alpha$ denotes the projection onto the $\alpha$-th component of the Whitney sum \eqref{eq:Whitney-sum} and $\tau^k_M$ is the projection onto the base manifold $M$.

\begin{definition}
    A {\it $k$-vector field} on $M$ is a section $\bfX\colon M\to\boplus^k\T M$ of the projection $\tau^k_M$. The space of $k$-vector fields on $M$ is denoted by $\X^k(M)$. 
\qeddiamond\end{definition}

Taking into account the diagram below,
\vspace{-3mm}
\[
\begin{tikzpicture}[baseline=(M.base), >=stealth]
    \node (M) at (0,0) {$M$};
    \node (E) at (3,2) {$\boplus^k \T M$};
    \node (TM) at (3,0) {$\T M$};

    \draw[->] (M) -- node[below] {$X_\alpha$} (TM);
    \draw[->] (M) -- node[above left] {$\bfX$} (E);
    \draw[->] (E) -- node[right] {$\tau^\alpha$} (TM);
\end{tikzpicture}
\]

a $k$-vector field $\bfX\in\X^k(M)$ amounts to a family vector fields $X_1,\dotsc,X_k\in\X(M)$ given by $X_\alpha = \tau^\alpha\circ\bfX$ with $\alpha=1,\ldots,k$. With this in mind, one can denote $\bfX = (X_1,\dotsc, X_k)$. A $k$-vector field $\bfX$ induces a decomposable contravariant totally skew-symmetric tensor field, $X_1\wedge\dotsb\wedge X_k$, which is a decomposable section of the bundle $\bigsqcup_{x\in M}\bigwedge^k\T_{x} M=\bigwedge^k\T M\to M$.

\begin{definition}\label{dfn:first-prolongation-k-tangent-bundle}
    Given a mapping $\phi\colon U\subset\R^k\to M$, its {\it first prolongation} to $\bigoplus^k\T M$ is the map $\phi'\colon U\subset\R^k\to\bigoplus^k\T M$ 
    defined as follows
    $$
        \phi'(t) = \left( \phi(t); \T_t\phi\left( \parder{}{t^1}\bigg\vert_t \right),\dotsc,\T_t\phi\left( \parder{}{t^k}\bigg\vert_t \right) \right) \equiv (\phi(t); \phi'_{1}(t),\ldots,\phi'_{k}(t))\,, \qquad t\in\R^k\,,
    $$
    where $t=(t^1,\dotsc,t^k)$ and $\{t^1,\ldots,t^k\}$ are the canonical coordinates on $\R^k$.
\qeddiamond\end{definition}

As for integral curves of vector fields, one can define integral sections of a $k$-vector field in the following manner.

\begin{definition}
    Let $\bfX = (X_1,\dotsc,X_k)\in\X^k(M)$ be a $k$-vector field. An {\it integral section} of $\bfX$ is a map $\phi\colon U\subset\R^k\to M$ such that $\phi' = \bfX\circ\phi\,$, namely $\T\phi \left(\parder{}{t^\alpha}\right) = X_\alpha\circ\phi$ for $\alpha=1,\ldots,k$. A $k$-vector field $\bfX\in\X^k(M)$ is said to be {\it integrable} if every point of $M$ lies in the image of an integral section of $\bfX$.
\qeddiamond\end{definition}

Let $\bfX = (X_1,\ldots, X_k)$ be a $k$-vector field on $M$ with local expression $ X_\alpha = \sum_{i=1}^{m}X_\alpha^i\parder{}{x^i}\,$ for $\alpha=1,\ldots ,k$.
Then, $\phi\colon U\subset\R^k\to M$, with local expression $\phi(t) = (\phi^i(t))$, is an integral section of $\bfX$ if, and only if, its components satisfy the following system of PDEs
\begin{equation}\label{eq:SysPDEs}
	\parder{\phi^i}{t^\alpha} = X_\alpha^i\circ\phi\,,\qquad i=1,\ldots,m\,,\qquad \alpha=1,\ldots,k\,.
\end{equation}

Then, $\bfX$ is integrable if, and only if, $[X_\alpha,X_\beta] = 0$ for $\alpha,\beta=1,\ldots,k$. These are necessary and sufficient conditions for the integrability of the system of PDEs \eqref{eq:SysPDEs} (see \cite{Lee_12,Olv_93} for details).

Every $k$-vector field $\bfX = (X_1,\dotsc,X_k)$ on $M$ defines a distribution $\mathcal{D}^\bfX\subset\T M$ given by $\mathcal{D}^\bfX_x = \mathcal{D}^\bfX \cap\T_xM = \langle X_1(x),\dotsc,X_k(x)\rangle$. However, the notion of an integral submanifold of the $k$-vector field $\bfX$ is stronger than the notion of an integral section of the distribution $\mathcal{D}^\bfX$. The distribution $\mathcal{D}^\bfX$ is integrable if, and only if, $[X_\alpha,X_\beta] = \sum_{\gamma=1}^{k}f_{\alpha\beta}^\gamma X_\gamma$, with $\alpha,\beta=1,\ldots,k$, for certain functions $f_{\alpha\beta}^\gamma\in \Cinfty(M)$ with $\alpha,\beta,\gamma=1,\ldots,k$, and $\mathcal{D}^{\bfX}$ is invariant relative to the one-parameter group of diffeomorphisms of any vector field taking values in $\mathcal{D}^{\bfX}$ \cite{Lav_18,Ste_74,Sus_73}. On the other hand, the $k$-vector field $\bm X$ is integrable if, and only if, $X_1,\ldots, X_k$ commute with each other.

\subsection{\texorpdfstring{$k$}{}-Symplectic manifolds}

This section briefly reviews the theory of $k$-symplectic manifolds and introduces the notions to be employed hereafter (see \cite{Awa_92,LSV_15,LV_13,RRSV_11} for further details).

\begin{definition}
    A {\it $k$-symplectic form} on a manifold $P$ is a closed non-degenerate $\R^k$-valued differential two-form $\bm\omega = \sum^k_{\alpha=1}\omega^\alpha\otimes e_\alpha$ on $P$. The pair $(P,\boldsymbol{\omega})$ is called a {\it $k$-symplectic manifold}. If the two-form $\boldsymbol{\omega}$ is exact, namely $\boldsymbol{\omega} = \d\boldsymbol{\theta}$ for some $\boldsymbol{\theta} \in\Omega^1(P,\R^k)$, then $\bm\omega$ is said to be an {\it exact $k$-symplectic form}. The pair $(P,\d\bm\theta)$ is called an \textit{exact $k$-symplectic manifold.} Additionally, let $P$ be $ n(k+1)$ dimensional for $n\in \mathbb{N}$. A \textit{polarisation} on $P$ is an integrable distribution $\mathcal{V}$ of rank $nk$ on $P$ such that
    $$ \restr{\bm\omega}{\mathcal{V}\times \mathcal{V}} = 0\,. $$
    We say that $(P,\bm\omega,\mathcal{V})$ is a \textit{polarised $k$-symplectic manifold}.
\qeddiamond\end{definition}

Note that $(P,\bm \omega)$ is a $k$-symplectic manifold if, and only if, $\bm \omega$ is closed and 
$$
    \ker\bm\omega = \ker\left(\sum^k_{\alpha=1}\omega^\alpha\otimes e_\alpha\right) = \bigcap^k_{\alpha=1}\ker \omega^\alpha=0\,.
$$

\begin{example}[Canonical $k$-symplectic manifold]\label{Ex:kCotan}
    Consider a vector bundle of the form $\pi^k_M\colon\boplus^k\cT M\rightarrow M$ and the projection $\pi^{\alpha}\colon\boplus^k\cT M\rightarrow\cT M$ onto the $\alpha$-th factor of $\boplus^k\cT M$. Then, one can consider the differential $\mathbb{R}^k$-valued two-form $\bm\omega_M = \sum_{\alpha=1}^k\omega_M^\alpha \otimes e_\alpha$, where $\omega^{\alpha}_{M}=\pi^{\alpha *}\omega_{M}$ for the canonical symplectic form $\omega_{M}$ on $\cT M$ and $\alpha=1,\ldots,k$, which is closed and non-degenerate.
    
\end{example}

\begin{theorem}[Darboux theorem for polarised $k$-symplectic manifolds \cite{Awa_92,GLRR_24}]
    Let $(P,\bm\omega,\mathcal{V})$ be a polarised $k$-symplectic manifold. Then, for every point of $x\in P$, there exists local coordinates around $x$, let us say, $(q^i,p_i^\alpha)$, such that
\begin{equation}\label{Eq:DarkSym} \bm\omega = \sum_{\alpha=1}^k\omega^\alpha\otimes e_\alpha = \sum_{\alpha=1}^k\sum_{i=1}^n(\d q^i\wedge \d p_i^\alpha)\otimes e_\alpha\ ,\qquad \mathcal{V} = \left\langle \parder{}{p_i^\alpha} \right\rangle
    \end{equation}
\end{theorem}
\begin{definition} Given a polarised $k$-symplectic manifold $(P,\bm\omega,\mathcal{V})$, the coordinates bringing $\bm \omega$ into the form \eqref{Eq:DarkSym} are called \textit{Darboux} or \textit{canonical coordinates} of the $k$-symplectic manifold.
\qeddiamond\end{definition}

\begin{definition} Given a $k$-symplectic manifold $(P,\bm\omega)$, a $k$-vector field $\bfX = (X_1,\ldots,X_{k})$ is {\it Hamiltonian} if there exists $h\in\Cinfty(P)$ such that
$$
    \sum_{\alpha=1}^k\inn{X_\alpha}\omega^\alpha = \d h\,.
$$
The function $h\in \Cinfty(P)$ is called the {\it Hamiltonian function} of ${\bm X}$. 
\qeddiamond\end{definition}

\begin{definition}
Let $V\subset\T_xP$ be a linear subspace. We define the \textit{$k$-symplectic orthogonal of $V$} as
$$
    V^{\perp_{\bm\omega}} = \{ w\in\T_xP\mid \bm\omega(w,v) = 0\,,\ \forall v\in V \}\,.
$$
Using this notion of orthogonality, we say that
\begin{itemize}
    \item $V$ is \textit{isotropic} if $V\subset V^{\perp_{\bm\omega}}$,
    \item $V$ is \textit{Lagrangian} if it is isotropic and there exists an isotropic subspace $W\subset\T_xP$ such that $\T_xP = V\oplus W$.
\end{itemize}
\qeddiamond\end{definition}
Unlike Lagrangian subspaces in symplectic geometry, there is no fixed dimension for Lagrangian subspaces in the $k$-symplectic setting \cite{Can_01,LSV_15}. Note that previous point-wise notions of orthogonality, isotropy and Lagrangianity can be extended to tangent bundles to submanifolds by considering the tangent space at every point of a given submanifold to be isotropic or Lagrangian, respectively.

\subsection{A review on \texorpdfstring{$k$}--contact geometry}

This section reviews $k$-contact geometry through the here newly defined $k$-contact forms, offering several examples. 
We also locally describe first-order jet bundles in relation to some adapted $k$-contact forms. While such $k$-contact forms are not globally defined, their kernels are. This will be of  relevance in further parts of this work. Moreover, we highlight several facts about $k$-contact geometry that significantly differ from the contact ($k=1$) case. For the sake of clearness, we recall that $k$ is assumed to be a natural number, namely $k\geq 1$. For completeness, the relation of $k$-contact distributions and $k$-forms with other types of generalisations of contact geometry are analysed. 

\begin{definition}\label{dfn:k-contact-manifold}
    A \textit{$k$-contact form on an open subset $U\subset M$} is a differential one-form on $U$ taking values in $\mathbb{R}^k$, let us say $\bm\eta\in\Omega^1(U,\mathbb{R}^k)$, such that
    \begin{enumerate}[(1)]
        \item $\ker \bm\eta\subset\T U$ is a regular non-zero distribution of corank $k$,
        \item $\ker \d\bm\eta\subset\T U$ is a regular distribution of rank $k$,
        \item $\ker \bm\eta\cap\ker \d\bm\eta  = 0$.
    \end{enumerate}
    Whenever the $k$-contact form $\bm \eta$ is globally defined on $M$, the pair $(M,\bm\eta)$ is called a \textit{co-oriented $k$-contact manifold} and $\ker\d\bm\eta$ is called the {\it Reeb distribution} of $(M,\bm\eta)$.  If, in addition, $\dim M = n+nk+k$ for some $n,k\in\N$ and there exists an integrable distribution $\mathcal{V}$ contained in $\ker \bm\eta$ with $\rk\mathcal{V} = nk$, we say that $(M,\bm\eta,\mathcal{V})$ is a \textit{polarised co-oriented $k$-contact manifold}. We call the distribution $\mathcal{V}$ a \textit{polarisation} of $(M,\bm\eta)$.
\qeddiamond\end{definition}

In the previous literature \cite{GGMRR_20,GGMRR_21},
the family of components of a $k$-contact form was
called a \textit{$k$-contact structure}.
Several generalisations of contact geometry have appeared in the literature and their relation with our $k$-contact geometry is analysed next.

Let us start for the so-called $q$-contact structures originated in  Almeida's doctoral thesis
\cite[Definition 3.2.1]{Almeida2018} and studied, for instance, in
\cite{Almeida2024,Finamore2024,
FinamoreQuasiconformal2024} or described explicitly in \cite[Definition~2.1]{LeokSardonZhao2026}. A particular subcase, the 
uniform $q$-contact structures, provide Hamiltonian and
Lagrangian models for systems with several dissipation channels \cite{LeokSardonZhao2026}, while Hamiltonian models on uniform and non-uniform $q$-contact manifolds
are developed in \cite{DeLeonSardonZhao2026}.
More exactly, a $q$-contact structure in the Almeida's sense is a family of differential one-forms $(\lambda^1,\ldots,\lambda^q)$ satisfying that
\[
\xi:=\bigcap_{i=1}^{q}\ker\lambda^i\,,\qquad
\ker\d\lambda^i=:\mathcal{R}\,,\qquad
i=1,\ldots,q\,,\qquad
\T M=\xi\oplus\mathcal{R}\,,
\]
while $\xi$ has corank $q$ and $\dim M=2n+q$. In particular, each $\d\lambda^i|_\xi$ is non-degenerate. 
These structures therefore give particular cases of
co-oriented $k$-contact manifolds with $k=q$ as in our Definition \ref{dfn:k-contact-manifold}. This has been made
explicit in the following proposition, which gives a proof of a fact already noted in 
\cite{Finamore2024}. 

\begin{proposition}\label{prop:q-contact-k-contact}
Let $(\lambda^1,\ldots,\lambda^q)$ be a $q$-structure on a manifold $M$
of dimension $2n+q$, with $n,q\geq1$, in the Almeida's sense.
Then,
\(
\bm\lambda
=\sum_{i=1}^{q}\lambda^i\otimes e_i
\in\Omega^1(M,\mathbb{R}^{q}),
\)
where $\{e_1,\ldots,e_q\}$ is the canonical basis of
$\mathbb{R}^{q}$, is a $q$-contact form as in  Definition~\ref{dfn:k-contact-manifold}.
In particular, $(M,\bm\lambda)$ is a co-oriented
$k$-contact manifold with $k=q$. Its  Reeb
distribution, in our $k$-contact sense, is $\mathcal{R}$.
\end{proposition}

\begin{proof}
The pointwise linear independence of
$\lambda^1,\ldots,\lambda^q$ implies that
\(
\ker\bm\lambda
=\bigcap_{i=1}^{q}\ker\lambda^i
=\xi
\)
is a regular distribution of corank $q$.
It is non-zero since $\dim\xi=2n$ and $n\geq1$.
Moreover, the defining condition
$\ker\d\lambda^i=\mathcal{R}$ for $i=1,\ldots,q$ gives
\(
\ker\d\bm\lambda
=\bigcap_{i=1}^{q}\ker\d\lambda^i
=\mathcal{R}.
\)
Since $\T M=\mathcal{R}\oplus\xi$, the distribution
$\mathcal{R}$ has rank $q$ and
\[
\ker\bm\lambda\cap\ker\d\bm\lambda
=\xi\cap\mathcal{R}=\{0\}.
\]
Hence, $\bm\lambda$ satisfies all the defining
conditions of a global $q$-contact form. In contrast, our definition only requires the restrictions $\d\eta^\alpha|_{\mathcal D}$ to have trivial common kernel; the individual restrictions may be degenerate. Note that $\ker \d\bm\lambda$ is the Reeb distribution in the $k$-contact sense given in Definition \ref{dfn:k-contact-manifold}.
\end{proof}

The {\it uniformity} of $q$-contact structures, namely $\d\lambda^i=\d\lambda^j$ for all
$i,j$, gives a particular case of $q$-contact structures. Almeida studied $q$-contact structures in connection with
contact Anosov actions and their algebraic models
\cite{Almeida2018,Almeida2024}. Finamore investigated
the associated Reeb foliations and generalised Weinstein
conjectures, obtaining results under additional Riemannian
or quasiconformal hypotheses
\cite{Finamore2024,FinamoreQuasiconformal2024}.

{\it Metric $f$-$K$-contact geometry} equips several
one-forms with a compatible metric and a
$(1,1)$-tensor $f$, requiring their exterior
derivatives to coincide with a common fundamental
two-form and their characteristic vector fields
to be Killing \cite{GoertschesLoiudice2020}.
Their defining one-forms determine a uniform
$q$-contact structure, so
Proposition~\ref{prop:q-contact-k-contact} also
applies after forgetting the additional metric
and tensor data.

\begin{remark} The condition $\ker\bm \eta\neq 0$ is assumed so as to avoid  non-interesting examples and retrieve standard contact geometry in the case $k=1$. In fact, if we consider the definition of a $k$-contact form and skip the condition $\ker \bm\eta\neq 0$, then one can consider $\eta$ to be any non-vanishing one-form on  $M=\mathbb{R}$. Thus, $\d\eta=0$ and $\ker \d\eta=\T \mathbb{R}$.  Hence, all conditions (1), (2), and (3) are satisfied. Nevertheless, we want a one-contact form to be a contact form and a non-vanishing $\eta$ on $\mathbb{R}$ is not a contact form as its kernel is not a contact distribution. The condition $\ker \bm \eta\neq 0$ has been added to the initial definition in \cite{GGMRR_20} to solve the previous problem of not retrieving contact manifolds for $k=1$ on manifolds of dimension one.
\end{remark}

\begin{remark}\label{Rem:CompositePolarised}
    Note that if $(M, \bm\eta)$ admits a polarisation $\mathcal{V}$, then $\dim M=n+ nk + k$ for some $n,k\in\N$. One may wonder when $\dim\,M$  can be written in this form. Observe that  
    $$
        n + nk + k = (n+1)(k+1) - 1\,,
    $$
    so $\dim M + 1$ must be a composite number. In other words, if $(M,\bm\eta,\mathcal{V})$ is a polarised $k$-contact manifold, $\dim M + 1$ cannot be a prime number.
\end{remark}

If $\rk \ker\d\bm\eta = k$, with $\ker \bm \eta\neq 0$, and $\cork\ker\bm\eta=k$, i.e. conditions (1) and (2) in Definition \ref{dfn:k-contact-manifold} hold, then condition (3) amounts to each one of the following equivalent equalities on $U$:
\begin{enumerate}
    \item[(3$''$)] $\T U = \ker \bm\eta\oplus\ker \d\bm\eta$,
    \item[(3$'''$)] $\cT U = (\ker \bm \eta)^\circ\oplus(\ker \d\bm \eta)^\circ$.
\end{enumerate}

\begin{remark}
    A one-contact form is a one-form $\eta$ satisfying Definition \ref{dfn:k-contact-manifold}. This implies that: (2) means that $\ker\d\eta$ has rank one and (1) implies that $\ker\eta$ has co-rank one and $\ker\eta\neq 0$. Since $(\ker\d \eta)^\circ$ has even rank, and taking (3$'''$), we obtain that $\dim M$ is odd. Since $\ker\eta$ is different from zero and it has corank one, then the manifold has, at least, dimension three. This retrieves that $\eta\wedge(\d \eta)^n$, for some $n\in \mathbb{N}$, is a volume form on $U$, which amounts to the notion of a {\it co-oriented contact manifold} when $U=M$ (see \cite{BCT_17,LL_19,GGMRR_20a}).
\end{remark}

\begin{theorem}[Reeb vector fields \cite{GGMRR_20}]\label{thm:k-contact-Reeb}
    Let $\left(M,\bm\eta = \sum_{\alpha=1}^k\eta^\alpha\otimes e_\alpha\right)$ be a co-oriented $k$-contact manifold. There exists a unique family of vector fields $R_1,\dotsc, R_k\in\X(M)$, such that
    \begin{equation*}\label{eq:k-contact-Reeb}
    \inn{R_\alpha}\eta^\beta = \delta_\alpha^\beta, \qquad \inn{R_\alpha}\d\eta^\beta = 0\,,    
    \end{equation*}
    for $\alpha,\beta = 1,\dotsc,k$. The vector fields $R_1,\dotsc,R_k$ commute between themselves, i.e.
    $$
        [R_\alpha,R_\beta] = 0\,,\qquad\alpha,\beta = 1,\dotsc,k\,.
    $$
    In addition, $\ker \d\bm\eta$ is spanned by those vector fields, namely $\ker \d\bm\eta = \langle R_1,\dotsc,R_k \rangle$. 
\end{theorem}

Note that $\ker \d\bm\eta$ is integrable because it is the intersection of the kernels of closed forms and has constant rank by assumption.

\begin{definition} Given a $k$-contact manifold $(M,\bm\eta)$, we call \textit{Reeb $k$-vector field} of $(M,\bm\eta)$ the integrable $k$-vector field ${\bf R}=(R_1,\ldots,R_k)$ on $M$ described in Theorem \ref{thm:k-contact-Reeb}. 
The vector fields $R_1,\dotsc,R_k\in\X(M)$ are the {\it Reeb vector fields} of the $k$-contact manifold $(M,\bm\eta)$.
\qeddiamond\end{definition}

\begin{corollary}
    The Reeb vector fields for a $k$-contact form $\bm \eta$ on $M$ are the unique vector fields $R_1,\ldots,R_k$ on $M$ such that
    $$
        \inn{R_\alpha}\eta^\beta = \delta_\alpha^\beta\,,\qquad \Lie_{R_\alpha}\eta^\beta=0\,,\qquad \alpha,\beta=1,\ldots,k\,.
    $$
    Moreover, $\ker \bm\eta$ is invariant under the one-parameter groups of diffeomorphisms of the Reeb vector fields.
\end{corollary}

It is worth stressing that if $\bm \eta$ is a $k$-contact form on $U\subset M$, then it admits also Reeb vector fields on $U$.

\begin{example}\label{ex:canonical-k-contact-structure}
    The manifold $M = (\boplus^k\cT Q)\times\R^k$ has a natural $k$-contact form
    $$ 
        \bm\eta_{Q}= \sum_{\alpha=1}^k(\d z^\alpha - \theta^\alpha)\otimes e_\alpha\,,
    $$
    where $\{z^1,\dotsc,z^k\}$ are the pull-back to $M$ of standard linear coordinates in $\R^k$ and each $\theta^\alpha$ is the pull-back of the Liouville one-form $\theta$ of the cotangent manifold $\cT Q$ with respect to the projection $M\to\cT Q$ onto the $\alpha$-th component of $\boplus^k\cT Q$. Note also that $M$ admits a natural projection onto $Q\times \mathbb{R}^k$ and a related vertical distribution $\mathcal{V}$ of rank $k\cdot\dim Q$ contained in $\ker \bm \eta_{Q}$. Hence, $\left((\boplus^k\cT Q)\times\R^k,\bm\eta_Q,\mathcal{V}\right)$ is a polarised co-oriented $k$-contact manifold.

    Local coordinates $\{q^1,\ldots,q^n\}$ on $Q$ induce natural coordinates $\{q^i,p_i^\alpha\}$, for a fixed value of $\alpha$, on the $\alpha$-th component of $\boplus^k\cT Q$ and $\{q^i,p_i^\alpha,z^\alpha\}$, with $\alpha=1,\ldots,k$, on $M$. Then,
    $$
        \bm\eta_{Q} = \sum_{\alpha=1}^k\left(\d z^\alpha - \sum_{i=1}^np_i^\alpha\d q^i\right)\otimes e_\alpha\,,
    $$
    and
    $$
        \ker\bm\eta_{Q} = \left\langle \parder{}{p_i^\alpha}\,,\ \parder{}{q^i} + \sum_{\alpha=1}^kp_i^\alpha\parder{}{z^\alpha} \right\rangle\,.
    $$
    Hence, $\d\bm\eta_{Q} = \sum_{i=1}^n\sum_{\alpha=1}^k(\d q^i\wedge\d p_i^\alpha)\otimes e_\alpha$, the Reeb vector fields are $ R_\alpha = \tparder{}{z^\alpha}$ for $\alpha=1,\ldots,k$, and
    $$
        \ker \d\bm\eta_{Q} = \left\langle\parder{}{z^1},\dotsc,\parder{}{z^k}\right\rangle\,.
    $$ 
    \finish
\end{example}

\begin{example}[Contactification of an exact $k$-symplectic manifold]\label{ex:contactification-k-symplectic-manifold}
    Let $(P,\bm \omega = \d\bm\theta )$ be an exact $k$-symplectic manifold and consider the product manifold $M = P\times\R^k$. Let $\{z^1,\ldots,z^k\}$ be the pull-back to $M$ of some Cartesian coordinates in $\R^k$ and denote  by $\theta_M^\alpha$ the pull-back of $\theta^\alpha$ to the product manifold $M$. Consider the $\mathbb{R}^k$-valued one-form ${\bm \eta} = \sum_{\alpha=1}^k(\d z^\alpha + \theta_M^\alpha)\otimes e_\alpha \in\Omega^1(M,\mathbb{R}^k)$.  Then, $(M,\bm\eta)$ is a co-oriented $k$-contact manifold, because $\ker \bm \eta\neq 0$ has corank $k$, while $\d{\bm \eta} = \d{\bm \theta}_M$ and $\ker \d{\bm \eta} = \langle\tparder{}{z^1},\dotsc,\tparder{}{z^k}\rangle$ has rank $k$ since $\bm\omega$ is non-degenerate. It follows that $\bm\eta$ is a globally defined $k$-contact form.
    
    Note that the so-called canonical $k$-contact form $\bm\eta_Q$ described in Example \ref{ex:canonical-k-contact-structure} is just the contactification of the $k$-symplectic manifold $(P = \bigoplus^k\cT Q,\bm\omega_Q)$ described in Example \ref{Ex:kCotan}.
    \finish
\end{example}

\begin{example}\label{Ex:1JetkCon}
    Consider the first-order jet  bundle $J^1=J^1(M,E)$ of a fibre bundle $E\rightarrow M$ of rank $k$ with adapted coordinates $\{x^i,y^\alpha,y_i^\alpha\}$. Its Cartan distribution
    $$
        \mathcal{C} = \left\langle \frac{\partial}{\partial x^i}+\sum_{\alpha=1}^ky^\alpha_i\frac{\partial}{\partial y^\alpha},\frac{\partial}{\partial y^\alpha_i}\right\rangle
    $$ 
    has rank $m(k+1)$ and it is globally defined \cite{Olv_93,Sau_89}. Note that $[\mathcal{C},\mathcal{C}] = \T J^1$.

    On the open subset $U$ of $J^1$ where our local adapted coordinates are defined, there exists a $k$-contact form
    $$
        \bm \eta=\sum_{\alpha=1}^k\left(\d y^\alpha-\sum_{i=1}^my^\alpha_i\d x^i\right)\otimes e_\alpha\,,
    $$
    such that $\ker \bm \eta=\mathcal{C}|_U$. Moreover,
    $$
       \left\langle \frac{\partial}{\partial y^\alpha}\right\rangle=\ker \d\bm\eta\,.
    $$
    This example is quite interesting due to the fact that, given another set of adapted coordinates $\{\bar x^i,\bar y^\alpha,\bar y^\alpha_i\}$ on $\bar U$ to $J^1$, the $\R^k$-valued differential one-form
    $$
        \bar{\bm\eta}=\sum_{\alpha=1}^k\left(\d \bar{y}^\alpha-\sum_{i=1}^m\bar {y}^\alpha_i\d \bar{x}^i\right)\otimes e_\alpha
    $$
    need not coincide with $\bm\eta$, but 
    $$
        \ker\bm\eta|_{U\cap \bar{U}}=\ker \bar{\bm\eta}|_{U\cap \bar{U}} = \mathcal{C}|_{U\cap \bar{U}}\,,\qquad \left\langle \frac{\partial}{\partial y^\alpha}\right\rangle = \ker \d\bm \eta,\qquad \ker \d\bar{\bm\eta}=\left\langle \frac{\partial}{\partial \bar y^\alpha}\right\rangle\,,
    $$ 
    and $\ker \bar{\bm\eta}\oplus \ker \d\bar{\bm\eta}=\T \bar{U}$ while $\ker \dd \bm\eta$ and $\ker \dd\bar{\bm\eta}$ differ in general. In fact, $\bar y^\alpha_i=d^\alpha_i(y_x,y,x)$, $\bar y^\alpha=f^\alpha(y,x)$ and $\bar x^i=g^i(x)$, where $y_x$ represents the variables induced in $J^1$ by some adapted coordinates in $E$ while $d^\alpha_i, f^\alpha, g^i$ are the functions giving the local change of variables. This means that, in general, one has
    $$
        \frac{\partial}{\partial \bar y^\alpha}=\sum_{\beta=1}^k\frac{\partial y^\beta}{\partial \bar y^\alpha }\frac{\partial}{\partial y^\beta}+\sum_{i=1}^m\sum_{\beta=1}^k\frac{\partial y_i^\beta}{\partial \bar y^\alpha }\frac{\partial}{\partial y^\beta_i}\,,\qquad \alpha=1,\ldots,k\,. 
    $$
    Consequently, the associated Reeb distributions $ \left\langle  \partial/\partial y^\alpha\right\rangle$ and $ \left\langle \partial/\partial \bar y^\alpha\right\rangle$ need not coincide.
 Hence, locally defined $k$-contact forms associated with adapted coordinates to $J^1$ and their differentials do not need to be globally defined, but all $k$-contact forms share the same kernel given by the Cartan distribution.

    To illustrate the above results, let us analyse a very simple but illustrative example given by a fibre bundle $(x,y)\in \mathbb{R}^2\mapsto x\in \mathbb{R}$. Let us study the first-order jet bundle $J^1(\mathbb{R},\mathbb{R}^2)$ with induced variables $\{x,y,\dot y\}$. The new bundle variables $(\bar x, \bar y) \in \mathbb{R}^2\mapsto 
  \bar x \in \mathbb{R}$ given by
    $$ \bar x = x\,,\qquad \bar y=(1+x^2)y\,, $$
    lead to a new adapted coordinate system on $J^1(\mathbb{R},\mathbb{R}^2)$ of the form
    $$ \bar x=x\,,\qquad\bar y = (1+x^2)y\,,\qquad\dot {\bar y}=(1+x^2)\dot y+2xy\,.$$
    Hence,
    $$
    \begin{gathered}
        \parder{}{x} = \parder{}{\bar x} + 2\frac{\bar x\bar y}{1+\bar{x}^2}\parder{}{\bar y} + \frac{2}{1+\bar x^2}\left(\bar{y} + \bar x\dot{\bar y}-\frac{2\bar x^2\bar y}{1+\bar x^2}\right)\parder{}{\dot{\bar y}}\,,\\\frac{\partial}{\partial y}=(1+\bar{x}^2)\frac{\partial}{\partial \bar y}+2\bar x\frac{\partial}{\partial \dot {\bar y}}\,,\qquad \frac{\partial}{\partial \dot y}=(1+\bar{x}^2)\frac{\partial}{\partial \dot {\bar y}}\,.
    \end{gathered}$$
    As stated before, $\langle \partial/\partial y\rangle\neq \langle \partial/\partial \bar y\rangle$ at a generic point, while $\bar \eta = (1+x^{2})\eta$. 
    \finish
\end{example}
For a contact form $\eta$ on $M$, one has that $[\ker \eta,\ker \eta]=\T M$. For $k$-contact forms, only the following result can be ensured. 
\begin{proposition}
\label{prop:ker_eta}
    Given a $k$-contact form $\bm\eta$ on $M$, one has that $\ker \d\bm\eta$ is an integrable distribution and $[\ker \bm\eta,\ker \bm\eta]_x\varsupsetneq\ker\bm\eta_x$ at every point $x\in M$.    
\end{proposition}
\begin{proof}
    Since $\d\bm\eta$ has closed components and $\ker\d\bm\eta$ has constant rank because $\bm \eta$ is a $k$-contact form, then $\ker\d\bm\eta$ is a regular distribution given by the intersection of the involutive distributions $\ker\d\eta^\alpha$ for $\alpha=1,\ldots,k$. To prove the second statement of this proposition, let us use reduction to absurd. Note that $[\ker \bm\eta,\ker\bm\eta]\supset \ker\bm \eta$ because for every vector field $X$ taking values in $\ker\bm\eta$ different from zero at $x\in M$, there exists a function $f$ such that $(Xf)(x)=1$ and $[X,fX]_x=X_x$. Assume that $[ \ker\bm\eta,\ker \bm\eta]_x=\ker \bm \eta_x$ at a point $x\in M$. Choose a basis $\{X_1,\ldots, X_s\}$ of $\ker \bm \eta$ around $x\in M$. Then,
    $$
        \d\bm \eta_{x}(X_i,X_j) = (X_i\inn{X_j}\bm\eta)_{x}-(X_j\inn{X_i}\bm\eta)_{x} - \inn{[X_i,X_j]_x}\bm \eta_{x}=-\inn{[X_i,X_j]_x}\bm\eta_x\,,\qquad i,j=1,\ldots,s\,.
    $$
   Since $[\ker\bm \eta,\ker\bm \eta]_x=\ker \bm \eta_x$ by assumption, one has $\d \bm \eta_x(X_i(x),X_j(x))=0$ for $i,j=1,\ldots,s$.  
    Since the Reeb vector fields take values in $\ker\d\bm\eta_{x}$, then $\d\bm\eta_{x}(X_i(x),R_\alpha(x))=0$ for $\alpha=1,\ldots,k$ and $i=1,\ldots,s$. Then, $X_1,\ldots,X_s$ take values in $\ker \d\bm\eta_{x}$ at $x$. Hence, $\ker\bm\eta_x\cap \ker \d\bm\eta_x\neq 0$,  which is a contradiction. Then, $[\ker \bm\eta,\ker \bm\eta]_x\varsupsetneq \ker\bm\eta_x$
\end{proof}

In some particular cases, one has $[\ker\bm\eta,\ker\bm\eta] = \T M$ for a $k$-contact form $\bm\eta$, but it is not needed in general as shown in the following example. Notwithstanding,  $[\ker\bm\eta,\ker\bm\eta] = \T M$ is satisfied for $k$-contact manifolds with  a polarisation, to be defined (in a new distributional manner) in Section \ref{sec:polarised-k-contact-distributions}.

\begin{example}
    Let us provide a $k$-contact form on  $M$ such that $[\ker \bm\eta,\ker \bm\eta]\neq \T M$. Consider $\mathbb{R}^2_\times\times\mathbb{R}^2$, where $\R_\times = (0,\infty)$, and a global coordinate system $\{x,p,z_1,z_2\}$ in $\mathbb{R}^2_\times\times \mathbb{R}^2$, namely $x,p\in \mathbb{R}_\times,z_1,z_2\in \mathbb{R}$. If 
    $$
        \bm\eta = (\d z_1 - p\,\d x)\otimes e_1+(\d z_2 - x\,\d p\,)\otimes e_2\,,
    $$
    it follows that
    $$
        \ker\bm\eta = \langle X_1 = x\partial_{z_2} + \partial_p\,,\ X_2 = p\partial_{z_1} + \partial_x \rangle\,,\qquad \d\bm\eta=\d x\wedge \d p \otimes e_1-\d x\wedge \d p\otimes e_2\,,
    $$
    $$
        \ker \d\bm\eta=\langle \partial_{z_{1}}, \partial_{z_{2}} \rangle.
    $$
    Then, $\ker \bm\eta\cap \ker \d\bm\eta=0$ with $\ker\bm \eta\neq 0$, $\corank \ker\bm \eta=2$, $\rank \ker\d\bm\eta=2$, and 
    $$
        [\ker\bm\eta,\ker \bm\eta]=\ker \bm\eta\oplus\langle \partial_{z_1}-\partial_{z_2}\rangle\varsubsetneq \T (\mathbb{R}_\times^2\times \mathbb{R}^2)\,.
    $$
    Note that $[X_1,X_2]$ does not take values in $\ker\bm\eta$ at any point of $M$, but $\bm \eta$ is a two-contact form.
    \finish
\end{example}

As in contact geometry, it is relevant to consider whether we can state a canonical form for the $k$-contact form.  The following theorem gives a Darboux normal form for a
fixed $k$-contact form on a polarised manifold. It also
corrects the simultaneous-adaptation assertion in
\cite[Theorem~3.10]{GGMRR_20} and
\cite[Theorem~7.1.11]{Riv_22}.
For $k>1$, we provide an independent proof that the
prescribed polarisation is unique and Reeb-invariant,
which yields simultaneous adaptation and solves a mistake in the proof of previous results.
For $k=1$, Example~\ref{Ex:CounterDarboux} gives a
counterexample to the simultaneous-adaptation
assertion in \cite[Theorem~3.10]{GGMRR_20} and
\cite[Theorem~7.1.11]{Riv_22}.

\begin{theorem}\label{Th:PolCokcontaDar} Consider a polarised co-oriented $k$-contact manifold $(M,\bm\eta,\mathcal{V})$. Then, there exists around every $x\in M$ a local coordinate system $\{x^i,y^\alpha,y^\alpha_i\}$ such that
$$
\bm \eta=\sum_{\alpha=1}^k\left(\d y^\alpha-\sum_{i=1}^ny^\alpha_i\d x^i\right)\otimes e_\alpha\,,\qquad R_\beta=\frac{\partial}{\partial y^\beta}\,,\qquad \beta=1,\ldots,k\,.
$$ For $k>1$, one has that coordinates can be chosen so that
$$
\mathcal{V}=\left \langle\frac{\partial}{\partial y^\alpha_i}\right\rangle.
$$
\end{theorem}
\begin{proof}
For $k=1$, the statement follows from the classical contact
Darboux theorem. Assume henceforth that $k>1$, and write
$\mathcal D=\ker\bm\eta$. We first show that $\mathcal V$ is preserved by the Reeb
flows. At every $x\in M$, one has  an isomorphism
\[
\mathcal V_x\overset{\Phi}{\longrightarrow}
\operatorname{Hom}(\mathcal D_x/\mathcal V_x,\mathbb R^k),
\qquad
v\longmapsto
\big([w]\mapsto\d\bm\eta_x(v,w)\big).
\]
Indeed, integrability of $\mathcal V$ makes this map well-defined,
non-degeneracy of $\d\bm \eta_x$ on $\ker\bm\eta$ makes it injective, and both spaces
have dimension $nk$. Let $W_x\subset\mathcal D_x$ be another $\d\bm\eta _x$-isotropic
subspace of dimension $nk$, and let $r$ be the dimension
of the projection of $W_x$ onto  $\mathcal D_x/\mathcal V_x$.
Let $A$ denote the image of $W_x$ in
$\mathcal D_x/\mathcal V_x$. Since $W_x$ is
$\d\bm\eta_x$-isotropic, every element of
$\Phi(W_x\cap\mathcal V_x)$ vanishes on $A$.
As $\Phi$ is injective, it follows that
\[
nk-r=\dim(W_x\cap\mathcal V_x)
\leq
\dim\operatorname{Hom}
\big((\mathcal D_x/\mathcal V_x)/A,\mathbb R^k\big)
=k(n-r).
\]
Thus, $(k-1)r\leq0$. Since $k>1$, we obtain $r=0$ and
$W_x=\mathcal V_x$. Thus, $\mathcal V_x$ is the unique $\d\bm\eta_x$-isotropic 
subspace of dimension $nk$. The Reeb flows preserve $\mathcal D$ and
its Levi bracket. Hence, they map integrable isotropic distributions of rank $nk$ in $\mathcal{D}$ onto integrable isotropic distributions of rank $nk$ in $\mathcal{D}$, which implies that they preserve $\mathcal V$.

Consider the local quotient
$\pi:U\to B$ by the Reeb distribution.
The differential form $\d\bm \eta$ descends to a closed 
$\bm \omega$ on $B$ with trivial common kernel.
Moreover, $\mathcal V$ projects to an integrable distribution
$\overline{\mathcal V}$ of rank $nk$, isotropic for 
$\bm\omega$. The $k$-symplectic Darboux theorem therefore
gives coordinates $(q^i,p_i^\alpha)$ on $B$ satisfying
\[
\bm\omega=\sum_{\alpha=1}^k\left(\sum_{i=1}^n\d q^i\wedge\d p_i^\alpha\right)\otimes e_\alpha\,,
\qquad
\overline{\mathcal V}
=\left\langle\partial_{p_i^\alpha}\right\rangle_{\substack{\alpha=1,\ldots,k\\i=1,\ldots,n}}\,.
\]
Using the same notation for their pull-backs to $U$, the form
$\bm \eta+\sum_{\alpha=1}^k\left(\sum_{i=1}^n  p_i^\alpha\d q^i\right)\otimes e_\alpha$ is  closed.
After shrinking $U$, write it as $\sum_{\alpha=1}^k\d z^\alpha\otimes e_\alpha$. Thus,
\[
\bm \eta=\sum_{\alpha=1}^k\left(\d z^\alpha-\sum_i p_i^\alpha\d q^i\right)\otimes e_\alpha\,.
\]
Since $R_\beta z^\alpha=\delta^\alpha_\beta$, these functions
give the required coordinates and
$R_\beta=\partial_{z^\beta}$.
Furthermore, $\mathcal V$ annihilates $q^i$ and $z^\alpha$,
so for $k>1$ these coordinates also satisfy
$\mathcal V=\langle\partial_{p_i^\alpha}\rangle$.
\end{proof}

\begin{definition} The local coordinates $\{x^i,y_i^\alpha,y^\alpha\}$ in Theorem \ref{Th:PolCokcontaDar} are referred to as \textit{$k$-contact Darboux coordinates}.
\qeddiamond\end{definition}

For simplicity, we will call $k$-contact Darboux coordinates simply Darboux coordinates, as this will not lead to any misunderstanding. The same applies to Darboux coordinates for different types of structures, whose precise meaning can be deduced from context. Note that there is no result about determining a basis of $\mathcal{V}$ in Theorem \ref{Th:PolCokcontaDar} of canonical form for $k=1$. Indeed, the following result explains this fact.
\begin{example}\label{Ex:CounterDarboux}
Consider $\mathbb{R}^3$ with coordinates $\{x,p,z\}$ and
\[
\eta=\mathrm dz-p\,\mathrm dx\,,\qquad
\mathcal V=\langle Y\rangle\,,\qquad
Y=\partial_x+p\partial_z+z\partial_p\,.
\]
Since $\iota_Y\eta=0$ and $\mathcal V$ is integrable of rank one,
$\mathcal V$ becomes a polarisation. However, the Reeb vector field,
$R=\partial_z$, does not preserve $\mathcal V$, since
\(
[R,Y]=\partial_p\) does not take values in \(\Gamma(\mathcal V).
\)
Consequently, there are no local coordinates $\{Q,P,Z\}$ satisfying simultaneously
\[
\eta=\mathrm dZ-P\,\mathrm dQ,\qquad
\mathcal V=\langle\partial_P\rangle,
\]
which would be, in principle, expectable.
Indeed, these expressions would imply $R=\partial_Z$ and
$[R,\Gamma(\mathcal V)]\subset\Gamma(\mathcal V)$.
Thus, even for $k=1$, Darboux coordinates for a fixed $k$-contact
form need not be adapted to a prescribed polarisation. This example also illustrates that Theorem 3.10 in \cite{GGMRR_20} has a  mistake in part (ii) of its proof, which was used to ensure that the Reeb vector fields leave invariant the polarisation, which does not hold in general.
\end{example}

Let us now detail a standard example for $k$-contact manifolds with a polarisation. 
\begin{example}\label{exa:first-jets}
    Let $J^{1}(M,E)$ be the first-order jet manifold  of the fibre bundle $\pi:E\rightarrow M$ with $ \dim E=k+m$ and $\dim M=m$. Let $\{x^{i},y^{\alpha},y^{\alpha}_ i\},$ with  $i=1,\ldots,m$ and $\alpha= 1,\ldots,k
$, be an adapted chart on an open neighbourhood $U$ of a point in $J^1(M,E)$. Recall that one may define a $k$-contact form $\bm \eta \in \Omega^{1}(U,\mathbb{R}^{k})$ for  $U\subset J^1(M,E)$ given by
\[
\bm \eta=\sum_{\
\alpha=1}^k\left(\d y^{\alpha} - \sum^{m}_{i=1} y^{\alpha }_i\d x^{i}\right)
\otimes e_\alpha.
\]
In fact, 
$$
\begin{gathered}
 \d\bm\eta=\sum_{
\alpha=1}^k\left(\sum^{m}_{i =1} \d x^{i} \wedge \d y^{\alpha}_i\right)\otimes e_\alpha,\qquad 
\ker \bm \eta=\left<\partial_{x^{i}}+\sum^{k}_{\alpha=1}y^\alpha_{i}
\partial_{y^{\alpha}}\,,\ \partial_{y^\alpha_{i}} \right> \,,\\
\ker\d\bm \eta=\left\langle\frac{\partial}{\partial y^{\alpha}}\right\rangle .
\end{gathered}
$$
If one considers the fibre bundle projection $j^1{\rm pr}_M\colon J^1(M,E)\rightarrow M$, it follows that 
$$
\ker \T j^1{\rm pr}_M\cap \ker \bm\eta=\left\langle \frac{\partial}{\partial y^\alpha_i}\right\rangle,
$$
which is globally defined because $\ker \bm \eta$ is the Cartan distribution on $U$ and $\ker \T j^1{\rm pr}_M$ is geometrically defined. Note that the adapted coordinates to $J^1(M,E)$ are indeed Darboux coordinates for the $k$-contact form $\bm \eta$, and $\mathcal{V} = \ker \T j^1{\rm pr}_M\cap \ker \bm\eta$ is the polarisation. 

\finish
\end{example}

Note that $k$-contact manifolds without a polarization do not need to be locally of the form of a first-order jet manifold. This means first-order jet manifolds do not give rise to a canonical local example to general $k$-contact manifolds.

Let us provide a last result concerning how to write $k$-contact forms for a general co-oriented $k$-contact manifold \cite[Proposition 3.6]{GGMRR_20}.

\begin{proposition}\label{Prop:GeneralDarbouxForm}

Let $(M,\bm \eta)$ be a co-oriented $k$-contact manifold. There exist local coordinates $z^\alpha,x^i$, with $\alpha=1,\ldots,k$ and $i=k+1,\ldots,m$ such that 
$$
R_\beta=\frac{\partial}{\partial z^\beta}\,,\qquad \bm \eta=\sum_{\alpha=1}^k\bigg(\d z^\alpha-\sum_{i=k+1}^mf^\alpha_i\d x^i\bigg)\otimes e_\alpha\,,\qquad \beta=1,\ldots,k\,,
$$
for certain functions $f^\alpha_i$, with $\alpha=1,\ldots,k$ and $i=k+1,\ldots,m$, that depend only on $x^{k+1},\ldots,x^m$.
\end{proposition}

\subsection{Field theories on co-oriented \texorpdfstring{$k$}{}-contact manifolds}
\label{Sec:kContHam}

Nowadays, the only approach to $k$-contact Hamiltonian dynamics is constructed exclusively for co-oriented $k$-contact manifolds \cite{GGMRR_20}. Let us review this topic.

\begin{definition}
    Given  $h\in \Cinfty(M)$ on a co-oriented $k$-contact manifold $(M,\bm\eta)$, called the \textit{Hamiltonian function}, the $k$-vector fields $\bfX^c_{h} = (X_\alpha)\in\X^k(M)$ satisfying the equations
    \begin{equation}\label{eq:contact-Ham-equations-trivial-1}
        \sum_{\alpha=1}^k\inn{X_\alpha}\d\eta^\alpha = \d h - \sum_{\alpha=1}^k(R_\alpha h)\eta^\alpha\,,\qquad \sum_{\alpha=1}^k\inn{X_\alpha}\eta^\alpha = -h\,,
    \end{equation}
     are called \textit{$k$-contact Hamiltonian $k$-vector fields}. Equations \eqref{eq:contact-Ham-equations-trivial-1} can be rewritten as
    \begin{equation}\label{eq:contact-Ham-equations-trivial-2}
        \sum_{\alpha=1}^k\Lie_{X_\alpha}\eta^\alpha = -\sum_{\alpha=1}^k(R_\alpha h)\eta^\alpha\,,\qquad \sum_{\alpha=1}^k\inn{X_\alpha}\eta^\alpha = -h\,.    
    \end{equation}
    We call $(M,\bm\eta,h)$ a \textit{$k$-contact Hamiltonian system}.
\qeddiamond\end{definition}

Note that in the co-oriented $k$-contact case, unlike in co-oriented contact manifolds, the components of a $k$-contact Hamiltonian $k$-vector field do not preserve the distribution $\ker \bm\eta$.

Let us consider a polarised co-oriented $k$-contact manifold of dimension $n+nk+k$. Taking Darboux coordinates $\{q^i,p_i^\alpha,z^\alpha\}$, the $k$-vector field $\bfX_{h}^c = (X_\alpha)$ has a local expression
$$ X_\alpha = \sum_{i=1}^n(X_\alpha)^i\parder{}{q^i} + \sum_{i=1}^n\sum_{\beta=1}^k(X_\alpha)^\beta_i\parder{}{p_i^\beta} + \sum_{\beta=1}^k(X_\alpha)^\beta\parder{}{z^\beta}\,, \qquad \alpha=1,\ldots,k\,,$$
where
\begin{equation}\label{eq:k-contact-HDW-fields-Darboux-coordinates}
    (X_\beta)^i = \parder{h}{p_i^\beta}\,,\qquad \sum_{\alpha=1}^k(X_\alpha)^\alpha_i = -\left( \parder{h}{q^i} + \sum_{\alpha=1}^kp_i^\alpha\parder{h}{z^\alpha} \right)\,,\qquad \sum_{\alpha=1}^k(X_\alpha)^\alpha = \sum_{j=1}^n\sum_{\alpha=1}^kp_j^\alpha\parder{h}{p_j^\alpha} - h\,,
\end{equation}
for $\beta=1,\ldots,k$ and $i=1,\ldots,n$. 
It is worth noting that each particular $h$ may have different $k$-contact Hamiltonian vector fields. In particular, there may exist $k$-contact Hamiltonian vector fields with a zero Hamiltonian function. We call them {\it$k$-contact gauge vector fields}.

Consider an integral section $\psi:\R^k\to M$ of the $k$-vector field $\bfX_{h}^c$, namely $\psi' = \bfX_{h}^c\circ\psi$, with local expression $\psi(t) = (q^i(t),p_i^\alpha(t),z^\alpha(t))$ where $t\in\R^k$. Then, $\psi$ satisfies the system of partial differential equations
\begin{equation}\label{eq:k-contact-HDW-Darboux-coordinates}
\begin{gathered}
    \parder{q^i}{t^\beta} = \parder{h}{p_i^\beta}\circ\psi\,,\qquad
\sum_{\alpha=1}^k\parder{p^\alpha_i}{t^\alpha} = -\left( \parder{h}{q^i} + \sum_{\alpha=1}^kp_i^\alpha\parder{h}{z^\alpha} \right)\circ\psi\,,\\
    \sum_{\alpha=1}^k\parder{z^\alpha}{t^\alpha} = \left( \sum_{\alpha=1}^k\sum_{j=1}^np_j^\alpha\parder{h}{p_j^\alpha} - h \right)\circ\psi\,,
\end{gathered}
\end{equation}
for $\beta=1,\ldots,k$ and $i=1,\ldots,n$. 
The equations \eqref{eq:k-contact-HDW-fields-Darboux-coordinates} and \eqref{eq:k-contact-HDW-Darboux-coordinates} are called \textit{$k$-contact Hamilton--De Donder--Weyl equations} or \textit{Herglotz--Hamilton--De Donder--Weyl equations} for $k$-vector fields and its integral sections, respectively \cite{GGMRR_20}.

In some very particular cases, the integral curves of a $k$-contact Hamiltonian $k$-vector field have been studied \cite{GGMRR_20}.

\section{A new approach: \texorpdfstring{$k$}{}-contact distributions}\label{Sec:kConMan}

Let us study the properties of a type of distribution that is locally given as the kernel of a $k$-contact form. This allows us to generalise the objects studied so far in $k$-contact geometry. Moreover, we will find that many properties analysed in $k$-contact geometry can be described without $k$-contact forms, sometimes even in a simpler manner via $k$-contact distributions or other differential one-forms taking values in $\mathbb{R}^k$. Nevertheless, let us give one of the most important definitions of this paper, which is a generalisation to $k$-contact geometry of the celebrated contact distributions.  It is worth noting that from the first version of this paper, this idea and related concepts have been used, written in different manners, and extended in several works (see for instance \cite{SobczakFrelik_26,HLM_25,deLucas2026JetBundles,DLRS_26,DLK_26,IzquierdoLopezDeLeonSchiavoneSoto_2025}).

\begin{definition}\label{def:k-contact-distribution}
    A {\it $k$-contact distribution} on $M$ is a distribution $\mathcal{D}\subset\T M$ such that, for each point $x\in M$, there is an open neighbourhood $U\ni x$ and a $k$-contact form $\bm\eta$ on $U$ such that $\restr{\mathcal{D}}{U} = \ker\bm\eta$. We say that $(M,\mathcal{D})$ is a {\it $k$-contact manifold}.
\qeddiamond\end{definition}

Since a one-contact form is a contact form, a one-contact manifold $(M,\mathcal{D})$ is a contact manifold. This means that every associated contact form $\eta$, i.e. $\restr{\mathcal{D}}{U} = \ker\eta$ on an open subset $U\subset M$, satisfies that $\eta\wedge (\d\eta)^n$ is a volume form for a unique $n\in\mathbb{N}$. As a consequence, it is said that $\mathcal{D}$ is maximally non-integrable.

This work aims to analyse $k$-contact distributions by means of the properties of the distribution itself, without relying on a given $k$-contact form or other geometric structures needlessly. To start with, let us study the relation of $k$-contact distributions with the maximal non-integrability notion defined via distributions \cite{Vit_15}. Shortly, it will be shown that this definition is equivalent to the maximally non-integrability notion for corank one distributions defined by their associated forms.

\begin{definition}
    Let $\mathcal{D}$ be a regular distribution on $M$ and let $\pi\colon\T M\rightarrow \T M/\mathcal{D}$ be the natural vector bundle projection. Then, $\mathcal{D}$  is {\it maximally non-integrable} in a {\it distributional sense} if $\mathcal{D}\neq 0$ and the vector bundle morphism $\rho\colon \mathcal{D}\times_M \mathcal{D}\rightarrow \T M/\mathcal{D}$ over $M$ given by
    $$
        \rho(v,v')=\pi([X,X']_x)\,,\qquad \forall v,v'\in \mathcal{D}_x\,,\qquad \forall x\in M\,,
    $$
    where $X,X'$ are vector fields taking values in $\mathcal{D}$ locally defined around $x$ such that $X_x=v$ and $X'_x=v'$, is non-degenerate.
\qeddiamond\end{definition}

To show that $\rho$ is well defined and it does not depend on the vector fields $X,X'$ chosen to extend $v,v'$, note that the vector bundle projection $\pi\colon\T M\rightarrow \T M/\mathcal{D}$ over $M$ can be locally described on a neighbourhood $U$ of each point $x\in M$, via an $\bm\zeta\in \Omega^1(U,\R^k)$ such that $\ker \bm\zeta = \mathcal{D}|_U$ and $\corank \mathcal{D}=k$. More specifically, there exists a trivialisation $\T U/\mathcal{D}|_U\simeq U\times \mathbb{R}^k$ and $\rho$ is locally given by $\rho(v,v')=\bm \zeta_x ([X,X']_x)$ for some $\bm\zeta\in \Omega^1(U,\R^k)$ such that $\ker \bm\zeta=\mathcal{D}|_U$. Hence,
\begin{equation}\label{eq:rho}
    \rho(v,v') = \bm\zeta_x([X,X']_x) = X_x\inn{X' }\bm\zeta  - X'_x\inn{X }
    \bm\zeta -\d\bm\zeta_x(X_x,X'_x) = -\d\bm\zeta_x(v,v')\,, \qquad \forall v,v'\in \mathcal{D}_x\,,
\end{equation}           
which is a well-defined skew-symmetric tensor field object regardless of the vector fields $X,X'$ taking values in $\mathcal{D}|_U$ with $X_{x}=v$ and $X'_{x}=v'$ used to define it. Note also that indeed any $\bm\zeta\in \Omega^1(U,\R^k)$ such that $\ker \bm\zeta=\mathcal{D}|_U$ gives rise to a trivialisation $\T U/\mathcal{D}|_U\simeq U\times \mathbb{R}^k$ and $\rho$ can be described  as $\rho(v,v')=-\d \bm\zeta_x(v,v')$ for every $v,v'\in \mathcal{D}_x$ and $x\in U$.

The notion of maximally non-integrability for corank one distributions is equivalent to their maximally non-integrability in a distributional sense, as illustrated by the following proposition.

\begin{proposition} Let $\mathcal{D}$ be a distribution of corank one. Then, $\mathcal{D}$ is 
 maximally non-integrable in a distributional sense if, and only if, $\mathcal{D}$ is maximally non-integrable in the contact sense.
\end{proposition}
\begin{proof}
    By definition of  maximal non-integrability in a distributional sense, $\rho\colon \mathcal{D}\times_M \mathcal{D}\rightarrow \T M/\mathcal{D}$ has a trivial kernel and $\mathcal{D}\neq 0$. Since $\T M/\mathcal{D}$ has rank one, $\rho_x$ is a skew-symmetric bilinear form $\rho_x\colon \mathcal{D}_x\times \mathcal{D}_x\rightarrow \T_xM/\mathcal{D}_x$, where $\dim \T_xM/\mathcal{D}_x=1$, and $\mathcal{D}$ has to have even rank different from zero because $\mathcal{D}\neq \{0\}$. Hence, $M$ is odd-dimensional and of dimension bigger than one. Let us define $\eta \in \Omega^{1}(U)$ on an open $U\subset M$ such that  $\ker \eta= \restr{\mathcal{D}}{U}$. This $\eta$ gives rise to a local trivialisation $\T U/\mathcal{D}|_U\simeq U\times \mathbb{R}$ and a representation of $\rho_x(v,v')=-\d \eta_x(v,v')$ for $x\in U$. Then, $\d\eta$ is a closed differential two-form on an odd-dimensional manifold of dimension at least three, and $\ker \d\eta$ is a regular distribution of odd rank. Moreover, $\ker \d\eta$ has zero intersection with $\ker\eta$ (otherwise $\rho$ would have non-trivial kernel), which implies that $\ker \d\eta$ has rank one. Thus, $\ker\d\eta\oplus \ker\eta = \T U$ and $\eta$ is a contact form. Hence, $\mathcal{D}$ is maximally non-integrable in a contact sense. 
    
    The converse follows from the fact that a contact form $\eta$ allows us to define a local representation of $\rho$ and $\ker \eta \neq 0$. Note that if $v_x\in \ker \rho_x$, then $-\d\eta(v_x,v'_x)=0$ for every $v'_x\in \mathcal{D}_x$. Meanwhile, $-\d\eta(v_x,R_x)=0$ for the value of the Reeb vector field $R$ of $\eta$ at $x$. Since $R_x\notin \ker \eta_x$, it follows that $v_x\in \ker \d\eta_x\cap \ker \eta_x$. Since $\eta$ is a contact form, it follows that $v_x=0$. 
\end{proof}

Since a distribution of corank one is maximally non-integrable in a distributional sense if and only if it is maximally non-integrable in the contact sense, we can drop the term ``in a distributional sense" to simplify the terminology without leading to any misunderstanding. Therefore, maximally non-integrable will mean maximally non-integrable in the distributional sense.

Note that the map $\rho$ can be surjective for a distribution that is not a contact one. For instance, consider the distribution on $\mathbb{R}^5$, with global coordinates $\{x_1,\ldots,x_5\}$, given by the kernel of the differential one-form $\zeta=\d x_1-x_2\d x_3$. It follows that  $\zeta$ is not a contact form since $\zeta\wedge (\d\zeta)^2=0$. Since every one-form associated with a contact distribution is a contact form, it follows that $\ker \zeta$ is not a contact distribution. Notwithstanding, 
$$
\ker\zeta=\left\langle X_1 = x_2\frac{\partial}{\partial x_1}+\frac{\partial}{\partial x_3}\,,\quad X_2=\frac{\partial}{\partial x_2}\,,\quad X_3=\frac{\partial}{\partial x_4}\,,\quad X_4=\frac{\partial}{\partial x_5}\right\rangle\,. 
$$
Hence, $[X_1,X_2]=-\partial/\partial x_1$, which does not take values in $\ker \zeta$ at any point of $M$. Thus, $[\ker\zeta,\ker\zeta]=\T M$ and $\rho$ is a surjective vector bundle morphism over $\mathbb{R}^{5}$, but $\ker \zeta$ is not a contact distribution.

\begin{example}\label{Ex:CartanDis}
    Consider again the first-order jet bundle $J^1=J^1(M,E)$ where $E\rightarrow M$ is a fibre bundle of rank $k$ with local adapted variables $\{x^i,y^\alpha,y_i^\alpha\}$ on an open $U\subset J^1$. The associated Cartan distribution,
    $$
        \mathcal{C}=\left\langle \frac{\partial}{\partial x^i}+\sum_{\alpha=1}^k y^\alpha_i\frac{\partial}{\partial y^\alpha},\frac{\partial}{\partial y^\alpha_i}\right\rangle\,,
    $$ 
    has rank $m(k+1)$. Define the local basis of vector fields on $U\subset J^1$ of the form 
    $$
        I_\alpha = -\frac{\partial}{\partial y^\alpha}\,,\qquad X_i = \frac{\partial}{\partial x^i} + \sum_{\beta=1}^ky_i^\beta\frac{\partial}{\partial y^\beta}\,,\qquad P_\alpha^i = \frac{\partial}{\partial y^\alpha_i}\,,\qquad \alpha=1,\ldots,k\,,\quad i=1,\ldots,m\,.
    $$
    Then, the only non-vanishing commutation relations are given by
    $$
        [X_i,P^j_\alpha] = \delta^j_i I_\alpha\,,\qquad \alpha=1,\ldots,k\,,\qquad i,j=1,\ldots,m\,.
    $$
    Consequently, $[\mathcal{C},\mathcal{C}]=\T J^1$, but we already know that this condition is not enough to ensure that a distribution on $J^1$ is a $k$-contact distribution. Indeed, this is not a sufficient condition even for contact manifolds. Let us define the distribution $\mathcal{I} = \langle I_\alpha\rangle$. Then, $[\mathcal{I},\mathcal{C}|_U] = \T U$ and $[\mathcal{I},\mathcal{I}]=\mathcal{I}$. By Example \ref{Ex:1JetkCon}, the distribution $\mathcal{I}$ is not globally defined in general. Meanwhile, $\rho\colon \mathcal{C}\times_{J^1}\mathcal{C}\to\T J^1/\mathcal{C}$ has trivial kernel, $\mathcal{C}\neq 0$, and $\mathcal{C}$ is maximally non-integrable. In fact, one can locally define $\bm \eta=\sum_{\beta=1}^k(\d y^\beta-\sum_{i=1}^m y^\beta_i\d x^i)\otimes e_\beta$ on $U$ and
    $$
\rho(X,Y)=-\iota_Y\iota_X\d \bm \eta=-\iota_Y\iota_{X}\sum_{\beta=1}^k\left(\sum_{i=1}^m\d x^i\wedge \d y^\beta_i\right)\otimes e_\beta
    $$for arbitrary vector fields taking values in $\mathcal{C}|_U$.  Hence, $\ker \rho^\beta|_U=\langle \partial/\partial y^\alpha_i\rangle_{\alpha\neq \beta}$ for $\beta = 1,\ldots,k$ and $\ker \rho|_U=\bigcap_{\alpha=1}^k\ker \rho^\alpha|_U =0$.
    \finish
\end{example}

Note that \eqref{eq:rho} shows that 
\begin{equation}\label{eq:max-non-int-result}
    \restr{\ker\rho}{U} = \ker (\restr{\d\bm \zeta}{\mathcal{D}\times_U \mathcal{D}})\,.
\end{equation}
In other words, $\mathcal{D}$ is maximally non-integrable if, and only if, $\bm\zeta$ is such that $\d\bm\zeta$ is non-degenerate when restricted to $\mathcal{D}$. This condition is invariant relative to the particular $\bm\zeta$ used to describe $\mathcal{D}$, which does not need to be a $k$-contact form. Hence, one has the following result.

\begin{proposition}
    A  distribution $\mathcal{D}$ on $M$ is maximally non-integrable if, and only if, for every $x\in M$ there exists an $\bm\zeta\in\Omega^1(U,\mathbb{R}^k)$, where $U$ is an open neighbourhood of  $x$, associated  with $\mathcal{D}$ such that $\d\bm\zeta$ is non-degenerate when restricted to $\restr{\mathcal{D}}{U}$ and $\mathcal{D}$ is not the zero distribution.
\end{proposition}

Let us show that every $k$-contact distribution is maximally non-integrable.

\begin{proposition}\label{Prop:MaxNonkCon}
    If $(M,\bm\eta)$ is a co-oriented $k$-contact manifold, then $\d\bm\eta$ is non-degenerate when restricted to $\ker \bm \eta$. In other words, $\ker\bm\eta$ is maximally non-integrable.
\end{proposition}
\begin{proof}
    Let us proceed by reducing to absurd. If $\d\bm \eta$ is degenerate when restricted to $\ker \bm \eta$ at a certain $x\in M$, then there exists a non-zero tangent vector $v\in \ker\bm\eta_x$ such that $\d\bm\eta_x(v,w) = 0$ for every tangent vector $w\in \ker \bm\eta_x$. One gets that $\d\bm\eta_x(v,R_x) = 0$ for each Reeb vector field $R$ at $x$. Since $\bm\eta$ has $k$ Reeb vector fields spanning $\ker\d\bm\eta$, these vector fields, along with a basis of $\ker\bm \eta$, give rise at $x$ to a basis of $\T_xM$. It follows that $0\neq v\in \ker \d\bm\eta_x\cap\ker\bm\eta_x$ and $\bm \eta$ is not a $k$-contact form. This is a contradiction and thus $\d\bm\eta$ must be non-degenerate when restricted to $\ker\bm\eta$. In view of \eqref{eq:max-non-int-result}, it follows that $\ker \bm\eta$ is maximally non-integrable.
\end{proof}
As every $k$-contact distribution is on a neighbourhood $U$ of every point $x \in M$ of the form $(U,\restr{\mathcal{D}}{U}=\ker\bm \eta)$, it follows that it is maximally non-integrable.
\begin{corollary}
    Every $k$-contact distribution is maximally non-integrable.
\end{corollary}

Once we have introduced $k$-contact distributions, it makes sense to talk about their relations to polycontact
structures \cite{VanErp2011,ChoiPonge2019}. For a regular distribution $\mathcal{D}$ of corank $k$,
the {\it polycontact condition} requires
$\lambda\circ\rho_x$ to be non-degenerate for
every $x\in M$ and every non-zero
$\lambda\in(\T_xM/\mathcal{D}_x)^*$.
This is stronger than maximal non-integrability of $\mathcal D$ in the distributional sense.
Meanwhile, the contact structures of arbitrary codimension
considered in \cite{ApostolovEtAl2020} require,
at every point $x\in M$, the existence of at least
one such $\lambda$ for which
$\lambda\circ\rho_x$ is non-degenerate, which implies $\ker \rho=0$, so that $\mathcal D$ is maximally non-integrable in the distributional sense. 
Neither polycontact nor contact structures of arbitrary codimension
ensure the existence of $k$ local pairwise
commuting infinitesimal symmetries whose span
is complementary to $\mathcal{D}$, as required
by Theorem~\ref{Prop:LocalEqu}. 

Apostolov et al.~\cite{ApostolovEtAl2020} study
Levi-nondegenerate distributions of arbitrary codimension endowed with torus actions. In their setting of Reeb type, a $k$-dimensional subalgebra of the torus Lie algebra, where $k$ is the codimension of the distribution, generates
commuting vector fields that preserve the distribution and span a complementary subbundle of the tangent bundle. Together with Levi nondegeneracy, these properties imply
that the distribution is $k$-contact in the sense of our paper.

\subsection{\texorpdfstring{$k$}{k}-Contact distributions and their Lie symmetries}

To prove that a maximally non-integrable distribution does not need to be a $k$-contact distribution, we will require to introduce some previous technical background. This will lead to study the Lie symmetries of $k$-contact distributions and the determination of a characterisation for $k$-contact distributions based on the properties of the distribution itself. In addition, this allows us to related $k$-contact geometry with the study of relevant types of distributions \cite{PR_01}.

\begin{definition}
    A {\it Lie symmetry} of a distribution $\mathcal{D}$ on $M$ is a vector field $X$ on $M$ such that $[X,\mathcal{D}]\subset \mathcal{D}$.
\qeddiamond\end{definition}

Since a $k$-contact distribution is locally the kernel of a $k$-contact form, which admits Reeb vector fields, one has the following immediate result.

\begin{corollary}\label{Cor:SymkCon}
    Every $k$-contact distribution $\mathcal{D}$ on $M$ admits, on an open neighbourhood $U$ of each point $x\in M$, an integrable $k$-vector field $\mathbf{S} = (S_1,\ldots,S_k)\in\X^k(U)$, whose components are Lie symmetries of $\mathcal{D}$ such that
\begin{equation}\label{eq:ConSymkCon}
        \restr{\mathcal{D}}{U}\oplus\langle S_1,\ldots,S_k\rangle = \T U\,.
    \end{equation}
\end{corollary}
\begin{proof}
     Since $\mathcal{D}$ is a $k$-contact distribution, there exists for every $x\in M$ an open neighbourhood $U$ of $x$ and a $k$-contact form $\bm \eta\in \Omega^1(U,\mathbb{R}^k)$ such that $\restr{\mathcal{D}}{U}=\ker\bm\eta$. Moreover, there exists a Reeb $k$-vector field $\bm R=(R_{1},\ldots,R_{k})$ and  $\ker\d\bm\eta = \langle R_{1},\ldots,R_{k}\rangle$. Since $\ker \bm \eta\oplus \ker \d\bm \eta=\T U$, one has that  $\restr{\mathcal{D}}{U}\oplus\langle R_{1},\ldots,R_{k}\rangle =\T U$. 

    One can verify that every Reeb vector field $R_{1},\ldots, R_{k}$ is a Lie symmetry of $\mathcal{D}$. Indeed, 
    $$
    \inn{[R_{\alpha},Y]}\bm\eta=(\Lie_{R_{\alpha}}\inn{Y}-\inn{Y}\Lie_{R_{\alpha}})\bm\eta=0\,, \qquad \alpha=1,\ldots,k\,,
    $$
    for every $Y \in \Gamma(\restr{\mathcal{D}}{U})=\Gamma(\ker\bm\eta)$. Hence, $[R_\alpha,Y]\in\Gamma(\ker\bm \eta)$ and each $R_\alpha$ is a Lie symmetry of $\mathcal{D}$. 
\end{proof}

It is immediate that there exist distributions that have $k$ commuting Lie symmetries but are not $k$-contact distributions (because they are involutive, for instance). We will show that maximal non-integrability is not a sufficient condition for a regular distribution to become a $k$-contact distribution by giving a counterexample. In fact, there are maximally non-integrable distributions that do not admit $k$ commuting Lie symmetries satisfying \eqref{eq:ConSymkCon}. Finding Lie symmetries for a general distribution depends on theorems concerning the existence of solutions for systems of partial differential equations, which are not generally available \cite{Olv_93}. Hence, we define instead the Lie flag of a distribution as follows \cite{PR_01}. This will be enough for our purposes.

\begin{definition}
    Let $\mathcal{D}\subset\T M$ be a distribution on $M$. Its {\it Lie flag} is a series of distributions $\mathcal{D},\mathcal{D}^{1)},\mathcal{D}^{2)},\dotsc$on $M$, where $\mathcal{D}^{\ell)}$ is the  distribution generated by the Lie brackets between the vector fields taking values in $\mathcal{D}$ and $\mathcal{D}^{\ell-1)}$, namely
    \begin{equation}\label{eq:derived}
       \mathcal{D}^{\ell)}=[\mathcal{D},\mathcal{D}^{\ell-1)}] \,,\qquad \ell\in\N\,,
    \end{equation}
    where we denote $\mathcal{D}^{0)}=\mathcal{D}$. We call {\it small growth function} of the Lie flag of $\mathcal{D}$ the vector function 
    $$\mathcal{G}_\mathcal{D}(x) = (\dim \mathcal{D}_x,\dim \mathcal{D}^{1)}_x,\ldots)\,,\qquad \forall x\in M\,.
    $$
\qeddiamond\end{definition}

\begin{remark}
    Since $\mathcal{D}$ is smooth by assumption, all the distributions $\mathcal{D}^{\ell)}$ are smooth. Note that each $\mathcal{D}^{\ell)}$ contains $\mathcal{D}^{\ell-1)}$. Let us prove it. Around a certain point $x\in M$, the distribution $\mathcal{D}$ is locally generated by a family of vector fields $X_1,\dotsc,X_s$. Moreover, we can assume that $X_1,\dotsc,X_p$ are linearly independent at $x$ with $p\leq s$. Consider some generators $Y_1,\dotsc,Y_{t}$ of $\mathcal{D}^{\ell-1)}$ so that the values of $Y_1,\dotsc,Y_r$ at $x$ form a basis of $\mathcal{D}^{\ell-1)}_x$, where $r\leq t$. Then, for every $X_i$, with $i=1,\dotsc,p$, there exists a function $f_{i}$ defined around $x$ such that $(X_if_{i})(x)=1$ and $f_{i}(x)=0$. Then,
    $$
        [X_i,f_{i}Y_j]_x= (Y_j)_x\,,\qquad i=1,\ldots,p\,,\qquad j=1,\ldots,r\,,
    $$
    and $\mathcal{D}_x^{\ell-1)}\subset \mathcal{D}^{\ell)}_x$ for every $x\in M$, so that $\mathcal{D}^{\ell-1)}\subset \mathcal{D}^{\ell)}$. Thus, for every $x\in M$, the sequence $\mathcal{G}_\mathcal{D}(x)$ is (not strictly) increasing and stabilises since it is upper-bounded by $\dim M$.
\end{remark}

Every distribution $\mathcal{D}$ of rank two such that $[\mathcal{D},\mathcal{D}]$ has rank three is maximally non-integrable. Moreover, every distribution $\mathcal{D}$ of rank three such that $[\mathcal{D},\mathcal{D}]$ has rank five is maximally non-integrable. Such distributions have many applications in control theory, non-holonomic systems, Riemannian geometry, et cetera \cite{PR_98,PR_01,Ra_23}.

\begin{proposition}\label{prop:derived-distributions}
    Let us define $\mathcal{D}=\ker \bm\eta$ for a $k$-contact form $\bm\eta$, and let $R_1,\ldots,R_k$ be the associated Reeb vector fields. Then, the one-parameter groups of diffeomorphisms of the Reeb vector fields leave invariant the rank of every $\mathcal{D}^{\ell)}$ for $\ell\in\N\cup\{0\}$. 
\end{proposition}
\begin{proof} 
    Let $\Phi_t^R$ denote an element of the one-parameter family of local diffeomorphism of a Reeb vector field $R\in\{R_1,\dotsc,R_k\}$ and $x\in M$. Note that $\bm\eta$ is invariant relative to its Reeb vector fields. In particular, $\Phi_t^{R*}\bm \eta_{\Phi^R_t(x)}=\bm\eta_x$ and, for every $v_x\in \mathcal{D}_x$, one has that
    $$
            \langle\bm \eta_{\Phi^R_t(x)},\T_x\Phi_t^{R}(v_x)\rangle = \langle\bm\eta_{x},v_x\rangle = 0 \quad\Longrightarrow \quad \T_x\Phi^R_{t}\mathcal{D}_x = \mathcal{D}_{\Phi^R_t(x)}\,,\qquad  \quad \forall x\in M\,,
    $$
    and for those values of $t\in \mathbb{R}$, where the flow of $R$ is well-defined. 
    Hence, $\mathcal{D}$ is invariant relative to the one-parameter group of diffeomorphisms of $R$. Meanwhile, it follows that
    $$ 
        \mathcal{D}^{\ell)}_{\Phi_t^R(x)} = \left\langle [X,Y]_{\Phi^R_t(x)} \mid X\in \Gamma(\mathcal{D})\,,\ Y\in \Gamma\left(\mathcal{D}^{\ell-1)}\right)\right\rangle\,,
    $$
    but by the induction hypothesis and elementary differential geometric facts
    $$
        [X,Y]_{\Phi^R_t(x)} = [\Phi^R_{t\ast}X',\Phi^R_{t\ast}Y']_{\Phi^R_t(x)} = \T_x\Phi^R_t [X',Y']_x\,,\qquad \forall x\in M\,,\quad\forall y\in \mathbb{R}\,,
    $$
    where $X' = \Phi^R_{-t\,\ast}X$ and $Y' = \Phi^R_{-t\,\ast}Y$ are vector fields taking values in $\mathcal{D}$ and $\mathcal{D}^{\ell-1)}$, respectively. Hence,
    $$ 
        \mathcal{D}^{\ell)}_{\Phi_t^R(x)} = \left\langle \T_x\Phi^R_t [X',Y']_x \mid X'\in \Gamma(\mathcal{D})\,,\ Y'\in \Gamma(\mathcal{D}^{\ell-1)})\right\rangle = \T_x\Phi^R_t\mathcal{D}^{\ell)}_x\,,
    $$
    for every $x\in M$ and $t\in\R$ where expressions are well-defined. 
\end{proof}

Proposition \ref{prop:derived-distributions} states that the one-parameter groups of diffeomorphisms of the Reeb vector fields of a $k$-contact distribution $\mathcal{D}$ leave invariant the subsets where $\mathcal{G}_\mathcal{D}(x)$ are constant. Since the Reeb vector fields commute between themselves and are linearly independent at each point, $\mathcal{G}_\mathcal{D}(x)$  must be constant along the integral submanifolds of the Reeb distribution. This gives a key to illustrate when a maximally non-integrable distribution is not a $k$-contact distribution.

\begin{theorem}
\label{thm:exis-max-nonintegr}
    There exist maximally non-integrable distributions that are not $k$-contact distributions.
\end{theorem}
\begin{proof}
    Consider $\mathbb{R}^4$ endowed with linear coordinates $\{x,y,z,t\}$ and the regular distribution given by $\mathcal{D}=\langle \partial_x,\partial_y+(x^3/3+z^2x+t^2)\partial_z+x\partial_t\rangle$. Then, 
    \vskip -.7cm
    \begin{align}
        \mathcal{D}^{1)} &= \langle \partial_x\,,\ \partial_y+(x^3/3+z^2x+t^2)\partial_z+x\partial_t\,,\ (x^2+z^2)\partial_z+\partial_t\rangle\,,\\
        \mathcal{D}^{2)} &= \langle \partial_x\,,\ \partial_y+(x^3/3+z^2x+t^2)\partial_z+x\partial_t\,,\ (x^2+z^2)\partial_z+\partial_t\,,\ 2x\partial_z\,,\ (2x^3z/3+t-zt^2)\partial_z\rangle\,.
    \end{align}
    \vskip -.3cm
    Evidently, $\mathcal{D}^{1)}$ has rank three everywhere, while $\mathcal{D}^{2)}$ has rank three when $x=0$ and $t(1-zt)=0$, and four elsewhere. The space $S\subset \mathbb{R}^4$ defined by $x=0$ and $t(1-zt)=0$ has to be invariant relative to the action of the Reeb vector fields if $\mathcal{D}$ is a two-contact distribution. Note that $S$ is a regular submanifold whose tangent space is given by the annihilator of the differential forms
    $$
        \theta^1=\d x\,,\qquad \theta^2=(1-2zt)\d t- t^2 \d z\,.
    $$
    Indeed, $\theta^1\wedge \theta^2$ does not vanish at any point of $\mathbb{R}^4$ and the subset  $x=0$ and $t(1-zt)=0$ is a regular submanifold. If $\mathcal{D}$ is a two-contact distribution, the Reeb vector fields of any associated $k$-contact form $\bm \eta\in \Omega^1(U,\mathbb{R}^2)$ must be tangent to $S$. In particular, at points where $x=0$ and $t=0$, one has that $\T S=\langle\theta^1,\theta^2\rangle^\circ=\langle\partial_y,\partial _z\rangle$. Note that the two Reeb vector fields must span such a subspace, i.e. $\ker \d \bm\eta=\langle \partial_y,\partial_z\rangle$ when $x=0$ and $t=0$. But  $\mathcal{D}=\langle \partial_x,\partial_y\rangle$ at such points of $S$, thus if $\mathcal{D}|_U=\ker \bm \eta$, then $\ker \d\bm\eta\cap \ker \bm \eta=\langle \partial_y\rangle$. This is a contradiction and $\mathcal{D}$ is not the kernel of any two-contact form. On the other hand, it is immediate that $\rho\colon \mathcal{D}\times_{\mathbb{R}^{4}}\mathcal{D} \rightarrow \T \mathbb{R}^{4}/\mathcal{D}$ is non-degenerate and $\mathcal{D}$ is maximally non-integrable. 
\end{proof}
Note that it is very difficult to find Lie symmetries of a general regular distribution. That is why the hint given by Proposition \ref{prop:derived-distributions} is so important to find a counterexample.

Previous results imply that a $k$-contact distribution  is a maximally non-integrable distribution admitting around each point $k$ commuting Lie symmetries giving a supplementary to the $k$-contact distribution. In what follows, we will see that the converse is also true and, in fact, the maximal non-integrability and the existence of $k$ commuting Lie symmetries spanning a supplementary distribution characterise $k$-contact distributions.

\begin{theorem}\label{Prop:LocalEqu} 
    A distribution $\mathcal{D}$ of corank $k$ on $M$ is a $k$-contact distribution if, and only if, it is maximally non-integrable and, around an open neighbourhood $U$ of every $x\in M$, admits an integrable $k$-vector field ${\bf S} = (S_1,\ldots, S_k)$ of Lie symmetries of $\mathcal{D}|_U$ such that
    \begin{equation}\label{eq:Dec}
        \langle S_1,\ldots,S_k\rangle \oplus \mathcal{D}|_U = \T U\,.
    \end{equation}
\end{theorem}
\begin{proof}
    The direct part is a consequence of Proposition \ref{Prop:MaxNonkCon} and Corollary \ref{Cor:SymkCon}.
    
    Let us prove the converse part. Given the vector fields $S_1,\ldots,S_k$ on $U$, consider the differential one-forms $\eta^1,\ldots,\eta^k$ vanishing on $\mathcal{D}$ and dual to $S_1,\ldots,S_k$ on $U$. Such differential one-forms are unique and exist due to decomposition \eqref{eq:Dec}. Then, they give rise to $\bm\eta = \sum_{\alpha=1}^k\eta^\alpha\otimes e_\alpha\in \Omega^1(U,\mathbb{R}^k)$ and $\eta^1\wedge\dotsb\wedge \eta^k\neq 0$ on $U$. For a basis $X_{k+1},\ldots,X_m$ of $\ker \bm \eta$ around $x$, one has
    $$
        \d\bm\eta(S_\beta,S_\gamma) = 0\,,\qquad \d\bm\eta(S_\beta,X_\alpha) = 0\,,\qquad \alpha=k+1,\ldots m\,,\quad \beta,\gamma=1,\ldots,k\,.
    $$
    Therefore, the rank of $\ker\d\bm\eta$ is at least $k$, and $S_1,\ldots,S_k$ span a supplementary distribution to $\mathcal{D}$ around $x\in M$. Moreover, there exists no tangent vector $v'_{x'}\in \ker \d\bm \eta$ such that $v_{x'}\wedge S_1(x')\wedge\dotsb\wedge S_k(x')\neq 0$ for $x'\in U$, as otherwise there would be an element $0\neq v_{x'}\in \ker \bm \eta_{x'}\cap \ker\d\bm\eta_{x'}$. Such an element would belong to the kernel of $\d\bm\eta_{x'}$ restricted to $\ker \bm\eta_{x'}$ and, since $\mathcal{D}$ is maximally non-integrable, one has that $v_{x'}=0$. This is a contradiction, $\ker \d \bm \eta$ has rank $k$,  while $
    \rk 
    \mathcal{D}>0$, $\mathcal{D}$ has corank $k$, and $\ker \d \bm\eta \cap \ker \bm\eta=0$. Hence, $\bm \eta$ is a $k$-contact form and $\mathcal{D}$ is a $k$-contact distribution.
\end{proof}

 \begin{definition}
    Given a $k$-contact distribution $\mathcal{D}$, we call an associated integrable $k$-vector field $\bf S$ of Lie symmetries of $\mathcal{D}$ a {\it $k$-contact $k$-vector field of Lie symmetries}.
 \qeddiamond\end{definition}

In many cases, Theorem \ref{Prop:LocalEqu} is significantly simpler to apply for determining if a distribution is a $k$-contact distribution than verifying the existence of $k$-contact forms. For instance, consider a Lie group $G$ and a basis $X_1^L,\ldots,X_r^L$ of left-invariant vector fields on $G$ and a basis $X_1^R,\ldots,X_r^R$ of right-invariant vector fields. Let $\mathcal{D}$ be the distribution spanned by $X_1^L$ and $X_2^L$. Assume that $[X_1^L,X_2^L]\notin \mathcal{D}$ and $X_3^R,\ldots,X_r^R$ commute among themselves. Assume that the left- and right-invariant bases agree at the identity. Then, $X_3^R,\ldots,X_r^R$ span a supplementary to $\mathcal D$ near the identity and are the Reeb vector fields of a locally associated $(r-2)$-contact form. Left translations give the same local construction near every point of $G$, hence $\mathcal D$ is an $(r-2)$-contact distribution of rank two.

\begin{example} An {\it Engel distribution} is a distribution of rank two on a four-dimensional manifold such that $[\mathcal{E},\mathcal{E}]$ and $[\mathcal{E},\mathcal{E}^{1)}]$ have rank three and four respectively  \cite{delPinoThesis,PV_20,PR_01}. Hence, $\mathcal{E}$ is maximally non-integrable. Apart from its mathematical interest, Engel distributions can also be found in the description of control theory systems \cite{HKMT_22}. Locally, an Engel distribution admits a coordinate system such that it is generated by the vector fields
$$
E_1=\frac{\partial}{\partial x^4},\qquad E_2=x^4\frac{\partial}{\partial x^3}+x^3\frac{\partial}{\partial x^2}+\frac{\partial}{\partial x^1}.
$$
It is immediate that $[\mathcal{E},\mathcal{E}]$ is spanned by  $E_1,E_2,\partial/\partial x^3$ giving rise to a distribution of rank three. This illustrates that $\mathcal{E}$ is maximally non-integrable. Moreover, one has the commuting Lie symmetries of $\mathcal{E}$  given by
$$
R_1=\frac{\partial}{\partial x^2},\qquad R_2=\frac{\partial}{\partial x^3}+x^1\frac{\partial}{\partial x^2}
$$
that span a supplementary distribution to $\mathcal{E}$. 
Hence, Engel distributions are two-contact distributions. In particular, the application of Theorem \ref{Prop:LocalEqu} shows that
$$
\bm \eta=(\d x^2-(x^3-x^4x^1)\d x^1-x^1\d x^3)\otimes e_1+(\d x^3-x^4\d x^1)\otimes e_2 
$$
is a two-contact form whose kernel is $\mathcal{E}$ and $R_1,R_2$ are associated Reeb vector fields.

Engel distributions are types of Goursat distributions \cite{PR_98,PR_01}. Goursat distributions have interesting mathematical properties and applications in control theory and non-holonomic systems \cite{PR_98,PR_01}. One can also prove that many Goursat distributions are $k$-contact distributions. For instance, consider the Goursat distribution $\mathcal{G}$ on $\mathbb{R}^5$ spanned by the vector fields
$$
G_1=\frac{\partial}{\partial x^5},\qquad G_2=x^5\frac{\partial}{\partial x^4}+x^4\frac{\partial}{\partial x^3}+x^3\frac{\partial}{\partial x^2}+\frac{\partial}{\partial x^1}.
$$
It is immediate that $\mathcal{G}$ is a maximally non-integrable distribution of rank two. 
Then, three commuting Lie symmetries are given by
$$
R_1=\frac{\partial}{\partial x^2}, \qquad R_2=x^1\frac{\partial}{\partial x^2}+\frac{\partial}{\partial x^3},\qquad R_3=\frac{1}{2}\left(x^1\right)^2\parder{}{x^2}+x^1\frac{\partial}{\partial x^3}+\frac{\partial}{\partial x^4}.
$$
They span a supplementary distribution to $\mathcal{G}$. Then, this gives rise to a three-contact distribution. Moreover, its associated three-contact form is
\begin{multline*}
\bm \eta=\left(\d x^2-(x^3-x^4x^1+\frac{1}{2}x^5(x^1)^2)\d x^1-x^1\d x^3+\frac{1}{2}(x^1)^2\d x^4\right)\otimes e_1\\+\left(\d x^3-x^1\d x^4+(x^1x^5-x^4)\d x^1\right)\otimes e_2 +(\d x^4-x^5\d x^1)\otimes e_3. 
\end{multline*}
These examples illustrate the relation of $k$-contact distributions with different types of relevant distributions. Many more examples can be found in \cite{PR_01}.
    
\end{example}

\begin{example} For an application  of two-contact distributions to control theory, one can consider the Engel distribution  $\mathcal{E}_c$ associated with the description of  a car  \cite{HKMT_22} on $\mathbb{R}^4$ with  global coordinates $p,q,x,y$ spanned by
$$
X_1=\frac{\partial}{\partial q} ,\qquad X_2=-\sin q\frac{\partial}{\partial p}+\ell\cos q\left(\cos p\frac{\partial}{\partial x}+\sin p\frac{\partial}{\partial y}\right),
$$
for a certain positive constant $\ell$. In fact,
$$
[X_1,X_2]=-\cos q\frac{\partial}{\partial p}-\ell \sin q\left(\cos p\frac{\partial}{\partial x}+\sin p\frac{\partial}{\partial y}\right),
$$
which does not take values in $\mathcal{E}_c$ at any point. In other words, $\mathcal{E}_c$ is maximally non-integrable.

The Lie algebra of Lie symmetries for $\mathcal{E}_c$ that respect the decomposition $\mathcal{E}_c=\langle X_1\rangle\oplus\langle X_2\rangle$ can be found in \cite{HKMT_22} and is isomorphic to $\mathfrak{so}(3,2)$. In particular, one has the Lie symmetries  given by
$$
R_1=\frac{\partial}{\partial x},\qquad R_2=\frac{\partial}{\partial y},
$$
and
$$
R_1'=x\frac{\partial}{\partial y}-y\frac{\partial}{\partial x}+\frac{\partial}{\partial p},\qquad R_2'=\ell \left(\sin p\frac{\partial}{\partial x}-\cos p\frac{\partial}{\partial y}\right)+\sin^2q\frac{\partial}{\partial q}.
$$
Note that $[R_1,R_2]=0$ and $[R'_1,R'_2]=0$, while $R_1,R_2$ or $R'_1,R'_2$ span two supplementary distributions to $\mathcal{E}_c$ for $q\notin \pi \mathbb{Z}$ and close to points with $q\in \pi\mathbb{Z}$, respectively. Additionally, $[R'_2,R_1]=[R'_2,R_2]=0$, but $[R'_1,R_i]\neq 0$ for $i=1,2$. This means that a $k$-contact distribution may have more than $k$ commuting Lie symmetries. 

But to have a $k$-contact distribution, one has to have on an neighbourhood of every point $k$ linearly independent Lie symmetries giving rise to a local supplementary of the $k$-contact distribution. Let us illustrate this fact for $\mathcal{E}_c$. For $q\notin \pi \mathbb{Z}$, the vector fields $R_1,R_2$ are Lie symmetries of $\mathcal{E}_c$ and they span a supplementary distribution to $\mathcal{E}_c$. At points with $q\in \pi \mathbb{Z}$, one has that $X_2$ is a linear combination of $R_1,R_2$. Nevertheless, if $q\in \pi\mathbb{Z}$, one has that $R_1',R_2'$ are commuting Lie symmetries of $\mathcal{E}_c$ locally spanning a supplementary to $\mathcal{E}_c$ . This perfectly illustrates the local character of Lie symmetries of $k$-contact distributions and some of their general properties, while analysing a practical   example with applications in control theory.
\end{example}

It is worth noting that the $k$-contact distribution notion is related to  problems in relativity, mechanical  systems, et cetera \cite{Ra_23}. 
\subsection{Topological properties of co-oriented \texorpdfstring{$k$}{k}-contact manifolds}

Let us study some facts concerning the $k$-contact distributions admitting a description via a globally defined $k$-contact form. In particular, we study topological properties of $k$-contact manifolds admitting a globally defined $k$-contact form and other related concepts.

First, let us give an example of a $k$-contact distribution that is not the kernel of some globally defined $k$-contact form.  

\begin{example}
    Consider the manifold $\R^2\times\mathbb{M}$, where $\mathbb{M}$ is the open Möbius band\footnote{We use the classical definition of the Möbius band as $\mathbb{M} = \quotient{[0,2\pi]\times(-1,1)}{\sim}$, where $\sim$ stands for the equivalence relation with $(0,\rho)\sim (2\pi,-\rho)$ for every $\rho\in (-1,1)$.}, with coordinates $(x,y,\phi,\rho)$ on an open subset $U$ for $\phi\in(0,2\pi)$ and $\rho\in(-1,1)$. Consider the regular distribution
    $$
        \restr{\mathcal{D}}{U} = \langle\partial_x,x\partial_\phi + \partial_y\rangle\,,
    $$
    which can be globally extended in a unique smooth  manner to $\mathbb{R}^2\times \mathbb{M}$. Let us prove that $\mathcal{D}$ is not the kernel of a globally defined two-contact form by reductio ad absurdum. If $\mathcal{D}$ is the kernel of a globally defined two-contact form, then one has two globally defined Reeb vector fields that, along with the given basis of $\mathcal{D}$, give rise to a global basis of vector fields on $\mathbb{R}^2\times \mathbb{M}$, which makes the manifold orientable. Hence, the first Stiefel--Whitney class of its tangent bundle is equal to zero \cite{Hus_94}. The Whitney product theorem states that if $E$ and $E'$ are vector bundles over $B$ and $E$ is trivial, then the first Stiefel--Whitney class of $E\oplus E'$ over $B$ is the class of $E'$. Hence, since the vector bundle $\T\mathbb{R}^2\times \mathbb{M}\rightarrow \mathbb{R}^2\times \mathbb{M}$ is trivial, the first Stiefel--Whitney class of $\T (\mathbb{R}^2\times \mathbb{M})$ is the first Stiefel--Whitney class of $\mathbb{R}^2\times \T\mathbb{M}\rightarrow \mathbb{R}^2\times\mathbb{M}$, which is not zero because $\mathbb{M}$ is not orientable. Hence, $\mathbb{R}^2\times \mathbb{M}$ is not orientable and we have reached a contradiction. Hence, $\mathcal{D}$ is not the kernel of any globally defined two-contact form.
    \finish
\end{example}

Since a $k$-contact distribution gives rise, only locally, to a $k$-contact form, then $k$-contact distributions are more general than $k$-contact forms. The reason is that the existence of a globally defined $k$-contact form assumes that $\mathcal{D}$ admits a globally defined family of Reeb vector fields, while Definition \ref{def:k-contact-distribution} does not. Although we already defined a co-oriented $k$-contact manifold as a pair $(M,\bm\eta)$, where $\bm \eta$ is a $k$-contact form on $M$, it is also convenient to introduce the following related definition.

\begin{definition}
    A $k$-contact manifold $(M,\mathcal{D})$ is {\it co-orientable} if there exists a $k$-contact form $\bm \eta\in \Omega^1(M,\mathbb{R}^k)$ such that $\mathcal{D}=\ker \bm \eta$.
\qeddiamond\end{definition}

There are numerous scenarios where it can be assumed that a $k$-contact distribution is not co-oriented. In particular, if $(M,\mathcal{D})$ is co-oriented, it admits $k$ linearly independent vector fields taking values in $\ker \d\bm \eta$, and $\ker \d\bm \eta$ is a vector bundle over $M$ whose first Stiefel--Whitney class is zero. More specifically, the Hairy Ball Theorem states that an even-dimensional sphere, $S^{2n}$, cannot have a globally defined non-vanishing vector field. This indicates that globally defined Reeb vector fields do not exist on $S^{2n}$, making any potential $k$-contact distribution on $S^{2n}$ not co-oriented. Note that if a compact manifold without boundary has a globally defined non-vanishing vector field, then its Euler characteristic is zero. The projective plane has no globally defined non-vanishing vector fields because has Euler characteristic equal to one. Since the Euler characteristic of the connected sum of two closed surfaces is the sum of their Euler characteristics minus 2, then the connected sum of $n$ tori $\mathbb{T}^{2}$ has global non-vanishing vector fields only for $n=1$. Other notable topological constraints also prevent the existence of such fields on $S^{2n-1}$. For instance, the Radon--Hurwitz numbers show that $
\rho(2n)-1$ determines the maximum number of globally defined, linearly independent vector fields on $S^{2n-1}$ \cite{Ada_62}. Specifically, the initial values of $\rho(2n)$, for $n\in\mathbb{N}$, are as follows:
$$ 2, 4, 2, 8, 2, 4, 2, 9, 2, 4, 2, 8, 2, 4, 2, 10, \ldots $$
These values indicate that $S^5$ can only have one globally defined non-vanishing vector field. In other words, no globally defined $k$-contact form exists on $S^5$ for $k>1$.

\begin{example}
Let us describe another $k$-contact distribution on a manifold $M$ that is not the kernel of a globally defined $k$-contact form. As a globally defined $k$-contact form induces $k$ globally defined Reeb vector fields, we describe a $k$-contact distribution that does not admit globally defined Reeb vector fields. In particular, let us describe a one-contact distribution on the first-order jet bundle $J^1(S^1,\mathbb{M})$ that is not co-oriented.

Consider now a non-trivial bundle given by the projection of a M\"obius open strip onto $S^1$, namely ${\rm pr}_{S^1}:\mathbb{M}\rightarrow S^1$,  its first-order jet bundle $j^1{\rm pr}_{S^1}:J^1(S^1,\mathbb{M})\rightarrow S^{1}$, and the projection $({\rm pr}_M)^1_0:J^1(S^1,\mathbb{M})\rightarrow \mathbb{M}$. Recall that $\mathbb{M}$ can be represented by the product $[0,2\pi]\times(-1,1)\subset \mathbb{R}^2$ with the identification $(0,\rho)\sim(2\pi,-\rho)$ for every $\rho\in (-1,1)$. Consider an atlas for $\mathbb{M}$ given by a coordinate system induced by a local coordinate $x$ on $(0,2\pi)\subset S^1$ and $y$ on their fibres of $\mathbb{M}$ pull-backed both to $U=(0,2\pi)\times(-1,1)$, and a second local coordinate system defined again by the pull-back to $\mathbb{M}$ of $\bar{y}=-y$ on $\bar{U}=(\pi,3\pi)|
_{\text{mod}\,2\pi}$ in $S^{1}$. Then, $\dot{\bar y}=-\dot{y}$ on one connected component of the intersection of domains and $y=\bar{y}$ on the other. The Cartan distribution is spanned by 
$$
\restr{\mathcal{C}}{U} = \left\langle \frac{\partial}{\partial x}+\dot y\frac{\partial}{\partial y},\frac{\partial}{\partial\dot y}\right\rangle,\qquad \restr{\mathcal{C}}{\bar{U}} = \left\langle \frac{\partial}{\partial x}+\dot {\bar{y}}\frac{\partial}{\partial \bar{y}},\frac{\partial}{\partial\dot {\bar{y}}}\right\rangle
$$
where $Y=\partial_x+\dot y\partial_y$ can be extended globally to $J^1(S^1,\mathbb{M})$. Consider the embedding of $\mathbb{M}$ into $J^1(S^1,\mathbb{M})$ given by $\jmath:(x,y)\in \mathbb{M}\hookrightarrow (x,y,0)\in J^1(S^1,\mathbb{M})$ and $\T_{(x,y,0)}({\rm pr}_{S^1})^1_0:\T_{(x,y,0)}J^1(S^1,\mathbb{M})\rightarrow \T_{(x,y)}\mathbb{M}$. Suppose that you have a globally defined Reeb vector field $R$ on $J^1(S^1,\mathbb{M})$. In particular, it is defined at points of $\jmath(\mathbb{M})$, and the projection of $R$ at points of $\jmath(\mathbb{M})$ onto $\mathbb{M}$ via the mappings $\T_{(x,y,0)}({\rm pr}_{S^1})^1_0$  is well-defined. Since $\restr{(\langle R\rangle \oplus \mathcal{C})}{\jmath(\mathbb{M})} = \restr{\T J^1(S^1,\mathbb{M})}{\jmath(\mathbb{M})}$, one has that $\T_{(x,y,0)}({\rm pr}_{S^1})^1_0(\langle R\rangle \oplus \mathcal{C}) = \T_{(x,y)}\mathbb{M}$. In fact,  $\T_{(x,y,0)}({\rm pr}_{S^1})^1_0(Y)$ and $\T_{(x,y,0)}({\rm pr}_{S^1})^1_0(R)$, for every $(x,y)\in \mathbb{M}$, give rise to two globally defined vector fields on $\mathbb{M}$ that span $\T \mathbb{M}$. This is a contradiction as $\mathbb{M}$ is not orientable.
\finish
\end{example}

Let us remark that every $k$-contact distribution on a manifold $M$ is the kernel of a globally defined differential one-form taking values in a vector bundle of rank $k$ over the manifold.

\begin{proposition}Every $k$-contact distribution $\mathcal{D}$ on a manifold $M$ is the kernel of a globally defined differential one-form $\bm \eta_E$ taking values in a vector bundle $E\rightarrow M$ of rank $k$.    
\end{proposition}
\begin{proof} Take $E=\T M/\mathcal D$ and let $\bm\eta_E$ be the quotient projection $\T M\to E$, regarded as a section of $\cT M\otimes_M E$. Its kernel is $\mathcal D$, and $E$ has rank $k$.
\end{proof}
The problem of $\bm \eta_E$  is that there exists no canonical manner to define a connection on $E\rightarrow M$, which in turn implies that there is no canonical differential for differential forms taking values in $E$. This makes this approach problematic, as the differential of $k$-contact forms is a key element in $k$-contact geometry.
\subsection{Representing \texorpdfstring{$k$}{}-contact distributions}

A $k$-contact distribution admits an integrable $k$-vector field of Lie symmetries. The determination of its existence for a certain regular distribution may be complicated, and determining conditions necessary and sufficient for its existence would be very useful. In particular, if a distribution $\mathcal{D}$ is given by the kernel of a $k$-contact form, the existence and form of an integrable $k$-vector field of Lie symmetries of $\mathcal{D}$, namely the Reeb $k$-vector field of the $k$-contact form, is immediate. In other cases, it is interesting to write $\restr{\mathcal{D}}{U}=\ker \bm \zeta$ for some $\bm \zeta\in \Omega^1(U,\mathbb{R}^k)$ and study whether there exists an integrable $k$-vector field of Lie symmetries via $\bm \zeta$ or whether $\mathcal{D}$ may or not be a $k$-contact distribution. Recall that two differential one-forms $\bm\zeta,\bar{\bm\zeta}\in \Omega^1(U,\mathbb{R}^k)$ are {\it compatible} if $\ker \bm\zeta=\ker \bar{\bm \zeta}$ is a regular distribution of corank $k$.

Let us hereafter determine whether an $\bm\zeta\in \Omega^1(U, \mathbb{R}^k)$ associated with a regular distribution is such that the distribution admits an associated $k$-contact  form $\bm {\eta}\in \Omega^1(U,\mathbb{R}^k)$. To illustrate one of the difficulties to determine the existence of $\bm \eta$, we provide the following examples showing that two compatible $\bm\zeta,\bm{\bar \zeta}\in \Omega^1(U,\mathbb{R}^k)$  may satisfy that $\ker \d\bm\zeta\neq \ker \d\bm {\bar\zeta}$.

\begin{example}\label{ex:DR-different-rank}
    Consider $\R^4$ with Cartesian coordinates $\{x,y,z,p\}$ and $\bm\eta= (\d x - y\,\d p)\otimes e_1+(\d z - p\,\d y)\otimes e_2$. Then, 
    $$ \ker \bm\eta = \left\langle y\parder{}{x} + \parder{}{p},p\parder{}{z} + \parder{}{y}\right\rangle\,,\qquad \ker \d\bm\eta = \left\langle\parder{}{x},\parder{}{z}\right\rangle\,$$
  and $\ker \d\bm\eta\cap \ker \bm\eta=0$. If follows that $\bm\eta$ is a two-contact form. Consider now ${\bm\zeta} =e^z\bm\eta$. In this case, $\ker\bm\eta=\ker {\bm\zeta}$, but $\d{\bm\zeta}=\d{\zeta}^1\otimes e_1+\d \zeta^2\otimes e_2$ has zero kernel since $\d(e^z \eta^1)$ is a symplectic form on $\R^4$. Thus, $\ker \d{\bm\zeta} = 0$ and ${\bm\zeta}$ is not a two-contact form. This example shows that $k$-contact geometry is much richer than contact geometry, where the multiplication of a contact form by a non-vanishing function gives rise to a new compatible contact form. In other words, every differential one-form associated with a contact distribution is a contact form.
    \finish
\end{example}
\begin{example}Recall Example \ref{Ex:1JetkCon} concerning $J^1(M,E)$ and its Cartan distribution. In that example, we got that, given adapted coordinates $\{x^i,y^\alpha,y^\alpha_i\}$, $\{\bar x^i,\bar y^\alpha,\bar y^\alpha_i\}$ on an open subset $U\subset J^1(M, E)$, one has that the $k$-contact forms associated with both coordinate systems satisfy
$$
\ker \bm\eta=\ker \bm{\bar \eta}\,,
$$
on $U$. This illustrates that
 $\bm\eta$ and $\bm{\bar \eta}$ are compatible. Quite remarkably, we also had $\ker \d{\bm\eta}\neq \ker \d{ \bar{\bm\eta}}$ in general.
 \finish
\end{example}
To study the problems of this section, it will be convenient to analyse the following mathematical object.

\begin{definition}
    Every $\bm \zeta\in \Omega^1(U,\mathbb{R}^k)$ associated with a regular distribution gives rise to a non-vanishing differential $k$-form $\Omega_{\bm\zeta} = \zeta^1\wedge\dotsb\wedge \zeta^k$ on $U$. A {\it conformal Lie symmetry of $\Omega_{\bm \zeta}$} is a vector field $X$ on $U$ such that $\Lie_X\Omega_{\bm \zeta}=f_X\Omega_{\bm\zeta}$ for some function $f_X\in\Cinfty(U)$.
\qeddiamond\end{definition}
The following results are immediate, but very practical. 

\begin{proposition}
Let $(M,\bm\eta)$ be a $k$-contact manifold and let $\Omega_{\bm\eta}$ be its associated differential $k$-form, then the Reeb vector fields are Lie symmetries of $\Omega_{\bm\eta}$. The space of conformal symmetries of $\Omega_{\bm \zeta}$, where $\zeta$ is compatible with $\bm \eta$, is a Lie algebra. 
\end{proposition}
\begin{proof} It is immediate that the space conformal Lie symmetries of $\Omega_{\bm \eta}$ is a Lie algebra.
 On the other hand,
    $$ \Lie_{R_\alpha}\Omega_{\bm\eta} = \Lie_{R_\alpha}(\eta^1\wedge\dotsb\wedge\eta^k) = (\Lie_{R_\alpha}\eta^1)\wedge\eta^2\wedge\dotsb\wedge\eta^k + \dotsb + \eta^1\wedge\dotsb\wedge(\Lie_{R_\alpha}\eta^k) = 0\,, $$
    for $\alpha=1,\ldots,k$. Hence, the Reeb vector fields are Lie symmetries of $\Omega_{\bm\eta}$.
    \end{proof}
\begin{proposition}
\label{prop:conf-Lie-sym}
    One has that $\bm \zeta,\bm {\bar\zeta}\in \Omega^1(U,\mathbb{R}^k)$ are compatible if, and only if, $\Omega_{\bm\zeta}$ and $\Omega_{\bm{\bar\zeta}}$ are proportional. Moreover, $\ker \bm\zeta=\ker \Omega_{\bm\zeta}$ and a vector field $X$ on $U$ is a Lie symmetry of $\ker\bm\zeta$ if, and only if, $X$ is a conformal Lie symmetry of $\Omega_{\bm\zeta}$. 
\end{proposition}
\begin{proof}
    The only non-trivial part of the proof of the statement of this proposition is that a Lie symmetry of $\ker\bm\zeta$ amounts to a conformal Lie symmetry of $\Omega_{\bm \zeta}$. Let us prove this using that, if $\bm \zeta=\sum_{\alpha=1}^k\zeta^\alpha\otimes e_\alpha$, one has that  $\ker \bm\zeta=\ker \Omega_{\bm \zeta}=\ker \left(\zeta^1\wedge\dotsb\wedge \zeta^k\right)$. On the other hand, if $Y\in \Gamma(\ker \bm \zeta)$ and $X$ is a Lie symmetry of $\ker \bm \zeta$, then $[X,Y]$ takes values in $\ker \bm \zeta$. Hence,
    \[
        0 = \inn{[X,Y]}\zeta^1\wedge\dotsb\wedge \zeta^k = (\Lie_X\inn{Y} - \inn{Y}\Lie_X)\zeta^1\wedge\dotsb\wedge \zeta^k = -(\inn{Y}\Lie_X)(\zeta^1\wedge\dotsb\wedge \zeta^k)\,.
    \]
    Thus,
    \begin{equation}\label{eq:Explanation}
        0=(\iota_Y\Lie_X\zeta^1)\zeta^2\wedge\dotsb\wedge \zeta^k+\dotsb+(-1)^{k-1}\zeta^1\wedge\dotsb\wedge\zeta^{k-1} (\inn{Y}\Lie_X\zeta^k)\,.
    \end{equation}
    Since $\zeta^1\wedge \dotsb\wedge \zeta^k$ is not vanishing, one obtains that $\inn{Y}\Lie_{X}\zeta^{\alpha}=0$ for every $Y \in \Gamma(\ker \zeta)$ for $\alpha=1,\ldots,k$, and $\Lie_{X}\zeta^\alpha=\sum_{\beta=1}^kf^\alpha_\beta\zeta^\beta$ for some functions $f_\beta^\alpha\in\Cinfty(U)$ with $\alpha,\beta = 1,\ldots,k$. Hence, $X$ is a conformal Lie symmetry of $\Omega_{\bm\zeta}$. 

    The converse follows immediately from the equality
    $$
        \inn{[X,Y]}\Omega_{\bm\zeta} = -\inn{Y}\Lie_X\Omega_{\bm\zeta} = -f_X\inn{Y}\Omega_{\bm\zeta} = 0\,
    $$
    for $Y\in \Gamma(\ker \bm \zeta)$ and $X$ being a conformal Lie symmetry of $\Omega_{\bm \zeta}$. 
\end{proof}
It should be stressed that Proposition \ref{prop:conf-Lie-sym} implies that 
 $X$ is a conformal  Lie symmetry of $\Omega_{\bm \eta}$, if, and only if, $\Lie_{X}\eta^\alpha=\sum_{\beta=1}^kf^\alpha_\beta\eta^\beta$ with $\alpha=1,\ldots,k$ for certain functions $f^\alpha_\beta\in \Cinfty(U)$ with $\alpha,\beta=1,\ldots,k$.
\begin{proposition}\label{Prop:Conkcontact}
    Let $\bm\zeta\in \Omega^1(M,\mathbb{R}^k)$ have a regular kernel different from zero and suppose that $\restr{\d\bm\zeta}{\ker \bm \zeta\times_U \ker \bm \zeta}$ is non-degenerate. Assume that $S_1,\ldots,S_k$ are conformal commuting Lie symmetries of $\Omega_{\bm \zeta}$ with $S_1\wedge\dotsb\wedge S_k\neq 0$ such that $\langle S_1,\dotsc,S_k\rangle$ is a supplementary to $\ker\bm\zeta$. Then, $\bm\zeta$ admits a compatible $k$-contact form ${\bm\eta}$.     
\end{proposition}
\begin{proof}
    If $S_1,\ldots,S_k$ are conformal Lie symmetries of $\Omega_{\bm \zeta}$, then each $\Lie_{S_\beta}\zeta^\alpha$ is a linear combination of $\zeta^1,\ldots,\zeta^k$ with some coefficient functions. In consequence, $[S_\beta, \ker \bm \zeta]\subset \ker \bm\zeta$.

    Consider now the unique one-forms ${\eta}^1,\ldots,{\eta}^k$ that are dual to $S_1,\ldots, S_k$ and vanish on $\ker \bm\zeta$. They exist because $S_1\wedge\dotsb\wedge S_k$ is not vanishing and they in fact span a distribution that is supplementary to $\ker \bm \zeta$. One has that $\restr{\d\bm\zeta}{\ker\bm\eta \times \ker\bm\eta}$ is non-degenerate.  If one works out the differentials $\d {\eta}^1,\ldots,\d {\eta}^k$, it turns out that $S_1,\ldots, S_k$ take values in the kernels of all them. Thus, $\ker\d {{\bm\eta}}$ has rank at least $k$. Moreover, $\ker \bm \zeta$ is maximally non-integrable and $\ker \bm \zeta=\ker\bm \eta$, which implies that $\ker \d \bm \eta$ has rank $k$ and $\ker \d\bm\eta\cap \ker \bm \eta=0$. Hence, ${\bm \eta}$ is a $k$-contact form compatible with $\bm \zeta$.
\end{proof}

\begin{proposition}\label{Prop:Necessary conditionskform}
An ${\bm\zeta}\in \Omega^1(U,\mathbb{R}^k)$ associated with a regular distribution $\mathcal{D}\neq 0$ of corank $k$ is compatible with a $k$-contact form if, and only if,  $\ker  {\bm\zeta}\cap \ker \d {\bm \zeta}=0$, while $\ker\left(\restr{\d {\bm\zeta}}{\ker \bm \zeta\times \ker {\bm \zeta}}\right)=0$, and $\Omega_{ {\bm \zeta}}$  has $k$ commuting conformal Lie symmetries spanning a supplementary to $\ker\bm\zeta$. 
\end{proposition}
\begin{proof}
    If $\ker \bm\eta=\ker {\bm \zeta}$ for a $k$-contact form $\bm\eta\in \Omega^1(U,\mathbb{R}^k)$, then $\ker \bm\eta=\mathcal{D}$ is maximally non-integrable. Moreover, $\bm\zeta$ gives rise to a trivialization $\T U/\restr{\ker\zeta}{U}$ and $\d\bm\zeta$ is also non-degenerate when restricted to $\ker\bm\zeta$. Moreover $\ker \bm\zeta\, \cap\, \ker \d\bm\zeta=0$, as otherwise $\ker\d\bm\zeta$ would be degenerated on $\ker\bm\zeta$. Since $\Omega_{\bm\zeta}$ is proportional to $\Omega_{\bm \eta}$, the Reeb vector fields of $\bm \eta$ are commuting conformal Lie symmetries of $\Omega_{\bm \zeta}$ spanning a supplementary to $\ker \bm \eta=\ker\bm\zeta$. The converse is given by Proposition \ref{Prop:Conkcontact}. 
\end{proof}

\begin{example} Let us suggest a simple method to obtain the Lie symmetries of a regular distribution $\mathcal{D}$ on a manifold $M$. Take a basis $X_1,\ldots,X_s$ of $\mathcal{D}$. A vector field $X$ on $M$ is a Lie symmetry of $\mathcal{D}$ if, and only if, $\Lie_XX_1\wedge\dotsb\wedge X_s=f_XX_1\wedge\dotsb\wedge X_s$ for a certain function $f_X$. This is in general simpler that verifying $\Lie_{X}\eta^{\alpha} = \sum_{\mu=1}^{k} f^{\alpha}_{\mu}\eta^{\mu}$ for $\alpha=1,\ldots,k$ and functions $f^{\alpha}_{\mu}\in\Cinfty(M)$ with  $\alpha, \mu = 1,\ldots,k$. Let us consider a particular example of a regular distribution $\mathcal{D}$ on $\mathbb{R}^4$ spanned by the vector fields
$$
X_1=\partial_x\,,\qquad X_2=\partial_y+x\partial_z+z\partial_t\,.
$$
Then,  $\bm \zeta\in \Omega^1(\mathbb{R}^4,\mathbb{R}^2)$ associated with $\mathcal{D}$, namely $\ker \bm\zeta=\mathcal{D}$, is given by
$$
\bm\zeta=(x\d y - \d z\,)\otimes e_1+(\d t -z\d y\,)\otimes e_2\,.
$$
This is not a two-contact form as $\d\bm\zeta=\d x\wedge \d y\otimes e_1+\d y\wedge \d z\otimes e_2$ and $\ker \d\bm\zeta$ has not rank two. Nevertheless, $\mathcal{D}$ is maximally non-integrable as $[X_1,X_2]=\partial_z$ does not take values in $\mathcal{D}$ at any point of $\mathbb{R}^4$ and $\mathcal{D}\neq 0$. 

To obtain the Lie symmetries of $\mathcal{D}$, we want to find $f_X\in \Cinfty(\mathbb{R}^4)$ such that 
$$
\Lie_X(X_1\wedge X_2)=\Lie_X[\partial_x\wedge (\partial_y+x\partial_z+z\partial_t)]=f_XX_1\wedge X_2\,.
$$

Taking into account the relations
$$
[X_1,X_2]=\partial_z=X_3\,,\qquad[X_3,X_2]=\partial_t=X_4\,,\qquad [X_4,X_\alpha]=0,\,\qquad [X_3,X_1]=0,
$$
for $\alpha=1,2,3,4$, 
one obtains a basis $X_1,X_2,X_3,X_4$ of the space of vector fields on $\mathbb{R}^4$. Using this fact, we write $X$ in the following form
$$
X=f_1X_1+f_2X_2+f_3[X_1,X_2]+f_4[[X_1,X_2],X_2]
$$
for certain uniquely defined functions $f_1,f_2,f_3,f_4\in \Cinfty(\mathbb{R}^4)$. Then,
\begin{align}
    \Lie_X (X_1\wedge X_2) &= [f_1 X_1+f_2X_2+f_3X_3+f_4X_4, X_1\wedge X_2] \\
    &= -(X_1f_1)X_1\wedge X_2+f_1X_1\wedge X_3 - f_2X_3\wedge X_2 -(X_2f_2)X_1\wedge X_2\\
    &- (X_1f_3)X_3\wedge X_2 - (X_2f_3)X_1\wedge X_3 \\
    &\quad + f_3X_1\wedge X_4 - (X_1f_4)X_4\wedge X_2 - (X_2f_4)X_1\wedge X_4=f_XX_1\wedge X_2\,.
\end{align}
Hence,
$$
f_1=X_2f_3\,,\qquad f_2=-X_1f_3\,,\qquad f_3=X_2f_4\,,\qquad  X_1f_4=0\,,
$$
and the coefficients of $X$ are of the form
$$
f_4=f_4(y,z,t)\,,\qquad f_3 =X_2f_4\,,\qquad f_2=-X_1X_2f_4\,,\qquad f_1=X_2^2f_4\,.
$$
and 
$$
    X=(X_2^2f_4)\partial_x-(X_1X_2f_4)\partial_y +(X_2f_4-xX_1X_2f_4)\partial_z+(f_4(y,z,t)-zX_1X_2f_4)\partial_t\,.
$$
In particular, if 
$f_4=t$, then $f_3=z$, $f_2=0$, $f_1=x$
and $X=
t\partial_t+z\partial_z+x\partial_x$. 
We can also choose $f_4=y$, and then $X=
\partial_z+y\partial_t\,.
$ Moreover, if $f_4=e^y$, then $X=e^y(\partial_x+\partial_z+\partial_t)$.  Note that both latter vector fields commute between themselves and span a distribution supplementary to $\mathcal{D}$ for $y\neq 1$, which implies that $\mathcal{D}$ is a $k$-contact distribution for points with $y\neq 1$. 
\finish
\end{example}




\section{Compact \texorpdfstring{$k$}{}-contact manifolds}\label{Sec:Compact}

This section analyses several examples of compact $k$-contact manifolds. We pay special attention to co-oriented and low-dimensional ones. Their analysis gives rise to the general study of integral submanifolds of Reeb distributions and an extension of the Weinstein conjecture on Reeb vector fields to the $k$-contact setting. It also suggests the analysis of a certain type of $k$-contact manifolds defined on Lie groups, hereafter called $k$-contact Lie groups. Additionally, the study of compact $k$-contact manifolds is relevant as a $k$-contact analogue of the analysis of compact contact manifolds.

\subsection{Co-oriented compact {\it k}-contact matrix Lie groups}

Let us provide some illustrative examples of compact matrix Lie groups admitting a $k$-contact distribution.  As a byproduct, a definition of $k$-contact Lie group appears.  
\begin{example}
\label{Ex:SU3}
    Let us provide a first example of compact $k$-contact manifold. Consider the matrix Lie group $\mathrm{SU}(3)$ of unitary $3\times 3$ complex matrices with determinant one and its matrix Lie algebra $\mathfrak{su}(3)$ of traceless skew-Hermitian $3\times 3$ matrices. One can define a basis of left-invariant vector fields $X_1^L,\ldots,X_8^L$ on $\mathrm{SU}(3)$ whose values at the identity in $\mathrm{SU}(3)$ are of the form $X_\alpha^L(\Id) = i\lambda_\alpha\in \T_\Id\mathrm{SU}(3)\simeq \mathfrak{su}(3)$ with $\alpha=1,\ldots,8$, for the {\it Gell-Mann matrices} \cite{Geo_82}:
\begin{equation}\label{gell-mann-matrices}
\begin{gathered}
    \lambda_1 = \begin{pmatrix}
        0 & 1 & 0 \\
        1 & 0 & 0 \\
        0 & 0 & 0
    \end{pmatrix}\,,\qquad
    \lambda_2 = \begin{pmatrix}
        0 & -i & 0 \\
        i & 0 & 0 \\
        0 & 0 & 0
    \end{pmatrix}\,,\qquad
    \lambda_3 = \begin{pmatrix}
        1 & 0 & 0 \\
        0 & -1 & 0 \\
        0 & 0 & 0
    \end{pmatrix}\,, \\
    \lambda_4 = \begin{pmatrix}
        0 & 0 & 1 \\
        0 & 0 & 0 \\
        1 & 0 & 0
    \end{pmatrix}\,,\qquad
    \lambda_5 = \begin{pmatrix}
        0 & 0 & -i \\
        0 & 0 & 0 \\
        i & 0 & 0
    \end{pmatrix}\,,\qquad
    \lambda_6 = \begin{pmatrix}
        0 & 0 & 0 \\
        0 & 0 & 1 \\
        0 & 1 & 0
    \end{pmatrix}\,,\\
    \lambda_7 = \begin{pmatrix}
        0 & 0 & 0 \\
        0 & 0 & -i \\
        0 & i & 0
    \end{pmatrix}\,,\qquad
    \lambda_8 = \frac{1}{\sqrt{3}}\begin{pmatrix}
        1 & 0 & 0 \\
        0 & 1 & 0 \\
        0 & 0 & -2
    \end{pmatrix}\,.
\end{gathered}
\end{equation}
Then,  
$$
    [X^L_\alpha,X^L_\beta]=\sum_{\gamma=1}^8 c_{\alpha \beta}\,^\gamma X^L_\gamma\,,\qquad \alpha,\beta=1,\dotsc,8\,,
$$
where the non-vanishing coefficients $c_{\alpha\beta}\,^\gamma$ are given by Table \ref{Tab:GellMann}.

\begin{table}[ht]
\begin{center}
\def\arraystretch{1.5}
\setlength{\tabcolsep}{20pt}
\begin{tabular}{|c|c||c|c|}
    \hline
    $\alpha\beta\gamma$ & $c_{\beta\alpha}\,\!^\gamma$ & $\alpha\beta\gamma$ & $c_{\beta\alpha}\,\!^\gamma$ \\
    \hline
    123 & $2$ & 345 & $1$ \\
    147 & $1$ & 367 & $-1$ \\
    156 & $-1$ & 458 & $ \sqrt{3}$ \\
    246 & $1$ & 678 & $ \sqrt{3}$ \\
    257 & $1$ & & \\
    \hline
\end{tabular}
\end{center}
\caption{Non-vanishing commutation relations for the basis matrices $\{i\lambda_1,\ldots,i\lambda_8\}$ in $\mathfrak{su}(3)$. Note that the structure constants are completely skew-symmetric in $\alpha,\beta,\gamma$, and only $c_{\beta\alpha}\,\!^\gamma$ for  $1\leq \alpha<
\beta<\gamma\leq 8$ are considered in this table.}
\label{Tab:GellMann}
\end{table}

Note that the dual one-forms to left-invariant vector fields are left-invariant differential one-forms. In fact, $\iota_{X^L_\alpha}\Lie_{X_\beta^R}\eta^L_\gamma=\Lie_{X^R_\beta}\iota_{X_\alpha^L}\eta^L_\gamma+\iota_{[X^L_\alpha,X^R_\beta]}\eta_\gamma^L=0$ for $\alpha,\beta,\gamma=1,\ldots,8$. Consider such a basis $\eta^1_L,\ldots,\eta^8_L$. Since one has that $\d \eta^\alpha_L(X^L_\beta,X^L_\gamma)=-\eta^\alpha_L([X^L_\beta,X^L_\gamma])$ with $\alpha,\beta,\gamma=1,\ldots,8$ for the dual  differential one-forms $\eta^1_L,\ldots,\eta^8_L$ to $X_1^L,\ldots,X_8^L$, it follows that
$$
\d \eta^{\alpha}_L=-\frac 12\sum_{\beta,\gamma=1}^8c_{\beta\gamma}\,\!^\alpha\eta_L^\beta\wedge \eta_L^\gamma\,,\qquad \alpha=1,\ldots,8\,.
$$
  
In particular,
$$
\d \eta_L^3=2\eta_L^1\wedge \eta_L^2+ \eta_L^4\wedge \eta_L^5- \eta_L^6\wedge \eta_L^7\,,\qquad \d \eta_L^8=  \sqrt{3}\left(\eta_L^4\wedge\eta^5_L+\eta_L^6\wedge\eta_L^7\right)\,.
$$
Thus, $\bm \eta=\eta^3_L\otimes e_1+\eta_L^8\otimes e_2$ is a two-contact form on $\mathrm{SU}(3)$. Moreover, the associated Reeb vector fields are $X^L_3,X^L_8$, which commute between themselves and whose integral curves passing through the identity give rise to a commutative Lie subgroup
$$
    \mathbb{T}^2= \left\{\begin{pmatrix}
        e^{i\theta_1} & 0 & 0 \\
        0 & e^{i\theta_2} & 0 \\
        0 & 0 & e^{-i(\theta_1+\theta_2)}
    \end{pmatrix}\ \Bigg\vert\  \theta_1,\theta_2\in\R\right\}
$$
diffeomorphic to a two-dimensional torus. Since $X_3^L,X_8^L$ are left-invariant, all integral submanifolds of the Reeb distribution $\ker\d\bm\eta=\langle X^L_3,X^L_8\rangle$ are of the form $A\mathbb{T}^2$ for an arbitrary element $A\in \mathrm{SU}(3)$. Consequently, all of them are diffeomorphic to a two-dimensional torus. 
\end{example}
\begin{example}\label{Ex:SU4}
    Let us show that the example on $SU(3)$ can be generalised to many other compact matrix Lie groups. In particular, the analysis of the Lie algebra, $\mathfrak{su}(4)$, of the $4\times 4$ unitary complex matrix Lie group $\mathrm{SU}(4)$ shows that a basis of $\mathfrak{su}(4)$ exists whose first eight elements $i\lambda_1,\ldots,i\lambda_8$ satisfy between themselves the same commutation relations given in Table \ref{Tab:GellMann}, while some additional non-vanishing commutation relations are established in Table \ref{Tab:SU4} involving $i\lambda_9,\ldots,i\lambda_{15}$.
\begin{table}[h!]
\begin{center}
\def\arraystretch{1.5}
\setlength{\tabcolsep}{20pt}
\begin{tabular}{|c|c||c|c|}
    \hline
    $\alpha\beta\gamma$ & $c_{\beta\alpha}\,\!^\gamma$ & $\alpha\beta\gamma$ & $c_{\beta\alpha}\,\!^\gamma$ \\
    \hline
    1 9 12 & $1$ & 6 11 14 & $1$ \\
    1 10 11 & $-1$ & 6 12 13 & $-1$ \\
    2 9 11 & $1$ & 7 11 13 & $1$ \\
    2 10 12 & $1$ & 7 12 14 & $1$ \\
    3 9 10 & $1$ & 8 9 10& $\sqrt{1/3}$\\
        3 11 12 & $-1$ & 8 11 12 & $\sqrt{1/3}$ \\
    4 9 14 & $1$ & 8 13 14 & $-\sqrt{4/3} $ \\
    4 10 13 & $-1$ & 9 10 15 & $\sqrt{8/3} $ \\
    5 9 13 & $1$ & 11 12 15 & $\sqrt{8/3} $ \\
    5 10 14 & $1$ & 13 14 15& $\sqrt{8/3}$ \\
    \hline
\end{tabular}
\end{center}
\caption{Second part of the commutation relations for $\mathfrak{su}(4)$ in the basis $\{i\lambda_1,\ldots,i\lambda_{15}\}$. The first part is as in Table \ref{Tab:GellMann}. Note that the structure constants are completely skew-symmetric in $\alpha,\beta,\gamma$, and only $c_{\beta\alpha}\,\!^\gamma$ for  $1\leq \alpha<
\beta<\gamma\leq 15$ are shown.}\label{Tab:SU4}
\end{table}

As in Example \ref{Ex:SU3}, define a basis of left-invariant vector fields $X_{1}^{L},\ldots,X_{15}^{L}$ on $\mathrm{SU}(4)$ such that $X^L_\alpha(\Id)=i\lambda_\alpha$ for $\alpha=1,\ldots,15$, where $\lambda_1,\ldots,\lambda_{15}$ can be found in \cite[pg. 88]{Pfe_03}. Then,  a basis $\eta^1_L,\ldots,\eta^{15}_L$  of  left-invariant one-forms on $\mathrm{SU}(4)$ dual to $X^L_1,\ldots,X^L_{15}$ shows that there exists a three-contact form $\bm \eta=\eta^3_L\otimes e_1+\eta^8_L\otimes e_2+\eta^{15}_L\otimes e_3$ whose Reeb distribution is spanned by $X^L_3,X^L_8,X^L_{15}$. In fact,  $X_3^L,X_8^L,X_{15}^L$ commute between themselves and
$$
    \d \eta_L^3=2\eta_L^1\wedge \eta_L^2+ \eta_L^4\wedge \eta_L^5- \eta_L^6\wedge \eta_L^7+ \eta^9_L\wedge \eta^{10}_L- \eta^{11}_L\wedge\eta^{12}_L\,,
$$
$$
    \d \eta_L^8 = \sqrt{3}\;(\eta_L^4\wedge\eta^5_L+\eta_L^6\wedge\eta_L^7)+\sqrt{1/3}\;\eta^9_L\wedge \eta^{10}_L+\sqrt{1/3}\;\eta^{11}_L\wedge \eta^{12}_L-\sqrt{4/3}\;\eta^{13}_L\wedge \eta^{14}_L\,,
$$
while 
$$
    \d \eta_L^{15}=\sqrt{8/3}\;\left(\eta^9_L\wedge \eta^{10}_L+\eta^{11}_L\wedge \eta^{12}_L+\eta^{13}_L\wedge \eta^{14}_L\right)\,.
$$
 The integral submanifolds of the Reeb distribution are diffeomorphic to a three-dimensional torus (see \cite{Pfe_03} for  further details on $\mathrm{SU}(4)$, its Lie algebra, and related facts)
 $$
    \mathbb{T}^3= \left\{\begin{pmatrix}
    e^{i\theta_1} & 0 & 0 &0\\
    0 & e^{i\theta_2} & 0 &0\\
    0 & 0 & e^{i\theta_3}&0\\
    0 & 0 & 0 &e^{-i(\theta_1+\theta_2+\theta_3)}
    \end{pmatrix}\ \Bigg\vert\  \theta_1,\theta_2,\theta_3\in\R\right\}\,.
$$
\end{example}   

It should be noted that the previous two examples could be extended to any ${\rm SU}(n)$ for $n>4$ in an analogous manner. This extension relies on choosing an analogue of the Gell-Mann matrices for every $\mathfrak{su}(n)$ and an associated basis $X_1^L,\ldots, X_r^L$ of left-invariant vector fields on ${\rm SU}(n)$ for $r=n^2-1$. Then, one can set the Reeb distribution to be spanned by the left-invariant vector fields related to a basis of the Cartan subalgebra of $\mathfrak{su}(n)$, namely $X_{2^2-1}^L,X_{3^2-1}^L,\ldots,X^L_{n^2-1}$. Then,  $\bm \eta=\sum_{i=2}^n\eta^{i^2-1}_L\otimes e_{i-1}$  gives rise to an $(n-1)$-contact form. Nevertheless, proving this requires the use of Lie algebra theory, e.g. root diagrams for Lie algebras, and a proper study will be postponed to a further work.  
\begin{example} \label{Ex:U2}
Consider the Lie group $\mathrm{U}(2)$ of unitary complex $2\times 2$ matrices. In this case, a basis of left-invariant vector fields is given by
$
    X_1^L,X_2^L,X_3^L,X_4^L\,, 
$ 
with non-vanishing commutation relations
$$
    [X_\alpha^L,X_\beta^L] = 2\sum_{\gamma=1}^3\epsilon_{\alpha\beta \gamma }X^L_\gamma \,,\qquad \alpha,\beta=1,\ldots,3\,,
$$
where $\epsilon_{\alpha\beta\gamma}$, with $\alpha,\beta,\gamma=1,\ldots,3$, is the Levi-Civita symbol. More in particular,
$$
X_1({\rm Id})=\left(\begin{array}{cc}
0&i\\i&0\end{array}\right),\quad 
X_2({\rm Id})=\left(\begin{array}{cc}
i&0\\0&-i\end{array}\right),\quad X_3({\rm Id})=\left(\begin{array}{cc}
0&1\\-1&0\end{array}\right),$$
$$X_4({\rm Id})=\left(\begin{array}{cc}
i&0\\0&i\end{array}\right).
$$
Then, $\eta_L^2$ and $\eta_L^4$ give rise to a two-contact form $\bm\eta=\eta_L^2\otimes e_1+\eta_L^4\otimes e_2$ because
$$
    \d\eta^2_L=-2\eta_L^3\wedge\eta_L^1\,,\qquad \d\eta^4_L=0\,.
$$
The Reeb vector fields are $X^L_2$ and $X^L_4$, which commute among themselves. Note that $X^L_2,X^L_4$ span a Reeb distribution with  integral compact submanifolds of the form 
$$
    A\left\{\begin{pmatrix}
        e^{i\theta_1} & 0 \\
        0 & e^{i\theta_2}
    \end{pmatrix} \Bigg\vert \,\theta_{1},\theta_{2} \in \mathbb{R}\right\}
$$
for each $A\in \mathrm{U}(2)$,
which are diffeomorphic to $\mathbb{T}^2$. It will be interesting in 
 following sections that $\d\eta_L^2\wedge \d\eta_L^4=0$.  \end{example}

 Consider a Lie group $G$ and a basis of left-invariant vector fields $X^L_1,\ldots,X^L_r$ on it. Let $\eta^1_L,\ldots,\eta^r_L$ be the associated basis of left-invariant differential one-forms on $G$. Assume that one can consider a case where $\bm \eta_L=\sum_{\alpha=1}^k\eta^\alpha_L\otimes e_\alpha$ is a $k$-contact form. Examples \ref{Ex:SU3} and \ref{Ex:U2} are particular instances of this. In general, one has the following definition.
 
 \begin{definition} A Lie group $G$ with a left-invariant $k$-contact form is hereafter called a \textit{$k$-contact Lie group}.
 \end{definition}
   
\subsection{On the Weinstein conjecture for compact {\it k}-contact manifolds}
Let us extend the typical Weinstein conjecture on the existence of closed integral curves of a Reeb vector field on a contact manifold to a $k$-contact realm. To do so, we start by an illustrative example of the behaviour of Reeb vector fields on compact $k$-contact manifolds.
\begin{example}
\label{Ex:T*R^{4}}
    Let us analyse an example of co-oriented $k$-contact compact manifold whose Reeb distribution admits closed and not closed integral submanifolds. Consider $\cT \mathbb{R}^4$ with global canonical coordinates $\{x^1,p_1,\ldots,x^4,p_4\}$ and the two-symplectic form 
    $$
        \bm \omega = \sum_{\alpha=1}^2\d x^\alpha\wedge \d p_\alpha\otimes e_1 + \sum_{\alpha=3}^4\d x^\alpha\wedge \d p_\alpha\otimes e_2\,.
    $$
    Moreover, $\bm\omega$ is exact as it admits an associated potential form 
    $$
        \bm \theta = \frac 12\sum_{\alpha=1}^2(-p_\alpha \d x^\alpha+x^\alpha\d p_\alpha)\otimes e_1 + \frac 12 \sum_{\alpha=3}^4(-p_\alpha \d x^\alpha+x^\alpha\d p_\alpha)\otimes e_2\,,
    $$
    namely $\bm\omega=\d\bm\theta$. Let us define the vector fields on $\cT \mathbb{R}^4$ of the form 
    $$
        X_1 = x^1\frac{\partial}{\partial p_1}-p_1\frac{\partial}{\partial x^1}+\sqrt{2}\left(x^2\frac{\partial}{\partial p_2}-p_2\frac{\partial}{\partial x^2}\right)\,,\qquad X_2 = x^3\frac{\partial}{\partial p_3}-p_3\frac{\partial}{\partial x^3}+x^4\frac{\partial}{\partial p_4}-p_4\frac{\partial}{\partial x^4}\,.
    $$
    Then, $[X_1,X_2]=0$  and  $X_1,X_2$ are Hamiltonian relative to $\bm\omega$, i.e. $\iota_{X_\alpha}\bm\omega=\sum_{\beta =1}^2\d h_\alpha ^\beta\otimes e_\beta$ for $\alpha=1,2$, with
    $$
    h_1^1=-(x^1)^2/2-(p_1)^2/2-(x^2)^2/\sqrt{2}-(p_2)^2/\sqrt{2}, \qquad h_1^2=0\,,
    $$
    $$h_2^1=0,\qquad h_2^2=-(x^3)^2/2-(p_3)^2/2-(x^4)^2/2-(p_4)^2/2\,.
    $$ 
    If we analyse 
    the compact submanifold $N\simeq S^3\times S^3\subset \cT \mathbb{R}^4$ given by the equations
    $$
        h_1^1=-1\,,\qquad h_2^2=-1\,,
    $$
    then, if $\jmath_N:N\hookrightarrow \cT\mathbb{R}^4$ is the standard embedding, one has that
    $$
        \bm\eta=\jmath_N^*\theta^1\otimes e_1+ \jmath_N^* \theta^2\otimes e_2 = \jmath_N^{*}\bm\theta
    $$
    is a two-contact form on $N$. In fact, $X_1,X_2$ are tangent to $N$, while $\d\bm \eta=\jmath_N^*\bm\omega$ and
    $$
        \inn{X_\alpha }\d\bm\eta=\sum_{\beta=1}^2\jmath_N^*\d h_\alpha^\beta\otimes e_\beta=0\,,\qquad  \jmath_N^*\iota_{X_\alpha}\theta^\beta=\delta_\alpha^\beta\,,\qquad \alpha,\beta=1,2\,.
    $$
    This means that $\ker \d\bm \eta=\langle X_1,X_2\rangle$, while $\bm\eta$ never vanishes and $\ker \bm \eta\neq 0$.  
    Moreover, $X_1,X_2$ become Reeb vector fields on $N$, which turns into a co-oriented two-contact six-dimensional compact manifold whose Reeb distribution admits non-closed and closed two-dimensional integral submanifolds. For instance, the integral submanifold 
    \begin{equation}\label{Eq:Dense}(\cos\theta_1,\sin\theta_1,2^{-1/4} \cos(\sqrt{2}\theta_1),2^{-1/4}\sin(\sqrt{2}\theta_1),\cos\theta_2,\sin\theta_2,\cos\theta_2,\sin\theta_2)\,,\qquad \forall \theta_1,\theta_2\in \mathbb{R}\,,
    \end{equation}
    is not closed, while
    \begin{equation}\label{Eq:Cyclic}  (\sqrt{2}\cos\theta_1,\sqrt{2}\sin\theta_1,0,0,\cos\theta_2,\sin\theta_2,\cos\theta_2,\sin\theta_2)\,,\qquad \forall \theta_1,\theta_2\in \mathbb{R}\,,
    \end{equation}
    is closed. In fact, the latter closed submanifold is diffeomorphic to a two-dimensional torus. To illustrate these facts, Figure \ref{Fig:Dense} shows the representation of the first three coordinates of \eqref{Eq:Dense}, and the last coordinates of \eqref{Eq:Cyclic}.

    \begin{figure}[ht]
\begin{center}\includegraphics[scale=0.5]{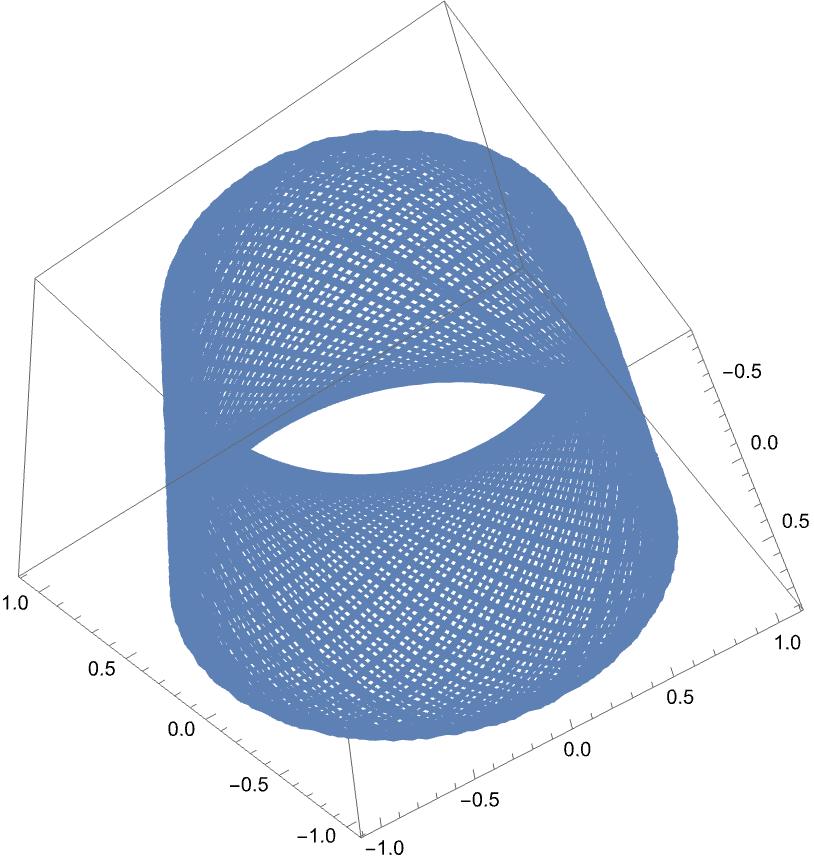}
$\qquad$\includegraphics[scale=0.5]{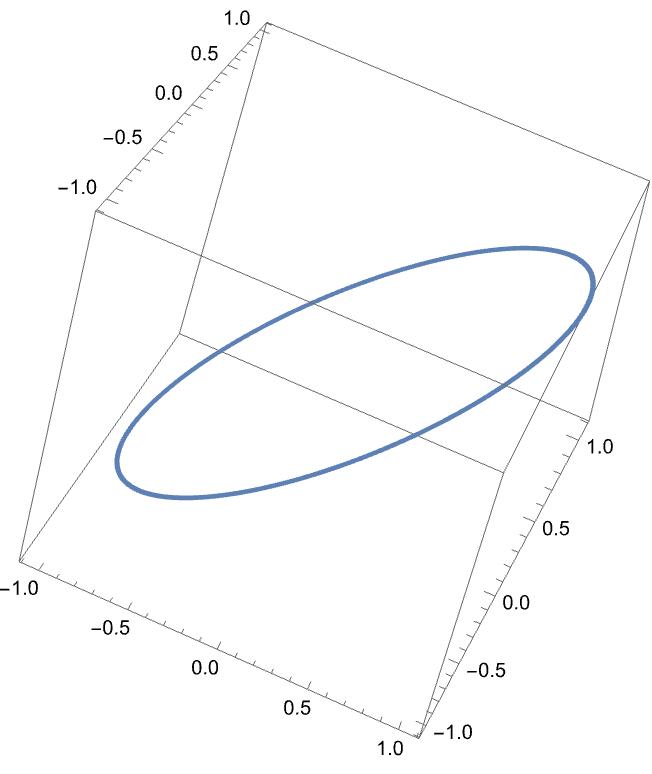}
\end{center}
    \caption{Representation of the first three coordinates in \eqref{Eq:Dense} in terms of $\theta_1\in[0,200\pi]$ and the last three coordinates of \eqref{Eq:Cyclic} for every $\theta_2\in \mathbb{R}$. }\label{Fig:Dense}
    \end{figure}
   
\end{example}

The previous examples raise the question of whether integral submanifolds of Reeb distributions for co-oriented $k$-contact compact manifolds can be homeomorphic to $k$-dimensional tori, extending the study of periodic orbits of Reeb vector fields in contact geometry to the $k$-contact setting. In particular, the following result can be proven.

Every connected closed $k$-dimensional integral submanifold
of the Reeb distribution of a compact co-oriented
$k$-contact manifold is diffeomorphic to $\mathbb{T}^k$.
Here, closed means compact and without boundary.
Indeed, the restrictions of the global Reeb vector fields
form a commuting frame on such a submanifold $L$.
Their completeness yields a transitive action of
$\mathbb{R}^k$, identifying $L$ with
$\mathbb{R}^k/\Gamma$ for a discrete subgroup $\Gamma$.
Compactness implies that $\Gamma$ is a lattice of rank $k$,
and hence $L\simeq\mathbb{T}^k$. This observation does not guarantee the existence of
a closed Reeb leaf. The following example shows that
a compact co-oriented $k$-contact manifold need not
admit such a leaf when $k>1$.

\begin{example}\label{Ex:CounterExample}
Let us consider the following counterexample showing that a co-oriented 
 $k$-contact compact manifold may not have Reeb leaves homeomorphic to a $\mathbb{T}^k$ torus for $k>1$. In particular, consider the two-contact distribution on the compact manifold $\mathbb{T}^2\times \mathbb{T}^2$ of the form
$$
\mathcal{D}=\left\langle X_1=\frac{\partial}{\partial \widetilde{\theta}_1}+\frac{\partial}{\partial \widetilde{\theta}_2},X_2=(1+\epsilon \cos \widetilde{\theta}_1)\frac{\partial}{\partial \theta_1}+(1+\epsilon \sin \widetilde{\theta}_1)\frac{\partial}{\partial \theta_2}+\frac{\partial}{\partial \widetilde{\theta}_2}\right\rangle
$$
and its commuting Lie symmetries
$$
R_1=\frac{\partial}{\partial \widetilde{\theta}_2},
\qquad R_2=\frac{\partial}{\partial \theta_1}+I\frac{\partial}{\partial \theta_2}
$$
for a certain $\epsilon$ and an irrational number $I$ such that $0<\epsilon<<I<<1$. Note that the above vector fields are well defined on $\mathbb{T}^2\times \mathbb{T}^2$, since their coefficients and the partial derivatives in terms of the chosen coordinates are well defined on $\mathbb{T}^2\times \mathbb{T}^2$. 

On the one hand, one has that $\mathcal{D}$ has rank two. Moreover, 
$$
X_1\wedge X_2\wedge R_1\wedge R_2=\left(1+\epsilon \sin\widetilde{\theta}_1-I(1+\epsilon \cos\widetilde{\theta}_1)\right)\frac{\partial}{\partial \theta_1}\wedge \frac{\partial}{\partial \theta_2}\wedge\frac{\partial}{\partial \widetilde{\theta}_1}\wedge\frac{\partial}{\partial \widetilde{\theta}_2}\neq 0,
$$
which is always different from zero due to the fact that $\epsilon,I$ are much smaller than one. Additionally, $[R_\alpha,X_\beta]=0$ for $\alpha,\beta=1,2$, which implies that $\mathcal{D}$ is invariant relative to the flows of $R_1,R_2$. Additionally,
$$
[X_1,X_2]=\epsilon \left(-\sin\widetilde{\theta}_1\frac{\partial}{\partial \theta_1}+\cos \widetilde{\theta}_1\frac{\partial}{\partial \theta_2}\right)\neq 0
$$
and it does not take values in $\mathcal{D}$ at any point. In other words, $\mathcal{D}$ is maximally non-integrable. Since $X_1,X_2,R_1,R_2$ are defined globally and linearly independent at every point, there exists a two-contact form $\bm\eta$, with $\ker\bm \eta=\mathcal{D}$, turning $(\mathbb{T}^2\times\mathbb{T}^2,\bm\eta)$ into a co-oriented two-contact compact manifold.

Let us now show that the above two-contact compact manifold has no Reeb leaves homeomorphic to $\mathbb{T}^2$. In fact, the Reeb leaves of $R_1,R_2$ are given by the immersed submanifolds parametrised in the form
$$
\widetilde{\theta}_1=c_1,\quad \widetilde{\theta}_2=\lambda,\quad \theta_1=\mu,\quad \theta_2=I\mu+c_2 ,\qquad \forall \mu,\lambda \in\mathbb{R},
$$
modulo $2\pi$, for constants $c_1,c_2\in \mathbb{R}$. Note that the above immersed submanifolds are not closed and they are homeomorphic to $S^1\times\mathbb{R}$. Indeed, they are not proper submanifolds in $\mathbb{T}^2\times\mathbb{T}^2$. The problem is that the integral curves of $R_2$ are not closed and their projections via $\pi:(\widetilde{\theta}_1,\widetilde{\theta}_2,\theta_1,\theta_2)\in\mathbb{T}^2\times\mathbb{T}^2\mapsto (\theta_1,\theta_2)\in\mathbb{T}^2$ give rise to dense subsets of $\mathbb{T}^2$. 
\end{example}

For products of compact contact manifolds without boundary, the following $k$-contact formulation of the Weinstein conjecture is equivalent to the classical Weinstein conjecture: {\it  every family of compact contact manifolds without boundary $(M_\alpha,\eta^\alpha)$, with $\alpha=1,\ldots,k$, gives rise to a $k$-contact manifold $(M=\prod_{\alpha=1}^kM_\alpha,{\bm \eta}=\sum_{\alpha=1}^k\eta^\alpha
\otimes e_\alpha)$ whose Reeb distribution admits, at least, one Reeb leaf that is homeomorphic to a torus  $\mathbb{T}^k$}. Note that we denoted by $\eta^\alpha$ the contact form on $M_\alpha$ and its natural pull-back to $M$. It is immediate that this conjecture retrieves the standard contact Weinstein one for $k=1$. On the other hand, if the Weinstein conjecture holds, the Reeb vector field of each $(M_\alpha,\eta^\alpha)$ has a closed orbit homeomorphic to $S^1$ and the Reeb leaf obtained by the Cartesian product of all of such orbits is homeomorphic to $S^1\times\overset{(k)}{\dotsb}\times S^1$, which amounts to $\mathbb{T}^k$.

The previous examples lead us to seek conditions ensuring
the existence of closed Reeb leaves. On a compact co-oriented
$k$-contact manifold, a Reeb leaf is closed precisely when it
is diffeomorphic, with its intrinsic manifold topology, to
$\mathbb{T}^k$. Thus, the existence of a toroidal Reeb leaf
and the existence of a closed Reeb leaf are equivalent.
We first make this observation explicit.

\begin{proposition}\label{prop:topology-two-contact-reeb-leaves}
Let $(M,\bm\eta)$ be a compact co-oriented $k$-contact
manifold without boundary. For every connected leaf $L$
of its Reeb distribution, the following conditions are
equivalent:
\begin{enumerate}
\item $L$ is closed as a subset of $M$;
\item $L$ is compact with its intrinsic manifold topology;
\item $L$ is diffeomorphic to $\mathbb{T}^k$.
\end{enumerate}
Under these conditions, $L$ is an embedded submanifold
of $M$.
\end{proposition}

\begin{proof}
The complete, commuting Reeb vector fields induce an
$\mathbb{R}^k$-action whose orbits are the connected Reeb
leaves. Their linear independence implies that the
stabiliser $\Gamma_x$ of $x\in L$ is discrete, so
$L\simeq\mathbb{R}^k/\Gamma_x$ in its intrinsic topology. In view of this, the below consequences follow:

$(1)\Rightarrow(2)$: a closed orbit is embedded, hence
its intrinsic topology agrees with the subspace topology;
compactness follows from that of $M$.

$(2)\Rightarrow(3)$: compactness of
$\mathbb{R}^k/\Gamma_x$ implies that $\Gamma_x$ has rank $k$ and becomes a
full lattice, i.e.
\(
\Gamma_x=\mathbb{Z}v_1+\cdots+\mathbb{Z}v_k
\)
for some basis $v_1,\ldots,v_k$ of $\mathbb{R}^k$, and $L\simeq\mathbb{T}^k$.

$(3)\Rightarrow(1)$: the inclusion of the compact leaf
$L$ into $M$ is an injective immersion, hence an
embedding with closed image. This also proves the
final assertion.
\end{proof}

More generally, if the discrete subgroup $\Gamma_x$
has rank $r$, then
\[
L\simeq\mathbb{R}^{k-r}\times\mathbb{T}^r,
\qquad
b_1(L):=\dim H_1(L;\mathbb{R})=r.
\]
Consequently, $b_1(L)=k$ is equivalent to $L$ being
closed in $M$. For $k=2$, this recovers the criterion
$b_1(L)=2$, or equivalently positive genus.
This topological characterisation does not, however,
ensure the existence of such a leaf. To obtain a
criterion that can be checked directly in examples,
we consider common first integrals of the Reeb
vector fields.

\begin{proposition}\label{prop:closed-reeb-leaves-level-sets}
Let $(M,\bm\eta)$ be a compact co-oriented $k$-contact
manifold without boundary, with $\dim M=n$ and Reeb
vector fields $R_1,\ldots,R_k$. Let $B$ be a smooth
manifold of dimension $n-k$, and let
\(
F\colon M\longrightarrow B
\)
be a smooth map satisfying
\[
\dd F(R_\alpha)=0,
\qquad \alpha=1,\ldots,k.
\]
If $b\in F(M)$ is a regular value of $F$, then every
connected component of $F^{-1}(b)$ is a closed
embedded Reeb leaf diffeomorphic to $\mathbb{T}^k$.
In particular, if $F$ is a submersion, this conclusion
holds for every nonempty fibre.
\end{proposition}

\begin{proof}
Let $L$ be a connected component of $F^{-1}(b)$.
Since $b$ is a regular value, $L$ is an embedded
submanifold of dimension $k$ and
\[
\T_xL=\ker \T_xF,
\qquad \forall x\in L.
\]
The condition $dF(R_\alpha)=0$ implies that the Reeb
distribution $\mathcal R$ satisfies
\(
\mathcal R_x\subseteq\ker dF_x.
\) 
Both spaces have dimension $k$, hence
\(
\mathcal R_x=\T_xL,
 x\in L.
\)

The map $F$ is constant along every Reeb leaf.
Thus, the Reeb leaf through a point of $L$ lies in
$L$, and the equality of tangent distributions makes
each such leaf open in $L$. Since $L$ is connected,
it consists of a single Reeb leaf.

The fibre $F^{-1}(b)$ is closed in the compact
manifold $M$, and its connected components are
closed. Hence, $L$ is compact and closed in $M$.
The preceding proposition yields
$L\simeq\mathbb{T}^k$.
\end{proof}

For a map
\(F=(F^1,\ldots,F^{n-k})\colon
M\longrightarrow\mathbb{R}^{n-k},
\)
the conditions of the proposition reduce to
\(
R_\alpha F^j=0\), for 
\(\alpha=1,\ldots,k,j=1,\ldots,n-k,
\)
and
\(
\operatorname{rank}\T_xF=n-k\) {for every } \(x\in F^{-1}(b).
\)
These conditions can be verified without integrating
the Reeb vector fields. Only the chosen nonempty
level needs to be regular; no global submersion onto
$\mathbb{R}^{n-k}$ is required.

The same argument applies to a map defined on an open
subset of $M$, provided that the chosen connected
component of its regular level is compact.
Local first integrals supplied by the Frobenius
theorem alone do not ensure this compactness.
The criterion therefore requires first integrals
defined on a sufficiently large domain together
with a compact regular level component.

Compact Reeb leaves also occur in earlier work. For example,
\cite[Example~2.3 and Proposition~2.5]{Ginzburg2007Coisotropic}
discusses characteristic foliations of stable coisotropic
submanifolds, whose compact leaves are tori. In the contact-type
case, these characteristic foliations coincide with the associated
Reeb foliations. Similarly, in the toric setting of
\cite{ApostolovEtAl2020}, the Reeb leaves are tori whenever the
transverse abelian action integrates to a closed subtorus.
 
\section{Polarised \texorpdfstring{$k$}{}-contact distributions}
\label{sec:polarised-k-contact-distributions}

A co-oriented $k$-contact manifold may not have Darboux-like coordinates (see \cite{GGMRR_20}). To ensure the existence of a Darboux theorem for a co-oriented $k$-contact manifold, we require the existence of a polarisation (see \cite{GGMRR_20} where polarisations are defined for $k$-contact structures). In this respect, this section defines and analyses polarised $k$-contact manifolds via distributions, which allows for the existence of a Darboux-like theorem, and we prove that every polarised $k$-contact manifold is locally diffeomorphic to the first-order jet manifold related to a bundle $E\rightarrow M$ of rank $k$. 

As already said, to have Darboux coordinates, we may impose the existence of a polarisation for a $k$-contact form. For instance, one has the following Darboux theorem for co-oriented $k$-contact manifolds $(M,\bm \eta)$ admitting a polarisation for $\bm \eta$ (see \cite{GGMRR_20} and Theorem \ref{Th:PolCokcontaDar}).

First, let us show now an example of a co-oriented two-contact manifold without Darboux coordinates.

\begin{example}\label{ex:2-contact-R6}
    Consider the manifold $\R^6$ with linear global coordinates $\{x, y, p, q, z, t\}$. Then,  
    $$
        \bm \eta= (\d z - p\,\d x)\otimes e_1+( \d t  - q\,\d y)\otimes e_2\in\Omega^1(\R^6,\R^2)
    $$
   is a two-contact form. Let us verify it. In first place, $\eta^1\wedge\eta^2 $ is a non-vanishing differential two-form and $\ker \bm \eta \neq 0$ has rank four. Moreover,
    $$
        \d\eta^1 = \d x\wedge\d p\,,\qquad\d\eta^2 = \d y\wedge\d q\,, \qquad \ker \d\bm\eta = \left\langle\parder{}{z},\parder{}{t}\right\rangle\,,
    $$
    and $\ker\d\bm\eta$ has rank two. Moreover, 
    $$
        \ker \bm\eta = \left\langle \parder{}{x} + p\parder{}{z}\,,\  \parder{}{y}  + q\parder{}{t}\,,\ \parder{}{p}\,,\ \parder{}{q} \right\rangle\,
    $$
and 
$\ker \d\bm\eta \cap \ker \bm\eta=0$, which is the third condition in Definition \ref{dfn:k-contact-manifold}. The Reeb vector fields are
    $$
        R_1 = \parder{}{z}\,,\qquad R_2 = \parder{}{t}\,.
    $$
    Note that the coordinates in this example are not Darboux-like. In fact, in view of Remark \ref{Rem:CompositePolarised}, this co-oriented two-contact manifold cannot be written as a polarised $k$-contact manifold: its dimension, six, is not a composite number minus one.  
\end{example}

To give more general Darboux forms for $k$-contact manifolds, let us prove that the existence of local Reeb vector fields in a $k$-contact manifold induces, at least locally, a coframe basis with special properties. This fact is detailed in the following proposition.

\begin{proposition}\label{Prop:Coframe}
    A $k$-contact form ${\bm \eta}$ on $U\subset M$ leads to a local basis $\eta^1,\ldots,\eta^k,\Xi^{k+1},\ldots,\Xi^{m}$ of differential one-forms on an open neighbourhood $U'$ of each $x\in U$ such that 
    $$
        \Lie_{R_\alpha}\Xi^\beta=0\,,\qquad \iota_{R_\alpha}\Xi^\beta=0\,,\qquad \alpha=1,\ldots,k\,,\qquad \beta= k+1,\ldots,m\,.
    $$
    In consequence, every $k$-contact distribution $\mathcal{D}$ induces a local basis $\eta^1,\ldots,\eta^k,\Xi^{k+1},\ldots,\Xi^m$ of differential one-forms on each neighbourhood $U_x$ of any point $x\in M$ so that $\bm \eta$ is a  $k$-contact form related to $\mathcal{D}$ on $U_x$ and $\Xi^{k+1},\ldots,\Xi^m$ vanish on the Reeb vector fields of ${\bm \eta}$.
\end{proposition}
\begin{proof}
    The proof follows from the fact that there exists, around every point $x$ in $M$, a local foliation of $M$ by $(m-k)$-dimensional submanifolds $\mathcal{S}_\mu$, with $\mu\in \mathbb{R}^k$, such that  $\T_x\mathcal{S}_\mu\oplus \ker(\d\bm\eta)_x=\T_x M$ at any   $x\in \mathcal{S}_{\mu}$ for every $\mu\in \mathbb{R}^k$. Taking locally defined linearly independent one-forms $\Xi^{k+1},\ldots,\Xi^m$ on a particular $\mathcal{S}_{\mu_0}$ that vanish on $R_1,\ldots,R_k$ on $\mathcal{S}_{\mu_0}$, one can extend them to an open subset $U\ni x$ of $M$ by requiring $\Lie_{R_\alpha}\Xi^\beta=0$ for $\alpha=1,\ldots,k$ and $\beta=k+1,\ldots,m$. The solution of these equations is unique as one fixes the value of $\Xi^{k+1},\ldots,\Xi^m$ on $\mathcal{S}_{\mu_0}$ and exists because the Reeb vector fields $R_1,\ldots,R_k$ commute between themselves. Since $\iota_{R_\alpha}\Xi^\beta=0$ for $\alpha=1,\ldots,k$
and $\beta=k+1,\ldots,m$ on $\mathcal{S}_{\mu_0}$ and $\Lie _{R_\gamma}\inn{R_\alpha}\Xi^\beta=0$ for $\gamma=1,\ldots,k$, one gets that $\iota_{R_\alpha}\Xi^\beta=0$ for $\alpha=1,\ldots,k$ and $\beta=k+1,\ldots,m$ on $U$.
    Since every $k$-contact distribution admits an associated, locally, $k$-contact form, the second part of our proposition follows.
\end{proof}

The coframe given by Proposition \ref{Prop:Coframe} allows for the following slight improvement of a result in \cite{GGMRR_20} characterising canonical expressions for $k$-contact forms.

\begin{theorem}\label{Th:TradkCont}
    For every $k$-contact form $\bm \eta$ on $M$, there exist, around every point of $M$, coordinates $z^1,\ldots,z^k$, Reeb-invariant one-forms $\Xi^{k+1},\ldots,\Xi^m$, namely $\Lie_{R_\alpha}\Xi^\beta=0$ and $\iota_{R_\alpha}\Xi^\beta=0$, and functions $f^\alpha_{\mu},g^\alpha_{\mu \nu},$ with $\alpha=1,\ldots,k$, and $\beta,\mu,\nu=k+1,\ldots,m$, such that   $$
        {\bm \eta}=\sum_{\alpha=1}^k\bigg(\d z^\alpha-\sum_{\mu=k+1}^{m}f^\alpha_{\mu}\Xi^\mu\bigg)\otimes e_\alpha\,,\qquad \d{\bm \eta}=\sum_{\alpha=1}^k\sum_{\mu,\nu=k+1}^mg^\alpha_{\mu\nu}\Xi^\mu\wedge\Xi^\nu\otimes e_\alpha\,.
    $$
\end{theorem}
\begin{proof}
It follows from  \cite[Proposition\,3.6]{GGMRR_20} or Proposition \ref{Prop:GeneralDarbouxForm} that there exists around an arbitrary point of $M$ a  coordinate system $\{z^\alpha,x^i\}$ and functions $h^\alpha_i$, with $\alpha=1,\ldots,k$ and $i=k+1,\ldots,m$ so that ${\bm \eta}=\sum_{\alpha=1}^k(\d z^\alpha-\sum_{i=k+1}^mh^\alpha_i\d x^i)\otimes e_\alpha$ with $R_\alpha=\partial/\partial z^\alpha$ for $\alpha=1,\ldots,k$. 
    As $\Xi^{k+1},\ldots,\Xi^m$ vanish  on Reeb vector fields, it follows that $$\sum_{i=k+1}^mh^\alpha_i\d x^i=\sum_{\mu=k+1}^mf^\alpha_\mu\Xi^\mu$$ 
    for some functions $f^\alpha_{k+1},\ldots,f^\alpha_m$, with $\alpha=1,\ldots,k$, that are uniquely defined. The condition $\Lie_{R_\alpha}\bm\eta=0$ for $\alpha=1,\ldots,k$, implies that $R_\nu f^\alpha_\mu=0$ for $\nu,\alpha=1,\ldots,k$ and $\mu=k+1,\ldots,m$. Then, $\d f^\alpha_\mu$ vanishes on $R_1,\ldots,R_k$, which implies that  $\d f^\alpha_\mu=\sum_{\nu=k+1}^m h^\alpha_{\mu\nu}\Xi^\nu$ for some Reeb invariant functions $h^\alpha_{\mu\nu}$.  Similarly, $\d\Xi^\mu$ vanishes on Reeb vector fields, which gives the form of $\d{\bm \eta}$.  
\end{proof}

In view of Theorem \ref{Th:TradkCont} and the fact that locally every $k$-contact manifold is locally a co-oriented $k$-contact manifold, it makes sense to define polarisations in $k$-contact geometry independently of the chosen $k$-contact form as follows.

\begin{definition}
    A $k$-contact distribution $\mathcal{D}$ on a manifold $M$ admits a {\it polarisation} if  $\dim M=(n+1)(k+1)-1$ and there exists an integrable subbundle $\mathcal{V}$ of $\mathcal{D}$ of rank $nk$. We call $(M,\mathcal{D},\mathcal{V})$ a {\it polarised $k$-contact manifold}.
\end{definition}

Note that this notion of polarisation is compatible with the polarisations introduced in the co-oriented case in Definition \ref{dfn:k-contact-manifold}. In fact, if $\mathcal{D}$ is chosen to be $\ker \bm\eta=\mathcal{D}|_U$ for a local $k$-contact form $\bm\eta$ associated with $\mathcal{D}$, then the triple $(U,\bm\eta,\mathcal{V}|_U)$ becomes a polarised co-oriented $k$-contact manifold.

\begin{example}\label{Ex:1JetkCon2} Let us consider again the first-order jet bundle
    $J^1=J^1(M,E)$ related to a fibre bundle $E\rightarrow M$ of rank $k$ with adapted variables $\{x^i,y^\alpha,y^\alpha_i\}$. In the given local adapted coordinates, its Cartan distribution,
    $$
        \mathcal{C}=\left\langle \frac{\partial}{\partial x^i}+\sum_{\alpha=1}^ky^\alpha_i\frac{\partial}{\partial y^\alpha},\frac{\partial}{\partial y^\alpha_i}\right\rangle\,,
    $$ 
has rank $m(k+1)$ and we already proved in Example \ref{Ex:CartanDis} that $\mathcal{C}$ is a $k$-contact distribution. Moreover, $\dim J^1(M,E)=m+mk+k=(m+1)(k+1)-1$. If $\pi^1_0:J^1(M,E)\rightarrow E$ is the standard first-order jet bundle projection onto $E$, then 
$$
\mathcal{C}\cap \ker \T\pi^1_0=\left\langle \frac{\partial}{\partial y_i^\alpha}\right\rangle
$$
is a globally defined distribution of rank $mk$ that is integrable and contained in $\mathcal{C}$. In fact, it is worth  noting that the distributions spanned by 
$$
\left\langle \frac{\partial}{\partial y_i^\alpha}\right\rangle,\qquad \left\langle \frac{\partial}{\partial \bar y_i^\alpha}\right\rangle
$$
for two different adapted coordinate systems $\{x^i,y^\alpha,y^\alpha_i\}$ and $\{\bar{x}^i,\bar{y}^\alpha,\bar{y}^\alpha_i\}$ are the same. 
Hence, $J^1(M,E)$ is a polarised $k$-contact manifold relative to the distribution $\mathcal{C}\cap \ker \T\pi^1_0$.

\end{example}

The previous example has shown that $J^1(M,E)$ is a natural instance of a polarised $k$-contact manifold. The following Darboux theorem, namely Theorem \ref{Th:DarkMan}, along with Theorem \ref{Th:DarbouxModel} will prove that it is indeed the canonical example, being every polarised $k$-contact manifold locally diffeomorphic to it.

\begin{theorem}\label{Th:DarkMan}(Darboux Theorem for polarised $k$-contact manifolds)
    Consider a polarised $k$-contact manifold $(M,\mathcal{D},\mathcal{V})$ of dimension $\dim M = n + nk + k$. Then, around every point of $M$, there exists a local associated $k$-contact form $\bm\eta$ and local coordinates $\{x^i,y^\alpha,y_i^\alpha\}$, with $1\leq\alpha\leq k$ and $1\leq i\leq n$, such that
    \begin{equation}\label{eq:kContPol}
        {\bm \eta} = \sum_{\alpha=1}^k\bigg(\d y^\alpha - \sum_{i=1}^ny_i^\alpha\d x^i\bigg)\otimes e_\alpha\,, \qquad \ker \d\bm\eta = \left\langle R_\alpha = \parder{}{y^\alpha}\right\rangle\,,\qquad \V = \left\langle\parder{}{y_i^\alpha}\right\rangle\,.
    \end{equation}
    These coordinates are called \textit{Darboux coordinates} of the polarised $k$-contact manifold $(M,\mathcal{D},\mathcal{V})$.
\end{theorem}
\begin{proof}
By Frobenius, there are local coordinates
$(x^i,y^\alpha,v^a)$, with $a=1,\ldots,nk$, such that,
after a suitable choice of the transverse coordinates,
\[
\mathcal{V}=\langle\partial_{v^a}\rangle,\qquad
\mathcal{D}=
\left\langle
\partial_{v^a},
D_i=\partial_{x^i}+
\sum_{\alpha=1}^k a_i^\alpha\partial_{y^\alpha}
\right\rangle.
\]
Since
\[
[\partial_{v^a},D_i]
=\sum_{\alpha=1}^k
\frac{\partial a_i^\alpha}{\partial v^a}
\partial_{y^\alpha},
\]
the non-degeneracy of the distributional curvature implies
that the square matrix
$(\partial a_i^\alpha/\partial v^a)_{(i,\alpha),a}$
is invertible. Hence, setting $y_i^\alpha=a_i^\alpha$
gives local coordinates $(x^i,y^\alpha,y_i^\alpha)$.
The associated form
\[
\bm\eta=\sum_{\alpha=1}^k
\left(\d y^\alpha-\sum_{i=1}^n y_i^\alpha\d x^i\right)
\otimes e_\alpha
\]
has kernel $\mathcal{D}$ and satisfies all the required
properties.
\end{proof}

The theorem below allows us to consider the manifold introduced in Example \ref{ex:canonical-k-contact-structure} as the canonical model for polarised $k$-contact manifolds.

\begin{theorem}\label{Th:DarbouxModel}Every polarised $k$-contact manifold $(M,\mathcal{D},\mathcal{V})$ is locally diffeomorphic to a first-order jet manifold related to a fibration of rank $k$.
\end{theorem}
\begin{proof}

By Theorem \ref{Th:DarkMan}, on a neighbourhood $U$ of every point, we may choose an associated local \(k\)-contact form \(\boldsymbol\eta\) and coordinates in which both \(\boldsymbol\eta\) and the prescribed polarisation have the expressions given in
 \eqref{eq:kContPol}.  More specifically, the polarisation $\mathcal{V}$ is such that  the involutive distribution $\ker \d\bm\eta$ satisfies $\mathcal{V}\cap \ker \d\bm\eta=0$. Moreover, $\mathcal{V}\oplus \ker\d\bm\eta$ is integrable. After shrinking $U$, the local leaf spaces $N=U/(\mathcal{V}\oplus \ker\d\bm\eta)$ and $E=U/\mathcal{V}$ are smooth, and the natural projection $E\to N$ is a fibre bundle of rank $k$. The adapted coordinates identify $U$ with an open subset of $J^1(N,E)$, preserving $\mathcal D$ and $\mathcal V$. Note that the Cartan distribution is $\mathcal{D} = \ker{\bm\eta}\supset\mathcal{V}$.
\end{proof}

\section{\texorpdfstring{$k$}{}-symplectization of \texorpdfstring{$k$}{}-contact structures}\label{Sec:kSymplecticCover}

This section is concerned with analysing the relation of a $k$-contact manifold with a $k$-symplectic manifold of larger dimension and some of its potential applications. In a sense, this is one of the possible generalisations of the results of \cite{GG_22} to the $k$-contact setting. Notwithstanding, as $k$-contact manifolds are much more complex than contact manifolds, the extensions to $k$-symplectic manifolds of larger dimension will lack some interesting properties, appearing only in the contact case.

\subsection{\texorpdfstring{$k$}{}-symplectic principal bundles}

Let us study a particular kind of $k$-symplectic manifolds that will be interesting in the study of some aspects of $k$-contact manifolds. We hereafter call $\mathbb{R}_\times$ the space of real positive numbers, which is a Lie group with respect to their multiplication. In what follows, $s$ is considered to be the natural pull-back to the Cartesian product of $\mathbb{R}_\times$ with another manifold of the natural variable in $\mathbb{R}_\times$.

\begin{definition}
    A \textit{$k$-symplectic $\R_\times$-principal bundle} $(P,\tau,M,\varphi,\bm \omega)$ is a $\R_\times$-principal bundle 
 $\tau:P\to M$ with an $\R_\times$-action
    $$ \varphi:\R_\times\times P\ni (s,x)\mapsto \varphi_s(x)\in P\,, $$
    endowed with a one-homogeneous $k$-symplectic form $\bm \omega$ on $P$, namely  $\bm\omega$ is a $k$-symplectic form satisfying  
    $ \varphi_s^\ast\bm \omega = s\bm \omega\,$ for every $s\in \mathbb{R}_\times. $
\end{definition}

Every $k$-symplectic $\R_\times$-principal bundle has a distinguished vector field $\Delta\in\X(P)$, called the \textit{Euler vector field}, given by
$$
    \Delta_x = \restr{\frac{\d}{\d s}}{s=0}(\varphi_{e^s}(x)) = \restr{\frac{\d}{\d s}}{s=1}\varphi_s(x)\,,\qquad \forall x\in P\,.
$$
The Euler vector field is the infinitesimal generator of the action $\varphi$ and it is therefore vertical relative to $\tau:P\rightarrow M$ and  does never vanish. One can also define a differential one-form on $P$ taking values in $\mathbb{R}^k$ given by $\bm\theta = \inn{\Delta}\bm\omega$, called the \textit{$k$-symplectic Liouville one-form}. Since $\Delta$ never vanishes and $\ker \bm \omega=0$, one has that $\bm\theta$ never vanishes too. Let us show now, among other interesting results, that $\d \bm\theta=\bm \omega$ and every $k$-symplectic $\mathbb{R}_\times$-principal bundle is therefore always exact. 

\begin{theorem}\label{thm:k-symplectic-bundle}
    If $(P,\tau,M,\varphi,\bm\omega)$ is a $k$-symplectic $\R_\times$-principal bundle, then $\Lie_\Delta\bm\omega = \bm\omega$, the $k$-symplectic Liouville one-form $\bm\theta$ is semi-basic, $\Lie_\Delta\bm\theta=\bm\theta$, and $\bm\omega = \d\bm\theta$. Furthermore, for any local trivialisation $\restr{P}{U}\simeq \R_\times\times U$ for some coordinated open $U\subset M$ of the principal $\mathbb{R}_\times$-principal bundle $\tau:P\to M$, we have $\Delta = s\dparder{}{s}$, and there exists a basic  $\bm{\widetilde{\zeta}}\in \Omega^1(P|_U,\mathbb{R}^k)$  such that
    \begin{equation}\label{eq:local}
        {\bm \omega}= \d s\wedge\bm{\widetilde{\zeta}}+ s\d\bm{\widetilde{\zeta}}\,,\qquad
        \bm\theta= s\bm{\widetilde{\zeta}}\,.
    \end{equation}
    Moreover, $\ker \bm{\widetilde{\zeta}}\cap \ker \d\bm{\widetilde{\zeta}}=\ker \T\tau$. If $\ker \bm\theta$ has corank $k$, $\ker \bm\theta/\ker \T\tau\neq 0$ and $\ker \bm\theta$ admits, locally around every point of $P$, some $k$ commuting Lie symmetries $\widetilde{R}_1,\dotsc,\widetilde{R}_k$ on $P|_U$  projectable onto $U$ such that $\widetilde{R}_1\wedge\dotsb\wedge \widetilde{R}_k\neq 0$ with $\langle \widetilde{R}_1,\ldots,\widetilde{R}_k\rangle \oplus \ker \bm \theta=\T(P|_U)$, and the $\mathbb R^k$-valued bilinear form induced by $\bm \omega$ on $\ker \bm \theta/\ker \T\tau$ has trivial common kernel, then $\bm{\widetilde{\zeta}}=\tau^*\bm\zeta$ for some unique $\bm \zeta\in \Omega^1(U,\mathbb{R}^k)$ and  $
    \ker \bm \zeta $ is a $k$-contact distribution on $U$. If $(M,\mathcal{D})$ is a $k$-contact manifold and $\mathcal{D}=\ker  \bm \zeta$ for some $\bm\zeta\in \Omega^1(M,\mathbb{R}^k)$, not necessarily a $k$-contact form, the converse also holds. 
\end{theorem}
\begin{proof}
    Note that
    $$ 
        \Lie_\Delta\bm \omega = \restr{\frac{\d}{\d s}}{s=0} \varphi_{e^s}^\ast \bm\omega = \restr{\frac{\d}{\d s}}{s=0} e^s\bm\omega = \bm\omega\,.
    $$
    Moreover, $\d\bm\theta = \d\inn{\Delta}\bm\omega = \Lie_\Delta\bm\omega - \inn{\Delta}\d\bm\omega = \bm\omega$, while $\inn{\Delta}\bm\theta = \inn{\Delta}\inn{\Delta}\bm\omega = 0$, which means that $\bm\theta$ is semi-basic relative to $\tau$, and $\Lie_\Delta\bm\theta = \Lie_\Delta\inn{\Delta}\bm\omega = \inn{\Delta}\Lie_\Delta\bm\omega = \inn{\Delta}\bm\omega =\bm \theta$.

    In a local trivialisation $\restr{P}{U}\simeq \R_\times\times U$ with local adapted coordinates $(s, x^i)$, we have
    $$
       \varphi_s(t,x) = (s\cdot t, x)\,, \qquad \forall s\in \mathbb{R}_\times\,,\qquad \forall t\in \mathbb{R}_\times\,,
    $$
    where $x = (x^i)$, which implies
    $$ 
        \Delta_{(s,x)} = \restr{\frac{\d}{\d t}}{t = 0} \varphi_{e^t}(s, x) = s\parder{}{s}\,,\qquad \forall (s,x)\in \mathbb{R}_\times\times U\,. 
    $$
    Moreover, $\bm\theta$ is semi-basic and we can write  $\bm\theta =\sum_{i=1}^m\sum_{\alpha=1}^k \mu^\alpha_i(s,x)\d x^i\otimes e_\alpha$, where we recall that we assume $m=\dim M$ and $\mu^\alpha_i(s,x)$ are certain uniquely defined functions on $\mathbb{R}_\times\times U$.
    
    Let us prove \eqref{eq:local} on $\tau^{-1}(U)$. Using that $\Lie_\Delta\bm\theta  = \bm\theta $, we have that there exist functions $f_i^\alpha(x)$ and $C^\alpha_i(x)$ such that, for every $i=1,\ldots,m$ and $\alpha=1,\ldots,k$, one gets\footnote{Note that the case of  zeros of the $\mu^\alpha_i$ has been taken into account in the last expression, while remaining ones assume non-vanishing $\mu^\alpha_i$.}
    $$
        s\parder{\mu^\alpha_i}{s} = \mu^\alpha_i\quad\Leftrightarrow\quad \frac{\partial \ln \abs{\mu^\alpha_i}}{\partial \ln s} = 1\quad\Rightarrow\quad \ln \abs{\mu^\alpha_i} = \ln s + f^\alpha_i(x)\quad\Rightarrow\quad \mu^\alpha_i = sC^\alpha_i(x)\,.
    $$
      
The variables $\{s,x^i\}$ are defined on a $\tau^{-1}(U)$ for a trivialisable open $U\subset M$. Hence, $\bm\theta(s,x)=s\bm{\widetilde{\zeta}}(x)$, where ${\widetilde{\bm \zeta}}$ is a basic differential one-form in $\mathbb{R}_\times \times U\rightarrow U$ taking values in $\mathbb{R}^k$.  It follows that
$$
\bm\omega =\d (s\bm {\widetilde\zeta}) = \d s\wedge \bm{\widetilde{\zeta}}+s\d \bm{\widetilde{\zeta}}\,.
$$
Note that if $v_p\in \ker {\widetilde{\bm\zeta}}\cap\ker \d{\widetilde{\bm\zeta}}$, then $v_p'=v_p-(\iota_{v_p}\d s)\Delta_p/s\in \ker \dd_p s$, then $v_p'\in \ker\bm\omega$ and $v_p'=0$. In other words, $\ker \T\tau\supset \ker {\widetilde{\bm\zeta}}\cap \ker\d \widetilde{\bm\zeta}$. Since $\ker \T\tau\subset \ker  {\widetilde{\bm\zeta}}\cap \ker\d\bm{\widetilde\zeta}$, it follows that $\ker \bm{\widetilde{\zeta}}\cap\ker \d\bm{\widetilde{\zeta}}=\ker \T\tau$.

Let us now impose the additional hypotheses. In particular, the form induced by $\bm\omega$ on $\ker \bm\theta/\ker \T\tau$ has trivial common kernel.  But the restriction of $\bm \omega$ to $\ker \bm{{\theta}}/\ker \T\tau$ is equal to the restriction of $s\d\bm{\widetilde{\zeta}}$ to $\ker \bm{\widetilde{\zeta}}/ \ker \T\tau$. Since ${\widetilde{\bm\zeta}}$ is basic and $\ker \bm \theta$ has corank $k$, it follows that  $\bm \zeta\in \Omega^1(U,\mathbb{R}^k)$ uniquely given by $\tau^*\bm\zeta=\widetilde{\bm \zeta}$ is such that $\ker \bm\zeta$ has corank $k$ and $\d\bm\zeta$ is non-degenerate on the subbundle $\ker \bm\zeta$. Since $\ker\bm\theta/\ker \T\tau\neq 0$, one has that $\ker \bm \zeta$ is non-zero and becomes a  maximally non-integrable distribution of corank $k$. On the other hand, the Lie symmetries of $\ker \bm \theta$ are projectable by assumption, and their projections are linearly independent since $\langle \widetilde{R}_1,\ldots,\widetilde{R}_k\rangle \cap \ker \T \tau=0$, and give rise to Lie symmetries of $\ker\bm \zeta$ that are supplementary to $\ker \bm \zeta$.  
Hence, $\ker \bm \zeta$ is a $k$-contact distribution. 

The converse is straightforward.
\end{proof}

Note that the last paragraph  in the previous proof concerned $\bm \zeta$, which is defined on $U$, but $\bm \zeta$ is not globally defined in general. It does not need to be a $k$-contact form either. If we apply Theorem \ref{thm:k-symplectic-bundle} to two trivialisations on $U_1,U_2$ with $U_1\cap U_2\neq \emptyset$, then ${\bm \theta}_1={\bm \theta}_2$ on $\tau^{-1}(U_1\cap U_2)$ because they are defined in terms of geometric objects. But the images of  a point $p\in \tau^{-1}(U)$ under the trivialisations are $(s_1,\tau(p))$ and $(t(\tau(p))s_1,\tau(p))$ for a certain function $t:U_1\cap U_2\rightarrow \mathbb{R}_\times$. Consequently, $ {\bm \zeta}_2=t^{-1}{\bm \zeta}_1$. 

The converse in Theorem \ref{thm:k-symplectic-bundle} is such that the $\bm \omega$ on $P$ depends on $\bm\zeta$. 
For a general one-contact manifold $(M,\mathcal D)$, the standard symplectic cover is the annihilator line bundle $\mathcal D^\circ$ with its zero section removed, with structure group $\mathbb R^*$. In the co-oriented case, a choice of co-orientation selects a positive-ray $\mathbb R_\times$-principal cover. The local constructions glue because associated contact forms are conformally proportional on overlaps. Nevertheless, this cannot be done in the $k$-contact case because the $\bm\eta_a\in \Omega^1(U_a,\mathbb{R}^k)$ are not conformally proportional. One could consider another extension to $\mathbb{R}^k_0\times M$, with $\mathbb{R}^k_0=\mathbb{R}^k\backslash{0}$, solving this particular problem, but new ones would appear (e.g. it is not a principal bundle anymore or there will be problems with lifts of other geometric structures from $M$ to $\mathbb{R}^k_0\times M$ as shown posteriourly in this work). Moreover, we prefer in this work to focus on the properties of geometric structures on the manifold where they are defined, and we use extensions or lifts to other manifolds when they seem to be adequate to study particular, although quite interesting, problems, e.g. a $k$-contact problem can be effectively studied through a better known geometric structure. This is the standard approach in differential geometry from the very first works by Riemann and Gauss.

Previous result suggests the following definition. 

\begin{definition} A {\it $k$-symplectic cover} of a $k$-contact manifold $(M,\mathcal{D})$ is a homogeneous $k$-symplectic manifold $(P,{\bm \omega})$ satisfying the conditions of Theorem \ref{thm:k-symplectic-bundle}  whose induced one-forms on $M$ with values in $\mathbb{R}^k$ are associated with $\mathcal{D}$. Given a  $k$-contact manifold $(M,\mathcal{D})$, with $\mathcal{D}=\ker \bm \zeta$, the {\it canonical $k$-symplectic cover} of $(M,{\bm \zeta})$ is the $k$-symplectic manifold $(\mathbb{R}_\times\times M,\bm\omega=\d s\wedge \tau^*\bm\zeta+s\tau^*\d\bm \zeta)$. 
\end{definition}

\begin{theorem}
    Every $k$-contact manifold $(M,\D)$ such that $\mathcal{D}=\ker \bm \zeta$ gives rise to a $k$-symplectic cover given by the positive-ray bundle $L^\times=\{s\bm\zeta(x)\mid s>0,\ x\in M\}$ determined by the section $(\zeta^1,\ldots,\zeta^k)$ in $\bigoplus^k\T^*M$ with respect to the canonical $k$-symplectic $\bm\omega_M$ form on $\bigoplus^k\T^*M$ and the natural multiplication by positive scalars in $L^\times$. 
\end{theorem}
\begin{proof}
    Consider the canonical $k$-symplectic form on $\boplus^k\cT M$, which admits the canonical $k$-symplectic Liouville one-form $\bm\theta_M$. Since $\bm\zeta$ never vanishes, one has  the diffeomorphism $\phi:(s,x)\in \mathbb{R}_\times \times M\mapsto s\bm \zeta(x) \in L^\times$. Let $\jmath:L^\times \hookrightarrow \bigoplus^k\cT M$ be the canonical embedding. Then,  one obtains
$$
    \phi^*(\jmath^*\bm\omega_M) = \phi^*\jmath^*(\d\bm\theta_M) = \d(\phi^* \jmath^*\bm\theta_M) = \d(s\bm\zeta) = \d s\wedge \bm \zeta + s\d\bm \zeta\,\in \Omega^2(\mathbb{R}_\times \times M,\mathbb{R}^k).
$$
Since $(\mathbb{R}_\times\times M,\phi^*\jmath^*\bm\omega_M)$ is a $k$-symplectic manifold by the converse part of Theorem \ref{thm:k-symplectic-bundle}, so is $L^\times$, which has dimension $\dim M+1$, relative to $\jmath^*{\bm\omega}_M$. Note that $(L^\times,\jmath^*\bm\omega_M)$ is homogeneous relative to the restriction to $L^\times$ of the flow of the Euler vector field of $\bigoplus^k\cT M$, let us say  $\Phi:\mathbb{R} \times \bigoplus^k\cT M\rightarrow \bigoplus^k\cT M$, which is denoted in the same manner. In fact, $L^\times$ is invariant relative to the $\Phi_s$, with $s\in \mathbb{R}$, which act freely and transitively on each fibre of $L^\times$. Note that $\jmath$ is equivariant relative to the Lie group action of $\mathbb{R}_\times$, namely $\{\Phi_{\ln t}\}_{t\in \mathbb{R}_\times}$, on $\jmath(L^\times)$ and $L^\times$, therefore $\jmath^*\bm\omega_M$ is such that 
$$
    \Phi_s^*\jmath^*\bm\omega_M = \jmath^*\Phi_s^*\bm\omega_M = \jmath^*(e^s\bm\omega_M) = e^s\jmath^*\bm\omega_M\,,
$$
which finishes our proof.

\end{proof}

\subsection{Submanifolds of \texorpdfstring{$k$}{}-contact manifolds}
\label{Sec:kConSub}

Let us extend the isotropic and Legendrian submanifolds in contact geometry to $k$-contact manifolds. This is accomplished by relating $k$-contact manifolds to $k$-symplectic covers and their submanifolds (see \cite{LV_13} for an analysis of different types of submanifolds in $k$-symplectic manifolds).

\begin{definition}\label{Def:IsoLeg}
    Consider a $k$-contact manifold $(M,\mathcal{D})$ and let $\restr{\mathcal{D}}{U} = \ker \bm \zeta$ for some $\bm \zeta\in \Omega^1(U,\mathbb{R}^k)$ on an open subset $U\subset M$. Given a subspace $F_x\subset \mathcal{D}_x$, with $x\in U\subset M$, its $k$-\textit{contact orthogonal} is defined as
    $$ F_x^{\perp_{\mathcal{D}}} = \{ v\in \mathcal{D}_x \mid \d\bm\zeta(v,w)= 0\ \text{ for every }\ w\in F_x\}\,. 
    $$
    
    Then, we say that $F_x$ is
    \begin{enumerate}[(i)]
        \item \textit{isotropic} if $F_x\subset F_x^{\perp_{\mathcal{D}}}$,        
        \item \textit{Legendrian} if $F_x$ is isotropic and, for an associated $k$-contact form $\bm\zeta$, it admits a supplementary $W_x\subset\T_xM$ such that $F_x\oplus W_x=\T_xM$ and $\restr{{\d\bm \zeta}_x}{W_x\times W_x} = 0$.
    \end{enumerate}
    A submanifold $N\subset M$ is isotropic (resp. Legendrian) if $\T_xN$ is isotropic (resp. Legendrian) for every $x\in N$.
\qeddiamond\end{definition}
Note that 
$$
\d\bm\zeta_x(v_x,w_x)=\d\bm\zeta_x(X_x,Y_x)=X_x\inn{Y}\bm\zeta-Y_x\iota_X\bm\zeta-\bm\zeta_x([X,Y]_x)=-\bm\zeta_x([X,Y]_x)\,,\qquad \forall v_x,w_x\in \mathcal{D}_x\,,
$$
where $X,Y$ are vector fields taking values in $\ker \bm\zeta$ such that $X_x=v_x,Y_x=w_x$. Hence, if $\d\bm\zeta_x(v_x,w_x)=0$, i.e. $v_x,w_x$ are {\it $k$-contact orthogonal} relative to $\bm \zeta$, this implies that for all possible vector fields $X,Y$ under given conditions,  their Lie brackets belong to $\mathcal{D}_x$ and vice versa. Hence, $v_x,w_x$ will be also $k$-contact orthogonal for any other $\bm\zeta'\in\Omega^1(U',\mathbb{R}^k)$ such that $\ker \bm \zeta'=\mathcal{D}|_{U'}$ on an open neighbourhood $U'$ of $x$ and the notion is geometric. The orthogonal and isotropic notions are independent of the associated $\bm\zeta$ because they can be defined in terms of $\rho:\mathcal{D}\times_M \mathcal{D}\rightarrow \T M/\mathcal{D}$. For an associated $k$-contact form, an isotropic supplementary to $F_x$ in $\T_xM$ exists if and only if one exists in $\mathcal D_x$: intersect the former with $\mathcal D_x$, or extend the latter by $\ker\d\bm\zeta_x$. Thus, the Legendrian condition is also intrinsic. On the other hand, it follows that $F_x$ is isotropic if, and only if, it is contained in $\mathcal{D}$ and  the Lie brackets of vector fields in $\mathcal{D}$ whose values at $x$ are in $F_x$ have Lie brackets whose values are contained in $\mathcal{D}_x$.

\begin{theorem}\label{Th:RelkSymSub}
    Let $(P = \R_\times\times M,\tau,M,\varphi,\bm\omega)$ be the $k$-symplectic cover of a co-oriented $k$-contact manifold $(M,\bm\eta)$. Consider the submanifold $\tau^{-1}(N)=\R_\times\times N\subset P$. Then,
    \begin{enumerate}[(i)]
        \item $N$ is an isotropic submanifold of $M$ if, and only if, $\tau^{-1}(N)$ is an isotropic submanifold of $P$.
        \item $N$ is a Legendrian submanifold of $M$ if, and only if, $\tau^{-1}(N)$ is a Lagrangian submanifold of $P$.
    \end{enumerate}
\end{theorem}
\begin{proof} Assume that $N$ is isotropic. Then $\tau^{-1}(N)\simeq \mathbb{R}_\times \times N$ satisfies that, given the embedding $\jmath:\tau^{-1}(N)\rightarrow P$, one has that 
$$
    \jmath^*\bm\omega =\jmath^*(\d s\wedge \tau^*\bm \eta+s\tau^*\d \bm \eta)=s\jmath^*\tau^*\d {\bm \eta}=0\,,
$$
and the lift, $\tau^{-1}(N)$, is isotropic relative to the $k$-symplectic cover. 

Let us prove the converse. If $\tau^{-1}(N)$ is isotropic relative to the restriction of $\bm\omega$ to $\T (\tau^{-1}(N))$, then $\restr{\bm\omega}{\T\tau^{-1}(N)\times_N\T\tau^{-1}(N)}=0$. Since $\partial/\partial s$ is tangent to $\tau^{-1}(N)$ and $\bm \omega = \d s\wedge\tau^*\bm \eta + s\tau^*\d\bm\eta$, then $\T\tau^{-1}(N)\subset \T\mathbb{R}_\times \times \restr{\mathcal{D}}{N}$. Analysing again the restriction of $\bm \omega$ to $\T\mathbb{R}_\times \times \restr{\mathcal{D}}{N}$, we see that $\restr{\d\bm\eta}{\T N\times_N \T N}=0$, which implies that $\T N$ is isotropic relative to $\bm \eta$.

Assume that $N$ is Legendrian. Then, $N$ is an isotropic submanifold and $\tau^{-1}(N)$ is isotropic relative to $\d(s\tau^{*}\bm\eta)$. Moreover, there exists for every $x\in N$ a subspace $S_x\subset \T_xM$ such that $\T_xN\oplus S_x=\T_xM$ and $\d\bm \eta|_{S_x\times S_x}=0$. Then, $S_x$ can be considered as a subspace $S_{(s,x)}$ in $\T_{(s,x)}(\mathbb{R}_\times\times M)$ in the natural way, with $\d s|_{S_{(s,x)}}=0$. Hence, $\restr{\bm \omega}{S_{(s,x)}\times S_{(s,x)}} = s\restr{\d\bm\eta}{S_{(s,x)}\times S_{(s,x)}}$. Hence, $S_{(s,x)}$ is isotropic relative to $\bm \omega$ and $\tau^{-1}(N)$ is a Lagrangian submanifold of $(P,\bm\omega)$. Conversely, let $W_{(s,x)}$ be an isotropic supplementary to $\T_s\mathbb R_\times\oplus\T_xN$. Its intersection with $\ker\tau^*\bm\eta$ projects injectively onto a supplementary to $\T_xN$ in $\mathcal D_x$. Since $\bm\omega=s\tau^*\d\bm\eta$ on $\ker\tau^*\bm\eta$, this projected supplementary is isotropic. Adding $\ker\d\bm\eta_x$ gives an isotropic supplementary to $\T_xN$ in $\T_xM$, so $N$ is Legendrian.

\end{proof}

Since every $k$-contact manifold $(M,\mathcal{D})$ locally admits an associated $k$-contact form, it is immediate that Theorem \ref{Th:RelkSymSub} can also be applied locally to every $k$-contact manifold.

\section{Dynamics on co-oriented \texorpdfstring{$k$}{}-contact manifolds}
\label{Sec:Dynamics}

This section studies Hamiltonian $k$-contact $k$-vector fields and Hamiltonian $k$-contact vector fields. $k$-Contact $k$-vector fields appear in the theory of partial differential equations \cite{GGMRR_20}. This section explains their relation to the $k$-symplectic cover of a co-oriented $k$-contact manifold. We also study Hamilton--De Donder--Weyl equations on $k$-contact Lie groups. Next, we introduce two new notions: $k$-contact vector fields and $k$-contact Hamiltonian vector fields, the latter being the corresponding notion when a global $k$-contact form is fixed for the associated $k$-contact distribution. It is shown that a co-oriented $k$-contact manifold gives rise to a presymplectic manifold of larger dimension, which is called a presymplectic cover. We prove that a $k$-contact Hamiltonian vector field gives rise to a Hamiltonian presymplectic vector field of its presymplectic cover projecting onto it. We apply our theory  to the study of Lie symmetries of first-order PDEs and their characteristics. Several applications of our techniques are developed. It is worth noting that the part of this work on $k$-contact Hamiltonian vector fields has been recently used in \cite{LRS_24}. 

Although most results in this section concern co-oriented
$k$-contact manifolds, they go beyond the preceding
form-based theory. In particular, the Hamiltonian description
of individual vector fields and its Lie bracket are distinct
from the Hamiltonian $k$-vector-field formalism recalled
in Section~\ref{Sec:kContHam}.

\subsection{On {\it k}-contact manifolds and homogeneous {\it k}-symplectic manifolds}
Contact manifolds can be related to symplectic covers, which allows up to a certain extent to use symplectic geometry to study some of their properties. It is interesting to investigate to what extent this can be done for $k$-contact manifolds. This section aims to give a first answer to this question.

\begin{definition}
    Let $(P,\tau,M,\varphi,\bm\omega)$ be an  $\mathbb{R}_\times$-principal $k$-symplectic manifold. Then, a \textit{$k$-symplectic homogeneous Hamiltonian} is a $k$-symplectic Hamiltonian function $H:P\to\R$ of a $k$-symplectic Hamiltonian $k$-vector field that is homogeneous, namely  $H\circ \varphi_s = sH$ for every $s \in \R_\times$.
\qeddiamond\end{definition}

\begin{theorem}\label{thm:k-vector-field-projectable}
Let $(P,\tau,M,\varphi,\bm \omega)$ be the $k$-symplectic cover of a co-oriented $k$-contact manifold $(M,\bm\eta)$ and let $\bm X$ be a $k$-contact Hamiltonian $k$-vector field. If
\begin{equation}
\label{thm:conditions}
(\mathcal{D}_\alpha)_x=\left(\bigcap_{1\leq \beta\neq \alpha\leq k}\ker \d\eta_x^\beta\right)\nsubseteq \ker \d\eta_x^\alpha,\qquad \alpha=1,\ldots,k,\qquad \forall x\in M,
\end{equation}
there exists a unique Hamiltonian $k$-vector field $\bfX_{\bm \omega}$ on $P$ related to the $k$-symplectic form $\bm \omega$ that projects onto ${\bm X}$ via  $\tau$. Moreover, every locally defined Hamiltonian function $h$ is homogeneous up to a locally constant function.
\end{theorem}

\begin{proof}
We will hereafter write  $\widehat{\phantom{m}}$ over the mathematical symbol of a structure on $M$ to represent its natural lift to $P$. Let $\bfX$ be a $k$-contact Hamiltonian  $k$-vector field. 
Let us find the conditions determining the Hamiltonian $k$-vector field ${\bf  {X}}_{\bm \omega}$ relative to the $k$-symplectic form ${\bm \omega}$. 

   A $k$-vector field ${\bf\widetilde{X}}$ on $P$ projects onto $\bm X$ on $M$ if, and only if,
    $$
        \widetilde{X}_\alpha = P_\alpha\frac{\partial}{\partial s} + X_\alpha, \qquad \alpha=1,\ldots,k\,,
    $$
    for some functions $P_\alpha\in \Cinfty(P)$ with $\alpha=1,\ldots,k$. Since $\bm \omega = \d(s\cdot \widehat{\bm{\eta}} ) = \d s\wedge\widehat{\bm{\eta}} + s\d\widehat{\bm{\eta}} \,, $ the $k$-vector field ${\bf\widetilde{X}}$ is Hamiltonian relative to $\bm \omega$ if, and only if,
    \begin{equation}\label{eq:NewHam}
        \sum_{\alpha=1}^k\inn{\widetilde{X}^\alpha}\omega^\alpha = \sum_{\alpha=1}^k(P_\alpha\widehat{{\eta}}^\alpha + s\inn{X^\alpha}\d\widehat{\eta}^\alpha )+\widehat{h}\d s = \sum_{\alpha=1}^k (P_\alpha \widehat{\eta}^\alpha + s(- \widehat{(R_\alpha h)}\widehat{\eta}^\alpha) )+ s\d \widehat{h}+ \widehat{h}\d s = \d f,
    \end{equation}
    for some $f\in \Cinfty(P)$ and $\widetilde{\bm X}$ is locally Hamiltonian if
    $$    \sum_{\alpha=1}^k\d\inn{\widetilde{X}^\alpha}\omega^\alpha = \sum_{\alpha=1}^k((\d P_\alpha)\wedge \widehat{\eta}^\alpha + P_\alpha \d\widehat{\eta}^\alpha - \d(s\widehat{R_\alpha h})\wedge  \widehat{\eta}^\alpha - s\widehat{R_\alpha h}\d\widehat{\eta}^\alpha) + \d^2(s\widehat{h}) = 0\,.
    $$
    In other words,
    \begin{equation}\label{eq:Exp1}
    \sum_{\alpha=1}^k[\d(P_\alpha-s\widehat{R_\alpha h})\wedge  \widehat{\eta}^\alpha + (P_\alpha - s\widehat{R_\alpha h}) \d \widehat{\eta}^\alpha] = 0\,.
    \end{equation}
    Contracting with a Reeb vector field $\widehat{R}_{\beta}$, one obtains
    \begin{equation}\label{eq:Reeb_contraction_beta2}
         \sum_{\alpha=1}^k\widehat{R}_\beta(P_\alpha - s\widehat{R_\alpha h}) \widehat{\eta}^\alpha - \d(P_\beta - s\widehat{R_\beta h}) = 0\,, \qquad \beta= 1,\ldots,k\,,
    \end{equation}
    and the successive contraction with $\widehat{R}_\mu$ gives
    \begin{equation} \label{eq:Reeb_contraction_mu}
        \widehat{R}_\beta(P_\mu - s\widehat{R_\mu h})- \widehat{R}_\mu(P_\beta - s\widehat{R_\beta h}) = \widehat{R}_\beta P_\mu- \widehat{ R}_\mu P_\beta = 0\,,\,\qquad \beta,\mu=1,\ldots,k.
    \end{equation}
    Then, using \eqref{eq:Reeb_contraction_beta2} and \eqref{eq:Reeb_contraction_mu} in \eqref{eq:Exp1}, one gets 
    $$
        \sum_{\alpha,\mu=1}^k \widehat{R}_\alpha(P_\mu - s\widehat{R_\mu h}) \widehat{\eta}^\mu\wedge  \widehat{\eta}^\alpha + \sum_{\alpha=1}^k(P_\alpha - s\widehat{R_\alpha h}) \d \widehat{\eta}^\alpha = \sum_{\alpha=1}^{k} (P_\alpha - s\widehat{R_\alpha h}) \d \widehat{\eta}^\alpha = 0 \,.
    $$
    Thus,
    $$
       \sum_{\alpha=1}^k (P_\alpha - s\widehat{R_\alpha h})\d {\eta}^\alpha = 0\,.    $$
   If the distributions $\mathcal{D}_\alpha=\left(\bigcap_{\beta\neq \alpha} \ker \d\eta^\beta\right)\nsubseteq \ker \d\eta^\alpha$  for $\alpha=1,\ldots,k$, (which happens for polarised $k$-contact manifolds as follows from Theorem \ref{Th:DarkMan}), then taking a vector field $X_\mu$ taking values in  $\mathcal{D}_\alpha\backslash\ker \d\eta^\alpha\neq 0$, one gets
    $$
        (P_\alpha - sR_\alpha h)\inn{X_\mu}\d {\eta}^\alpha = 0\quad\Longrightarrow\quad P_\mu - s\widehat{R_\mu h}  = 0\,,\qquad \mu=1,\ldots,k\,.
    $$
    Using the value of $P_\alpha$ in \eqref{eq:NewHam}, it follows that $f=\widehat{h}s+C$ with $C$ locally constant. Choosing $C=0$ gives a homogeneous Hamiltonian function for the $\bm\omega$-Hamiltonian $k$-vector field $\widetilde{\bm X}={\bm X}_{\bm\omega}$. Moreover, $\widetilde{\bm X}$ is unique.
    
\end{proof}
Note that if the conditions \eqref{thm:conditions} are skipped, the $k$-symplectic Hamiltonian extension exists by fixing $P_\alpha=s\widehat{R_\alpha h}$, but it does not need to be unique.

\subsection{Hamilton--De Donder--Weyl equations on \texorpdfstring{$k$}{}-contact Lie groups}\label{Sec:HamkConLie}

In most works in the literature, one is concerned with Hamilton--De Donder--Weyl  equations on polarised manifolds. Nevertheless, one can consider the not polarised case, which is more general and only briefly studied \cite{GGMRR_20}. In particular, we here study Hamilton--De Donder--Weyl equations on $k$-contact Lie groups. But our  theory, in any case, can be straightforwardly extended to many other cases.

 In the case of a $k$-contact Lie group, $\ker \bm \eta_L$ is spanned by a family of $r-k$ left-invariant vector fields on $G$.  Reeb vector fields are univocally determined by a family of equations whose solutions are invariant relative to left-invariant translations. Hence, Reeb vector fields are left-invariant vector fields. Since the Reeb vector fields are dual to $\eta_L^1,\ldots,\eta_L^k$, one can always redefine $\eta_L^{k+1},\ldots,\eta_L^r$ so as to ensure that the Reeb vector fields are $X^L_1,\ldots,X^L_k$ for a basis $X^L_1,\ldots,X^L_r$ dual to the new family $\eta^1_L,\ldots,\eta^r_L$. This happens in Examples \ref{Ex:SU3} and \ref{Ex:U2}, although this is not necessarily the case for a general basis of left-invariant one-forms. We consider anyway this case hereafter. Every $k$-vector field on $G$ can then be written as
$$
X_\alpha=\sum_{\beta=1}^rf_\alpha^\beta X_\beta^L\,,\qquad \alpha=1,\ldots,k\,,
$$
for certain functions $f_\alpha^\beta\in \Cinfty(G)$, with $\alpha=1,\ldots,k,\beta=1,\ldots,r$. Then, the Hamilton--De Donder--Weyl equations read
$$
    \sum_{\alpha=1}^k\inn{X_\alpha}\d\eta^\alpha_L = \d h-\sum_{\alpha=1}^k(R_\alpha h)\eta^\alpha_L\,,\qquad \sum_{\alpha=1}^k\iota_{X_\alpha}\eta^\alpha_L=-h\,.
$$
Since $\d h=\sum_{\alpha=1}^r (X^L_\alpha  h)\eta^\alpha_L$, one has that
$$
    \sum_{\alpha=1}^k\inn{X_\alpha}\d\eta_L^\alpha = \sum_{\alpha=1}^r(X_\alpha^Lh)\eta^\alpha_L-\sum_{\alpha=1}^k(R_\alpha h)\eta_L^\alpha = \sum_{\alpha=k+1}^r(X_\alpha^Lh)\eta_L^\alpha\,.
$$
On the other hand, $\d\eta_L^\alpha=-\frac12 \sum_{\beta,\gamma=1}^rc^\alpha_{\beta\gamma}\eta_L^\beta\wedge\eta_L^\gamma$ with $\alpha=1,\ldots,k$ for some structure constants $c^\alpha_{\beta\gamma}$ with $ \beta,\gamma=1,\ldots,r$ and $\alpha=1,\ldots,k$.
In consequence, 
$$
    -\sum_{\gamma=1}^r\sum_{\alpha=1}^k\sum_{\beta=1}^rc^\alpha_{\beta\gamma}f_\alpha^\beta\eta_L^\gamma = \sum_{\gamma=k+1}^r (X_\gamma^Lh)\eta_L^\gamma\,.
    $$
In particular, it is worth noting that one has the algebraic equations 
\begin{equation} \label{eq:alg_eq}\sum_{\alpha=1}^k\sum_{\beta=1}^rc^\alpha_{\beta\gamma}f^\beta_\alpha=0\,,\qquad \gamma=1,\ldots,k\,.
\end{equation}
Moreover, 
$$
    \sum_{\alpha=1}^k\inn{X_\alpha}\eta_L^\alpha =  \sum_{\alpha=1}^kf_\alpha^\alpha\,,
$$
and hence we obtain the linear system of partial differential equations
\begin{equation}\label{Eq:CondAlgeHDW}
\sum_{\alpha=1}^k\sum_{\beta=1}^rc^\alpha_{\beta\gamma}f^\beta_\alpha=X_\gamma^L\sum_{\alpha=1}^kf_{\alpha}^\alpha\,,\qquad \gamma=k+1,\ldots,r\,.
 \end{equation}
 At the very end, one has to obtain the integral sections of ${\bm X}=(X_1,\ldots,X_k)$. Hence, it is interesting to require their Lie brackets to be equal to zero. This implies that
 $$
    [X_\alpha,X_\beta] = \sum_{\nu=1}^r(\sum_{\mu=1}^r(f^\mu_\alpha X^L_\mu f^\nu_\beta -f^\mu_\beta X^L_\mu f^\nu_\alpha)+\sum_{1\leq \kappa<\pi\leq r}(f^\kappa_\alpha f^\pi_\beta-f_\alpha^\pi f_\beta^\kappa) c_{\kappa\pi}^\nu )X^L_\nu = 0\,,\quad 1\leq \alpha<\beta\leq k\,.
 $$
Let us apply the above theory to a particular example.  
 
Consider the matrix Lie group of the form
$$
    \mathbb{R}_\times\times H_3 = \left\{\begin{pmatrix}
        \lambda_1 & \lambda_2& \lambda_3\\
        0 & \lambda_1 &\lambda_4\\
        0 & 0 & \lambda_1
    \end{pmatrix}\ \Bigg\vert\ \lambda_1\in \mathbb{R}_\times\,,\quad  \lambda_2, \lambda_3, \lambda_4\in \mathbb{R}\right\}\,. 
 $$
 The associated Lie algebra of left-invariant vector fields admits a basis $X_1^L,\ldots,X_4^L$, whose values at ${\rm Id}$ are given by
 $$
 \begin{gathered}
 X_1^L({\rm Id})=\left(\begin{array}{ccc}1&0&0\\0&1&0\\0&0&1\end{array}\right)\,,  \quad X_2^L({\rm Id})=\left(\begin{array}{ccc}0&1&0\\0&0&0\\0&0&0\end{array}\right)\,,\\
  X_3^L({\rm Id})=\left(\begin{array}{ccc}0&0&1\\0&0&0\\0&0&0\end{array}\right)
,\quad  X_4^L({\rm Id})=\left(\begin{array}{ccc}0&0&0\\0&0&1\\0&0&0\end{array}\right)\,.
\end{gathered}
 $$
 that satisfy the non-vanishing commutation relations
$$
    [ X^L_2,X_4^L] = X^L_3\,.
$$
 Hence, a left-invariant two-contact form is given by
$$
    \bm\eta = \eta_L^1\otimes e_1+\eta_L^3\otimes e_2\,.
$$
The Reeb vector fields $X^L_1,X^L_3$ commute between themselves. On the other hand, note that $-h=f^1_1+f^3_3$.  The Hamilton--De Donder--Weyl equations in \eqref{Eq:CondAlgeHDW}, use that the indexes for the two-contact form are $1,3$ and therefore
$$
    X_2^Lh=f^{4}_{3}\,,\qquad X_4^Lh=-f^{2}_{3}\,,
$$
while the algebraic part \eqref{eq:alg_eq} is trivially satisfied because $X_1^L,X_3^L$ belong to the centre of the Lie algebra of $\mathbb{R}_\times\times H_3$. Finally, one has the integrability condition
$$
    \sum_{\mu=1}^4(f_1^\mu X^L_\mu f_3^\nu -f_3^\mu X^L_\mu f_1^\nu)+(f_1^2f_3^4-f_1^4f_3^2)\delta_3^\nu=0\,,\qquad \nu=1,\ldots,4\,.
$$

\subsection{Hamiltonian vector fields on \texorpdfstring{$k$}{}-contact manifolds and applications}

Let us show that every $k$-contact form on $M$ can be extended to a presymplectic form on $\mathbb{R}^k\times M$, a so-called presymplectic cover. Then, we develop a Hamiltonian $k$-contact vector field notion and prove that co-oriented $k$-contact manifolds admit a Hamiltonian presymplectic extension relative to its presymplectic cover. As an application of our ideas, we retrieve the theory of characteristics of Lie symmetries of first-order partial differential equations \cite{Olv_93}.

Let us start by introducing the main definitions of this section.

\begin{definition}
\label{def:etaHamiltonianvectorfield}A {\it $k$-contact vector field} relative to a $k$-contact manifold $(M,\mathcal{D})$ is a Lie symmetry $X$ of $\mathcal{D}$, namely $[X,\mathcal{D}]\subset \mathcal{D}$. If additionally $\mathcal{D}=\ker \bm \eta$ for a $k$-contact form $\bm \eta$, then $X$ is called an {\it $\bm\eta$-Hamiltonian vector field} and $-\inn{X}{\bm \eta}$ is its {\it $\bm\eta$-Hamiltonian $k$-function}. 
\qeddiamond\end{definition}

Denote by $\X_\mathcal{D}(M)$ the space of $k$-contact  vector fields on $M$ relative to $(M,\mathcal{D})$ and denote by $\X_{\bm\eta}(M)$ the space of $\bm\eta$-Hamiltonian vector fields, respectively. Meanwhile, $\Cinfty_{\bm\eta}(M,\mathbb{R}^k)$ stands for the space of $\bm\eta$-Hamiltonian $k$-functions relative to a $k$-contact form $\bm\eta$. 

The following proposition explains why $-\inn{X}{\bm \eta}$ plays the role of an $\bm\eta$-Hamiltonian $k$-function for an $\bm\eta$-Hamiltonian vector field $X$. 

\begin{proposition}
\label{prop:eta-Hamiltonian-vector-field}
A vector field $X$ on $M$ is an $\bm\eta$-Hamiltonian vector field relative to $(M,\mathcal{D}=\ker\bm\eta)$ if, and only if, 
\begin{equation}\label{eq:ConHam}
    \inn{X}\eta^\alpha=-h^\alpha\,,\qquad \inn{X}\d\eta^\alpha=\d h^\alpha-\sum_{\beta=1}^k(R_\beta h^\alpha)\eta^\beta\,,\qquad \alpha=1,\ldots,k\,,
\end{equation}
for some $k$-function ${\bm h} = \sum_{\alpha=1}^k h^\alpha\otimes e_\alpha\in\Cinfty(M,\R^k)$.
\end{proposition}
\begin{proof}
    If $X$ is an $\bm\eta$-Hamiltonian vector field, then $[X,\mathcal{D}]\subset \mathcal{D}$. This implies that $\Lie_X\eta^\alpha=\sum_{\beta=1}^kf^\alpha_\beta\eta^\beta$ for certain functions $f^\alpha_\beta\in \Cinfty(M)$ and $\alpha,\beta=1,\ldots,k$. Then, fixing $h^\alpha=-\iota_X\eta^\alpha$ for $\alpha=1,\ldots,k$, one has
    $$
        (\d \iota_X+\iota_X\d)\eta^\alpha=\sum_{\beta=1}^kf^\alpha_\beta\eta^\beta, \quad \alpha=1,\ldots,k \quad \Rightarrow \quad \iota_X\d \eta^\alpha=\d h^\alpha +\sum_{\beta=1}^kf^\alpha_\beta\eta^\beta,\qquad \alpha=1,\ldots,k\,.
    $$
    Contracting with $R_\gamma$, one obtains
    $$
        0=\inn{R_\gamma}\inn{X}\d \eta^\alpha=R_\gamma h^\alpha+f^\alpha_\gamma,\qquad \alpha,\gamma=1,\ldots,k\,,
    $$
    which implies that $f_\gamma^\alpha = -R_\gamma h^\alpha$, for $\alpha,\gamma=1,\ldots,k$, and thus
    $$
        \inn{X}\d\eta^\alpha=\d h^\alpha -\sum_{\beta=1}^k(R_\beta h^\alpha) \eta^\beta,\qquad \alpha=1,\ldots,k\,.
    $$
    The converse is immediate.
\end{proof}

Note that the work \cite{DeLeonSardonZhao2026} describes Hamiltonian vector fields relative to uniform $q$-contact structures. Since $q$-contact structures are particular cases of $q$-contact manifolds and in view of   \cite[Definition 2.3]{DeLeonSardonZhao2026} and references therein, one can see that Hamiltonian vector fields in uniform $q$-contact manifolds in Almeida's sense are particular cases of our notion, which is more intrinsic as it appears as a consequence of being a symmetry of a $k$-contact distribution and a particular global $k$-contact form.

The expressions \eqref{eq:ConHam} allow one to devise a series of results about $\bm\eta$-Hamiltonian vector fields. Note that if $X_1,X_2$ are $\bm\eta$-Hamiltonian vector fields relative to $(M,\bm \eta)$ with the same $\bm\eta$-Hamiltonian $k$-function, then the right-hand sides of \eqref{eq:ConHam} are the same for $X_1,X_2$, and  then $X_1-X_2$ takes values in  $\ker \bm\eta\cap \ker\d\bm\eta=0$. Hence, $X_1=X_2$ and an  $\bm\eta$-Hamiltonian $k$-function $\bm h$ has a unique associated $\bm \eta$-Hamiltonian vector field denoted by $X_{\bm h}$. 

\begin{proposition}\label{Prop:LieAlgebraContact}

    One has that $\mathfrak{X}_{\bm \eta}(M)$ is a Lie algebra and, for $ {\bm h}, {\bm g}\in \Cinfty_{\bm \eta}(M,\mathbb{R}^{k})$, one has that $[X_{\bm h},X_{\bm g}]$ is the $\bm \eta$-Hamiltonian vector field related to the $\bm \eta$-Hamiltonian $k$-function $-\inn{[X_{\bm h},X_{\bm g}]}\bm\eta$.
\end{proposition}
\begin{proof}

   Note that $\X_{\bm\eta}(M)$ is a Lie algebra because the $\bm \eta$-Hamiltonian vector fields are the vector fields that satisfy $[X, \ker \bm \eta ]\subset \ker \bm \eta $. In particular, this implies that $\mathfrak{X}_{\bm \eta}(M)$ is a linear space. Moreover, given two $\bm \eta$-Hamiltonian vector fields $X_1,X_2$, one has that, for $Y\in \Gamma(\ker \bm \eta)$, it follows that
    $$
        [[X_1,X_2],Y]=[X_1,[X_2,Y]]-[X_2,[X_1,Y]]\in \Gamma(\ker \bm \eta)\,.
    $$
    Hence, $[X_1,X_2]$ is an $\bm \eta$-Hamiltonian vector field and its $\bm \eta$-Hamiltonian $k$-function is $-\iota_{[X_1,X_2]}\bm \eta$ in view of equations \eqref{eq:ConHam}. 
\end{proof}

The following definition will be quite useful to make following results more elegant and clear. 

\begin{definition} Given a co-oriented $k$-contact manifold $(M,\bm \eta)$ every $k$-contact Hamiltonian $k$-function ${\bm h}\in \Cinfty_{\bm\eta}(M,\mathbb{R}^k)$ is related to its {\it Reeb derivation}, namely the vector field on $M$ of the form
$$
R_{\bm h}=\sum_{\alpha=1}^kh^\alpha R_\alpha.
$$
Moreover, let us define the so-called 
{\it Reeb tensor field} $\mathfrak{R}_{\bm\eta}:\cT M\rightarrow \cT M$ of the form
$$
\mathfrak{R}_{\bm \eta}\theta=\sum_{\alpha=1}^k\eta^\alpha \iota_{R_\alpha} \theta.
$$
\end{definition}

The following proposition is quite useful.

\begin{proposition}\label{Prop:Calculation}
Let $X_{\bm f}$ be the $\bm\eta$-Hamiltonian vector field of  $\bm f\in \Cinfty_{\bm\eta}(M,\mathbb{R}^k)$. Then,
\begin{enumerate}[(i)]
    \item $\Lie_{X_{\bm f }}\bm\eta  = -\sum_{\alpha,\beta=1}^{k}(R_\beta f^\alpha)\eta^\beta\otimes e_\alpha=-\mathfrak{R}_{\bm \eta}\d \bm f$,
    \item $X_{\bm f}\bm f = -\sum_{\alpha,\beta=1}^k(R_\beta f^\alpha)f^\beta\otimes e_\alpha=-R_{\bm f}\bm f$,
    \item $[X_{\bm f},R_\beta]=-X_{\bm g}, \text{where}\; \bm g= {\sum_{\alpha=1}^{k} R_\beta f^\alpha\otimes e_\alpha=R_\beta \bm f}$.
\end{enumerate}
\end{proposition}
\begin{proof} To prove \textit{(i)}, it is enough to note that, by Proposition \ref{prop:eta-Hamiltonian-vector-field}, one has 
$$
\Lie_{X_{\bm f}}\bm\eta=\inn{X_{\bm f}} \d \bm\eta+\d\inn{X_{\bm f}}\bm\eta=\d \bm f-\sum_{\alpha,\beta=1}^k(R_{\beta}f^{\alpha})\eta^{\beta}\otimes e_\alpha-\d\bm f=-\sum_{\alpha,\beta=1}^k\eta^\beta\iota_{R_{\beta}}\d f^{\alpha}\otimes e_\alpha=-\mathfrak{R}_{\bm \eta}\d \bm f.
$$
Meanwhile, \textit{(ii)} follows from the fact that
$$
\inn{X_{\bm f}}\d\bm \eta= \d \bm f - \sum_{\alpha,\beta=1}^k(R_\beta f^\alpha)\eta^\beta \otimes e_\alpha,
$$
and contracting both sides with $X_{\bm f}$, one gets
$$ 
0=\inn{X_{\bm f}}\left(\sum_{\alpha=1}^k\d f^{\alpha}\otimes e_\alpha\right)-\sum_{\alpha,\beta=1}^k(R_{\beta}f^{\alpha})\inn{X_{\bm f}}\eta^{\beta}\otimes e_\alpha 
$$
and
$$
0=\sum_{\alpha=1}^{k} \left(X_{\bm f}f^{\alpha} + \sum_{\beta=1}^k(R_{\beta}f^{\alpha})f^{\beta}\right) \otimes e_\alpha\,,$$
which gives $X_{\bm f}{\bm f} = -\sum_{\alpha,\beta=1}^k(R_\beta f^\alpha)f^\beta\otimes e_\alpha=-R_{\bm f}\bm f$.

To prove \textit{(iii)}, one computes the contraction of $[X_{\bm f},R_\beta]$ with $\bm\eta$ and its differential. In particular,
$$
\inn{[X_{\bm f},R_\beta]}\eta^\alpha=\left(\Lie_{X_{\bm f}}\inn{R_{\beta}}-\inn{R_{\beta}}\Lie_{X_{\bm f}}\right)\eta^{\alpha}=R _\beta f^\alpha\,,\qquad \alpha,\beta=1,\ldots,k\,.
$$
Moreover, one has that 
$$
\Lie_{[X_{\bm f},R_\beta]}\eta^\alpha=(\Lie_{X_f}\Lie_{R_\beta}-\Lie_{R_\beta}\Lie_{X_{\bm f}})\eta^\alpha=-\Lie_{R_\beta}\Lie_{X_{\bm f}}\eta^\alpha\,,\qquad \alpha,\beta=1,\ldots,k,
$$
and
$$
\Lie_{[X_{\bm f},R_\beta]}\eta^\alpha=\Lie_{R_\beta}\left(\sum_{\gamma=1}^{k}(R_\gamma f^\alpha)\eta^\gamma\right)=\sum_{\gamma=1}^{k}R_{\beta}(R_\gamma f^\alpha)\eta^\gamma\,,\qquad \alpha,\beta=1,\ldots,k.
$$
Then,
$$
\iota_{[X_{\bm f},R_\beta]}\d\eta^\alpha=\sum_{\gamma=1}^{k}(R_\gamma R_{\beta}f^\alpha)\eta^\gamma-\d (R_\beta f^\alpha)\,,\qquad \alpha,\beta=1,\ldots,k.
$$
Hence, 
$$\iota_{[X_{\bm f},R_\beta]}\d\eta^\alpha=\d (-R_\beta f^\alpha)-\sum_{\gamma=1}^{k}R_\gamma(-R_\beta f^\alpha)\eta^\gamma\,,\qquad \alpha,\beta=1,\ldots,k.
$$
Thus, the vector field $[X_{\bm f},R_\beta]$ is $\bm\eta$-Hamiltonian and its corresponding $\bm\eta$-Hamiltonian $k$-function is $-\sum_{\alpha=1}^kR_\beta f^\alpha\otimes e_\alpha=-R_\beta \bm f$.
\end{proof}

\begin{example} Note that the Reeb vector fields of a $k$-contact form $\bm \eta$ are $\bm\eta$-Hamiltonian. Indeed, the Reeb vector field $R_\alpha$ is $\bm\eta$-Hamiltonian with an $\bm\eta$-Hamiltonian $k$-function $-e_\alpha$. In this manner, one can see that the result (iii) in Proposition \ref{Prop:Calculation} is natural, as $[X_{\bm f},R_\beta]$ is the commutator of two $\bm \eta$-Hamiltonian vector fields. 
\end{example}

\begin{definition}
    Given a co-oriented $k$-contact manifold $(M,\bm \eta)$, we define a bracket on $\Cinfty_{\bm \eta}(M,\mathbb{R}^{k})$ of the form \begin{equation}\label{eq:kcontactbracket}
        \{\bm h_1,\bm h_2\}_{\bm\eta} = \bm\eta([X_{\bm h_1},X_{\bm h_2}])\,,\qquad \forall \bm h_1,\bm h_2\in \Cinfty_{\bm \eta}(M)\,.
\end{equation}\qeddiamond\end{definition}
\begin{corollary}\label{Cor:MorLieBrackett}
    The bracket introduced in the previous definition induces a Lie algebra isomorphism $\bm h\in \Cinfty_{\bm \eta}(M,\mathbb{R}^{k})\mapsto -X_{\bm h}\in\X_{\bm \eta}(M)$. 
\end{corollary}
\begin{proof}
    Recall that every $\bm \eta$-Hamiltonian $k$-function determines a unique $\bm \eta$-Hamiltonian vector field, so the mapping is well-defined. Moreover, every $\bm \eta$-Hamiltonian vector field determines uniquely its $\bm \eta$-Hamiltonian $k$-function. Hence, two different $\bm \eta$-Hamiltonian $k$-functions cannot be associated with the same $\bm \eta$-Hamiltonian vector field. Thus, the mapping is injective, and a surjection by definition of $\mathfrak{X}_{\bm\eta}(M)$. Hence, the mapping is a bijection. Moreover, expressions \eqref{eq:ConHam} shows that the mapping is also linear.

Let us prove that the mapping is also a Lie algebra morphism.    Given two $\bm\eta$-Hamiltonian $k$-functions $\bm h,\bm g$, its bracket is the $\bm\eta$-Hamiltonian $k$-function, $\{\bm h,\bm g\}_{\bm\eta}$, of  the $\bm\eta$-Hamiltonian vector field $-[X_{\bm h},X_{\bm g}]$. Then, 
    $$
        \{\bm f,\{\bm g,\bm h\}_{\bm \eta}\}_{\bm \eta}+\{\bm g,\{\bm h,\bm f\}_{\bm \eta}\}_{\bm \eta}+\{\bm h,\{\bm f,\bm g\}_{\bm \eta}\}_{\bm \eta}
    $$
    is the $\bm\eta$-Hamiltonian $k$-function of
    $$
        [X_{\bm f},[X_{\bm g},X_{\bm h}]]+[X_{\bm g},[X_{\bm h},X_{\bm f}]]+[X_{\bm h},[X_{\bm f},X_{\bm g}]]=0\,.
    $$
    Hence,
    $$
        \{\bm f,\{\bm g,\bm h\}_{\bm \eta}\}_{\bm \eta}+\{\bm g,\{\bm h,\bm f\}_{\bm \eta}\}_{\bm \eta}+\{\bm h,\{\bm f,\bm g\}_{\bm \eta}\}_{\bm \eta}=0\,,
    $$
    and $\{\cdot,\cdot\}_{\bm\eta}$ is a Lie bracket. Hence, the mapping is a Lie algebra isomorphism. 
\end{proof}

For practical purposes, it is worth rewrite the Lie bracket \eqref{eq:kcontactbracket} as follows
\begin{multline*}
    \{\bm h_1,\bm h_2\}_{\bm\eta}=(\Lie_{X_{\bm h_1}}\iota_{X_{\bm h_2}}-\iota_{X_{\bm h_2}}\Lie_{X_{\bm h_1}})\bm\eta\\=-\Lie_{X_{\bm h_1}}\bm h_2+\iota_{X_{\bm h_2}}\left[\sum_{\alpha,\beta=1}^k(R_\beta h_1^\alpha)\eta^\beta\otimes e_\alpha\right]=-X_{\bm h_1}\bm h_2-\sum_{\alpha,\beta=1}^k(R_\beta h_1^\alpha)h_2^\beta\otimes e_\alpha\,.
\end{multline*}
And finally
$$
\{{\bm h}_1,{\bm h}_2\}_{\bm \eta}=-\Lie_{X_{{\bm h}_1}}{\bm h_2}-R_{{\bm h}_2}{\bm h}_1,\qquad \forall {\bm h}_1,{\bm h}_2\in \Cinfty_{\bm \eta}(M,\mathbb{R}^k).
$$

For the sake of clarity, let us resume previous results in the following corollary, whose proof is immediate.
\begin{corollary}\label{Prop:MorphismkContact}One has the following exact sequence of Lie algebra morphisms
$$\label{morph-sequence}
0\longrightarrow \Cinfty_{\bm\eta}(M,\mathbb{R}^k)\overset{\phi}{\longrightarrow} \X_{\bm \eta}(M)\longrightarrow 0\,,
$$
where $\Cinfty_{\bm\eta}(M,\mathbb{R}^k)$ is endowed with the Lie bracket \eqref{eq:kcontactbracket} and $\phi({\bm f})=-X_{\bm f}$. 
\end{corollary}

It is remarkable that in the contact case, a dissipated quantity (see \cite{GGMRR_20a} and references therein) is a function such that $X_hf=-f(Rh)$. Hence, the expression above can be used to define an analogue of dissipated quantities in $k$-contact geometry.
\begin{definition} A {\it dissipated $k$-contact $k$-function} relative to a $k$-contact Hamiltonian system $(M,\bm \eta,\bm h)$ is an  ${\bm \eta}$-Hamiltonian $k$-function ${\bm f}\in \Cinfty_{\bm \eta}(M,\mathbb{R}^k)$ such that
$$
X_{\bm h}{\bm f}=-R_{\bm f}{\bm h}.
$$
\end{definition}
In particular every ${\bm h}\in \Cinfty_{\bm \eta}(M,\mathbb{R}^k)$ is dissipated relative to itself, namely $X_{
\bm h}{\bm h}=-R_{\bm h}{\bm h}$. Moreover, 
$$
\{{\bm h},{\bm f}\}=0\qquad \Longleftrightarrow   \qquad 
X_{\bm h}\bm f=-R_{\bm f}{\bm h}.
$$
In other words, dissipated $k$-contact $k$-functions quantities are $\bm \eta$-Hamiltonian $k$-functions that $\bm \eta$-commute with ${\bm h}$. Moreover, a $\bm \eta$-Hamiltonian $k$-function is dissipated if and only if its associated $\bm\eta$-Hamiltonian vector field is a Lie symmetry of $X_{\bm h}$.

\subsection{Lifting {\it k}-contact manifolds to other types of manifolds}

Co-oriented $k$-contact manifolds can be lifted in different manners to other types of differential geometric structures. In each case, the relation allows one to study the properties of some aspects of the initial $k$-contact manifold and their related $\bm \eta$-Hamiltonian vector fields via geometric structures of other nature. For instance, contact manifolds can be studied by means of symplectic manifolds in a manifold of larger dimension \cite{Arn_89,GG_22}. In the case of $\bm \eta$-Hamiltonian vector fields, the following proposition is a key to study them via presymplectic techniques. It is worth noting that we do not assume any condition on the rank of the presymplectic form, which is here assumed to be depending on the point of the manifold. 
\begin{theorem}
\label{thm:presymplecticmanifold}
    Every $k$-contact manifold $(M,\mathcal{D})$ gives rise to a presymplectic manifold $(E,\omega)$ where $\tau:E\rightarrow M$ is a vector bundle over $M$ of rank $k$ such that, there exists a family of trivializations $\tau^{-1}(U_a)\simeq \mathbb{R}^k\times U_a$ for a cover $\bigcup_{a}U_{a}=M$ so that $\omega|_{\tau^{-1}(U_a)}=\d\big(\sum_{\alpha=1}^kz_\alpha \widehat{\eta}^\alpha_a\big)$ for some $k$-contact form $\bm\eta_a$ associated with $\mathcal{D}$ on $U_a$.  
\end{theorem}
\begin{proof}
    Consider a  cover of $M$ given by open sets $U_a$ where $\restr{\mathcal{D}}{U_a} = \ker \bm \eta_a$ for some $k$-contact form $\bm \eta_a$ on $U_a$. 
    On each open $U_a$, define $\theta_a=\sum_{\beta=1}^kz_\beta {\eta}_{a}^\beta$. On the overlaps, $U_{ab}=U_a\cap U_b\neq \emptyset$, one has $\eta_{a}^\alpha=\sum_{\beta=1}^k(A_{ab})^\alpha_\beta \eta _{b}^\beta$ for some functions $(A_{ab})^\alpha_\beta\in \Cinfty(U_{ab})$, with $\alpha,\beta=1,\ldots,k$ giving rise to an invertible $k\times k$ matrix $A_{ab}$. Let us construct the fibre bundle $\tau:E\rightarrow M$ by considering the transition maps $z^a_\alpha=\sum_{\beta=1}^k(A_{ab}^{-1})_\alpha^\beta z^b_\beta$ for $\alpha=1,\ldots,k$. Hence, $$
    \theta_a=\sum_{\alpha=1}^kz^a_\alpha{\eta}_a^\alpha=\sum_{\alpha,\beta=1}^k(A^{-1}_{ab})_\alpha^\beta z^b_\beta{\eta}_a^\alpha=\sum_{\alpha,\beta,\gamma=1}^k(A^{-1}_{ab})_\alpha^\beta z^b_\beta(A_{ab})^\alpha_\gamma{\eta}_b^\gamma=\sum_{\beta=1}^kz^b_\beta {\eta}_b^\beta=\theta_b$$ on $U_{ab}$. This gives rise to a well-defined vector bundle over $M$ of rank $k$ with a globally defined differential one-form $\theta$ such that $\restr{\theta}{\tau^{-1}(U_a)} = \theta_a$ for every $U_a$ of the cover of $M$. The differential $\d\theta$ is the presymplectic form we are willing to define.

    Moreover, $\omega$ may be degenerate since the dimension of $\mathbb{R}^k\times M$ is not even in general. 
\end{proof}

\begin{proposition}\label{Prop:LiftkHamVectorField}
    Every co-oriented $k$-contact manifold $(M,\bm\eta)$ gives rise to a presymplectic manifold $\left(\mathbb{R}^k\times M,\omega=\d(\sum_{\alpha=1}^kz_\alpha\widehat{\eta}^\alpha)\right)$. Moreover, if $X$ is an $\bm\eta$-Hamiltonian vector field with the $\bm\eta$-Hamiltonian $k$-function $\bm h$, then the vector fields on $\mathbb{R}^k\times M$ that project onto $X$ and are locally Hamiltonian relative to $\omega$ are given locally by
    $$
        Y=\sum_{\alpha=1}^k\Bigg(\sum_{\beta=1}^k z_\beta \widehat{(R_\alpha h^\beta)}+ \widehat{(R_\alpha f)}\Bigg)\frac{\partial}{\partial z_\alpha}+X\,,
    $$
    where $f\in\Cinfty(U)$ on a sufficiently small open subset $U\subset M$ is any function such that $R_1 f,\dots ,R_k f\in \Cinfty(U)$ are first integrals of the vector fields on $M$ taking values in $\ker\bm \eta$ and $\sum_{\alpha=1}^k(R_\alpha f)\d \eta^\alpha=0$. In particular, if $\ker \bm \eta$ admits a polarisation, then $R_1f=\cdots=R_kf=0$ and $Y$ is uniquely defined.
\end{proposition}
\begin{proof} The extension of the co-oriented $k$-contact manifold is an application of Theorem \ref{thm:presymplecticmanifold}.  
    Let us show that there exist vector fields $Y$ on $\mathbb{R}^k\times M$ that project onto $X$ and are Hamiltonian relative to $\omega$. As in previous results, a hat on a mathematical structure on $M$ stands for its natural lift  to $\mathbb{R}^k\times M$. If $Y$ projects onto $X$, then it is of the form
    $$
        Y=\sum_{\beta=1}^kf_\beta\frac{\partial}{\partial z_\beta}+X\,,
    $$
    for some functions $f_1,\ldots,f_k\in \Cinfty(\mathbb{R}^k\times M)$ and
    \begin{multline*}
        \iota_Y\omega=\iota_Y\sum_{\alpha=1}^k(\d z_\alpha\wedge \widehat{\eta}^\alpha+z_\alpha\d \widehat{\eta}^\alpha)= \sum_{\alpha=1}^k(f_\alpha\widehat{\eta}^\alpha+\widehat{h}^\alpha \d z_\alpha +z_\alpha\iota_X\d \widehat{\eta}^\alpha)\\
        =\sum_{\alpha=1}^k(f_\alpha\widehat{\eta}^\alpha+\widehat{h}^\alpha \d z_\alpha +z_\alpha\d \widehat{h}^\alpha)-\sum_{\alpha,\beta=1}^kz_\alpha\widehat{(R_\beta h^\alpha )}\widehat{\eta}^\beta\,.
    \end{multline*}
    In other words,
    $$
        \iota_Y\omega=\d\Bigg(\sum_{\alpha=1}^kz_\alpha \widehat{h}^\alpha \Bigg)+\sum_{\alpha=1}^k\Bigg(f_\alpha-\sum_{\beta=1}^kz_\beta\widehat{(R_\alpha h^\beta)}\Bigg)\widehat{\eta}^\alpha\,.
    $$
    Then, $\iota_Y\omega$ is locally Hamiltonian if, and only if, 
    $$
        \d\left( \sum_{\alpha=1}^k\Bigg(f_\alpha-\sum_{\beta=1}^kz_\beta \widehat{(R_\alpha  h^\beta)}\Bigg)\widehat{\eta}^\alpha\right) = 0\,.
    $$
    Denoting $g_\alpha= f_\alpha-\sum_{\beta=1}^kz_\beta \widehat{(R_\alpha h^\beta)}$ for $\alpha=1,\ldots,k$, we obtain that $0=\sum_{\alpha=1}^k[\d g_\alpha\wedge \widehat{\eta}^\alpha+g_\alpha\d \widehat{\eta}^\alpha]$ implies that
    $$
        \iota_{\widehat{R}_\gamma}\iota_{\partial_{z_\beta }}\d\left(\sum_{\alpha=1}^kg_\alpha\widehat{\eta}^\alpha\right)=\parder{g_\gamma}{z_\beta} = 0\,,\qquad \gamma,\beta=1,\ldots,k\,.
    $$
    These expressions give
    $$
        \frac{\partial f_\gamma}{\partial z_\beta}=\widehat{R_\gamma h^\beta}, \qquad \gamma,\beta=1,\ldots,k\,,\quad\Longrightarrow\quad f_\gamma=\sum_{\beta=1}^k z_\beta \widehat{(R_\gamma h^\beta)}+\widehat{\zeta_\gamma}\,, \qquad \forall \gamma,\beta=1,\ldots,k\,,
    $$
where $\partial \widehat{\zeta}_\gamma /\partial z_\alpha=0$ for $\alpha,\gamma=1,\ldots, k$. Therefore, $0=\sum_{\alpha=1}^k[\d \zeta_\alpha\wedge \eta^\alpha+\zeta_\alpha\d \eta^\alpha]$. Contracting the last expression with two Reeb vector fields, it follows that $R_\beta \zeta_\gamma=R_\gamma\zeta_
\beta$ for $\beta, \gamma = 1,\ldots,k$. Since $\sum_{\alpha=1}^k\zeta_\alpha\eta^\alpha$ is closed, the Poincar\'e lemma gives, locally on a sufficiently small $U\subset M$, a function $f\in\Cinfty(U)$ such that $\d f=\sum_{\alpha=1}^k\zeta_\alpha\eta^\alpha$. Hence $\zeta_\alpha=R_\alpha f$ for $\alpha=1,\ldots,k$. Moreover, $\sum_{\alpha=1}^k(R_\gamma\zeta_\alpha)\eta^\alpha-\d \zeta_\gamma=0$ for $\gamma=1,\ldots,k$. Then, the contractions of the last expression with vector fields taking values in  $\ker \bm\eta$ show that $\zeta_1,\ldots,\zeta_k$ are first integrals for the elements of $\Gamma(\ker \bm \eta)$. Moreover, $0=\sum_{\alpha=1}^k[\d \zeta_\alpha\wedge \eta^\alpha+\zeta_\alpha\d \eta^\alpha]=\sum_{\alpha,\beta=1}^{k}(R_{\alpha}\zeta_{\beta})\eta^{\beta}\wedge\eta^{\alpha}+\sum_{\alpha=1}^{k}\zeta_{\alpha}\d\eta^{\alpha}=\sum_{\alpha=1}^{k}\zeta_{\alpha}\d\eta^{\alpha}$. 

Note that if $\ker \bm \eta$ has a polarisation, then the Darboux theorem shows that $\d \eta^1,\ldots,\d\eta^k$ are linearly independent and $R_1 f=\dotsb =R_kf=0$. 
\end{proof}

\subsection{Characteristics of Lie symmetries of  \texorpdfstring{$k$}{}-contact Hamiltonian \texorpdfstring{$k$}{}-functions and applications}

Let us recall that $J^1(M,E)$ is a polarised $k$-contact manifold relative to its Cartan distribution. Our above theory allows us to describe the theory of characteristics of Lie symmetries \cite{Olv_93} as a $k$-contact Hamiltonian theory, at least locally.  

\begin{proposition}\label{prop:HamLift}
    Consider a vector field $X$ on the total space $E$ of a fibre bundle $p:E\rightarrow M$ of rank $k$ over an $m$-dimensional manifold $M$ given in adapted coordinates $\{x^i,y^\alpha\}$ on $U\subset E$ by
    $$
        X=\sum_{i=1}^m\xi^i\partial_i+\sum_{\alpha=1}^k\zeta^\alpha\partial_\alpha\,,
    $$
    for some functions $\xi^1,\ldots,\xi^m,\zeta^1,\ldots,\zeta^k\in\Cinfty(U)$.
    Then, there exists a unique $k$-contact Hamiltonian vector field $\pr X$ on $(p^{1}_{0})^{-1}(U)\subset J^{1}(M,E)$ projecting onto $X$ via $p^1_0:J^1(M,E)\rightarrow E$. In adapted coordinates $\{x^i,y^\alpha,y^\alpha_i\}$ in $J^1(M,E)$, the vector field $\pr  X$, the so-called {\it prolongation} of $X$ on $E$ to $J^1(M,E)$, takes the local form
    $$
        \pr X = -\sum_{i=1}^m\sum_{\alpha=1}^k\Bigg[\bigg(\partial_{x^{i}}+\sum_{\beta=1}^ky^\beta_i\partial_{y^{\beta}}\bigg)h^\alpha\Bigg]\partial_{y^\alpha_i}+X\,,
    $$
    having locally an $\bm \eta$-Hamiltonian $k$-function given by
    \begin{equation}\label{Eq:CharHAm}
        \bm h=\sum_{\alpha=1}^k\bigg(-\zeta^\alpha+\sum_{i=1}^my^\alpha_i\xi^i\bigg)\otimes e_\alpha\,.
    \end{equation}
\end{proposition}
\begin{proof}
    One finds that  $\pr X$ has to satisfy that
    \begin{equation}\label{eq:Conprol}
        \inn{\pr  X}\eta^\alpha=-h^\alpha,\qquad \inn{\pr X}\d\eta^\alpha=\d h^\alpha-\sum_{\beta=1}^{k} (R_\beta h^\alpha)\eta^\beta\,,\qquad \alpha=1,\ldots,k\,.
    \end{equation} 
  Since $J^1(M,E)$ is polarised, one can choose  Darboux coordinates  $\eta^\alpha=\d y^\alpha-\sum_{i=1}^my^\alpha_i\d x^i$ and $\d \eta^\alpha=\sum_{i=1}^m\d x^i\wedge \d y^\alpha_i$ for $\alpha=1,\ldots,k$ given by adapted coordinates. This gives
    $$
        {\pr}\,X=\sum_{\alpha=1}^k\sum_{i=1}^m\phi^\alpha_i\partial_{y^\alpha_i}+\sum_{i=1}^m\xi^i\partial_{x^{i}}+\sum_{\alpha=1}^k\zeta^\alpha\partial_{y^{\alpha}}\,.
    $$
    Hence, the above conditions show that $h^\alpha=-\zeta^\alpha+\sum_{i=1}^m y^\alpha_i\xi^i$, with $\alpha=1,\ldots,k$, and
    \begin{equation}
        \inn{\partial_{x^{i}}}\iota_{\pr  X}\d\eta^\alpha=-\phi^\alpha_i=\partial_{x^{i}}h^\alpha+\sum_{\beta=1}^k(R_\beta h^\alpha)y^\beta_i =
        \left(\partial_{x^{i}}+\sum_{\beta=1}^k y^\beta_i\partial_{y^{\beta}}\right)h^\alpha\,,
    \end{equation}
    for $\alpha=1,\ldots,k\,,i=1,\ldots,m\,.$ 
    These conditions ensure that expressions \eqref{eq:Conprol} are locally satisfied. 
    
    Note that above considerations are local, but ${\rm pr}\,X$ is uniquely defined. Hence, it must coincide in the intersection of domains of different adapted coordinates, which imply that ${\rm pr}\,X$ is globally defined. Meanwhile, $\bm \eta$ is defined only locally and different adapted coordinates give different $\bm\eta$. Hence, $\bm h$ is defined only locally.
\end{proof}

Note that the $\bm\eta$-Hamiltonian $k$-function for $\pr  X$ is the characteristic function of $\pr  X$ occurring in the theory of Lie symmetries of differential equations \cite{Olv_93}. It is defined locally, as it is related to adapted coordinates.

As a consequence of the above, one has that the zero set of ${\bm h}$ is an invariant of the flow of $\pr X$. Let us study this in the following proposition.

\begin{proposition}\label{Prop:Inv}
    Given a vector field $X$ on the total space of the fibre bundle  $E\rightarrow M$, the zero set of its $\bm\eta$-Hamiltonian $k$-function is invariant under the local flow of $\pr X$ on $J^1(M,E)$.
\end{proposition}
\begin{proof}
    Since $\bm h$ is the $\bm\eta$-Hamiltonian $k$-function of $\pr  X$, then
    $$
        0 = \inn{\pr  X} \inn{\pr  X} \d \eta^\beta = \inn{\pr  X} \bigg(\d h^\beta-\sum_{\alpha=1}^k(R_\alpha h^\beta)\eta^\alpha\bigg) \qquad \Longrightarrow \qquad \pr  X h^\beta = -\sum_{\alpha=1}^k(R_\alpha h^\beta)h^\alpha\,,
    $$
    for $\beta=1,\ldots,k$. Along each integral curve, this is a homogeneous linear system for the components of $\bm h$. Uniqueness implies that ${\bm h}^{-1}(0,\ldots,0)$ is invariant under the local flow, without requiring it to be a submanifold.
\end{proof}

\begin{example}
    Denote $x_0 = t, x_1 = x, x_2 = y, x_3 = z$ the space-time variables, which play the role of independent variables. Let us consider the first-order partial differential equations given by a simple Hamilton--Jacobi equation on $\mathbb{R}\times \mathbb{R}^4
    \to \mathbb{R}^4$ of the form
    \begin{equation}\label{eq:HJ}
        u_0 + u^2_1 + u^2_2 + u^2_3=0\,,
    \end{equation}
    where $u$ is the dependent variable and $u_i = \tparder{u}{x_i}$ for $i=0,1,2,3,$. For simplicity, we write $\partial_0,\ldots,\partial_3$ for the time-space derivatives and $\partial_{u_0},\ldots,\partial_{u_3}$ for the partial derivatives relative to the induced variables in $J^1(\mathbb{R}^4,
    \mathbb{R}^5)$. The Lie symmetries of the Hamilton--Jacobi equations \eqref{eq:HJ} were studied in \cite{BP_76,Yeh_00}. Consider for instance the lifts to $J^1(\mathbb{R}^4,\mathbb{R}^5)$ of the Lie symmetries of \eqref{eq:HJ} given by
    \begin{gather}
        P_0 = \partial_0\,,\qquad P_j = \partial_j\,,\qquad P_u = \partial_u\,,\qquad J_{kj} = x_k\partial_j - x_j\partial_k\,, \\
        D^{(1)} = x_{0}\partial_0 + \frac12 \sum_{i=1}^3x_i\partial_i\,,\qquad D^{(2)}=\frac{1}{2}\sum_{i=1}^3x_{i}\partial_{i}+u\partial_{u}\\
        G_j^{(1)} = x_{0}\partial_j + \frac12 x_j\partial_u\,,\qquad G_j^{(2)} = \frac12 x_j\partial_0+u\partial_j\,, \\
        A^{(1)} = x_{0}^2\partial_0 + x_{0}\sum_{i=1}^3x_i\partial_i + \frac14 \sum_{i=1}^3x_i^2\partial_u\,,\qquad
        A^{(2)}=\frac{1}{4}\sum_{i=1}^3x_i^2\partial_{0}+u\sum_{i=1}^3x_i\partial_{i}+u^2\partial_{u}\,, \\
        K_{j}=\frac{1}{2}x_{0} x_{j}\partial_{0}+\sum_{i=1}^{3}\left(\frac{1}{2}x_{j}x_{i}+\left(x_{0}u-\frac{1}{4}\sum_{i=1}^{3}x_{i}^{2}\right)\delta_{ij}\right)\partial_{i}+\frac{1}{2}x_{j}u\partial_{u}\,.
    \end{gather}
    for $k,j=1,2,3.$
    By our results, which retrieve the standard theory in \cite{Olv_93}, they read
    {\small \begin{gather*}
        \pr P_0 = \partial_0\,,\qquad \pr P_j = \partial_j\,,\qquad
        \pr P_u = \partial_u\,, \qquad
        \pr J_{kj} = x_k\partial_j - x_j\partial_k - u_j\partial_{u_k} + u_k\partial_{u_j}\,,\\
        \pr D^{(1)} = x_0\partial_0 + \frac12 \sum_{i=1}^3x_i\partial_i  -u_0\partial_{u_0} - \frac12\sum_{i=1}^3 u_i\partial_{u_i}\,,\\
        \pr D^{(2)}=\frac{1}{2}\sum_{j=1}^3x_{j}\partial_{j}+u\partial_{u}+u_0\partial_{u_{0}}+\frac{1}{2}\sum_{i=1}^{3}u_{i}\partial_{u_{i}}\,,
        \\
        \pr G_j^{(1)} = x_{0}\partial_j + \frac12 x_j\partial_u - u_j\partial_{u_0} +\frac12\partial_{u_j} \,,\\
        \pr G_j^{(2)} =\frac12 x_j\partial_0 + u\partial_j  - u_0u_j\partial_{u_0} - \sum_{i=1}^3\left(\frac{1}{2}u_0\delta_{ij}+ u_iu_j\right)\partial_{u_i}\,, \\
         \pr A^{(1)} = x_{0}^2\partial_0 + x_{0}\sum_{i=1}^3x_i\partial_i + \frac14 \sum_{i=1}^3x_i^2\partial_u - \left(2x_{0}u_0 + \sum_{i=1}^3 x_iu_{i}\right)\partial_{u_0} + \sum_{i=1}^{3}\left(\frac12x_i - x_{0}u_i\right)\partial_{u_i}\,, 
                 \end{gather*}
        \begin{gather*}
         \pr A^{(2)}=\frac{1}{4}\sum_{i=1}^3x_{i}^2\partial_{x_0}
+u\sum_{i=1}^3x_{i}\partial_{x_{i}}
+u^2\partial_u\\+u_0\left(2u-\sum_{i=1}^3x_{i}u_i\right)\partial_{u_0}+\sum_{i=1}^3\left[u_i\left(u-\sum_{j=1}^3x_{j}u_j\right)-\frac12x_{i}u_0\right]\partial_{u_i}\,, \\
         \pr K_j=\frac{x_0x_j}{2}\partial_{x_0}
+\sum_{i=1}^{3}\left[
\frac{x_jx_i}{2}
+\left(x_0u-\frac14\sum_{k=1}^{3}x_k^2\right)\delta_{ij}
\right]\partial_{x_i}
+\frac{x_ju}{2}\partial_u-u_j(u+x_0u_0)\partial_{u_0}\\+ \sum_{k=1}^{3}\left[
\frac12\delta_{kj}\left(
u-x_0u_0-\sum_{i=1}^{3}x_i u_i
\right)
+u_j\left(\frac{x_k}{2}-x_0u_k\right)
\right]\partial_{u_k}.
    \end{gather*}}
    In fact, since one has the contact one-form $\d u-u_t\d t-u_x\d x-u_y\d y-u_z\d z$, their characteristics are
    \begin{gather}
        h_{P_0} = u_{0}\,,\qquad
        h_{P_j} = u_j\,,\qquad
        h_{P_u} = -1\,,\qquad
        h_{J_{kj}} = u_jx_k-x_ju_k\,,\\
        h_{D^{(1)}} = x_0 u_0+\frac{1}{2}\sum_{i=1}^{3}u_{i}x_{i}\,,\qquad
        h_{D^{(2)}} = -u+\frac{1}{2}\sum_{i=1}^{3}u_{i}x_{i}\,,\\
        h_{G^{(1)}_{j}} = x_0u_{j}-\frac{x_{j}}{2}\,, \qquad h_{G^{(2)}_j} = u u_j+\frac{u_0x_j}{2}\,,\\
        h_{A^{(1)}} = x_0^2u_0+x_0\sum_{i=1}^3u_ix_i-\frac 14\sum_{i=1}^3x_i^2\,,\quad h_{A^{(2)}}=-u^{2}+u\sum_{i=1}^{3}x_iu_{i}+\frac{1}{4}u_{0}\sum_{i=1}^{3}x_{i}^{2}\,, \\
        h_{K_j}=-\frac12 x_j u
+\frac12 x_0x_j u_0
+\frac12x_j\sum_{i=1}^{3}x_i u_i+\left(-\frac14\sum_{i=1}^{3}x_i^2
+x_0u
\right)u_j,
    \end{gather}
    which shows that the zero sets given by Proposition \ref{Prop:Inv} are in fact invariant. Moreover, one sees that previous characteristics are the Hamiltonian functions given by \eqref{Eq:CharHAm}.
    \finish
\end{example}

\begin{example} Let us study the Dirac equation
\begin{equation}\label{eq:DiracEquation} \Bigg(i\hbar \sum_{\mu=0}^3\gamma^\mu\partial_\mu-mc\Bigg)\psi=0\,,
\end{equation}
on $\mathbb{R}^4$, where $\gamma^0,\ldots,\gamma^3$ are the Dirac matrices, $\hbar$ is the Planck constant, $\partial_0=\partial_t$, $\partial_1=\partial_x$, $\partial_2=\partial_y$, $\partial_3=\partial_z$, and we use units with $c=1$. The solutions take the form
$$
    \psi:x\in \mathbb{R}^4\longmapsto \big(\psi_1(x),\ldots,\psi_4(x)\big)\in  \mathbb{C}^4\,.
$$
The Dirac equation \eqref{eq:DiracEquation}  can be written as a first-order system of PDEs on eight dependent variables $\psi_a^R,\psi_a^I$, where $\psi_a=\psi_a^R+i\psi_a^I$, and four independent variables $\{x^0=t,x^1=x,x^2=y,x^3=z\}$. Geometrically, the problem can be studied in the first-order jet bundle $J^1(\mathbb{R}^4,\mathbb{R}^4\times \mathbb{R}^8)$ with the globally defined eight-contact form
$$
\bm\eta=\sum_{a=1}^4\left(\d \psi^{R}_{a}-\sum_{i=0}^3\left(\psi^{R}_{a}\right)_{i}\d x^i\right)\otimes e^R_a+\sum_{a=1}^4\left(\d \psi^{I}_{a}-\sum_{i=0}^3\left(\psi^{I}_{a}\right)_{i}\d x^i\right)\otimes e^I_a\,.
$$
The Dirac equation \eqref{eq:DiracEquation} in complex form can be rewritten as follows
$$
    \begin{pmatrix}
        i\hbar\partial_t - m & 0 & i\hbar\partial_z & i\hbar\partial_x + \hbar\partial_y \\
        0 & i\hbar\partial_t - m & i\hbar\partial_x - \hbar\partial_y & -i\hbar\partial_z \\
        -i\hbar\partial_z & -i\hbar\partial_x - \hbar\partial_y & -i\hbar\partial_t - m & 0 \\
        -i\hbar\partial_x + \hbar\partial_y & i\hbar\partial_z & 0 & -i\hbar\partial_t - m
    \end{pmatrix}\begin{pmatrix}
        \psi_1 \\ \psi_2 \\ \psi_3 \\ \psi_4
    \end{pmatrix} = \begin{pmatrix} 0 \\ 0 \\ 0 \\ 0 \end{pmatrix}\,,
$$
where the matrix multiplication is understood as the application of each element of the matrix as an operator on the coordinates of $\psi$. 

One can see that $\partial_0,\ldots,\partial_3$ on $\mathbb{R}^{4+8}$ can be  lifted to $J^1(\mathbb{R}^4,\mathbb{R}^{4+8})$. Their lifts are Lie symmetries of \eqref{eq:DiracEquation}. By Proposition \ref{prop:HamLift}, they are $\bm \eta$-Hamiltonian vector fields. Then, their $\bm \eta$-Hamiltonian $k$-functions are
$$
    {\bm h}_i=\sum_{a=1}^4\left(\left(\psi^{R}_a\right)_i\otimes e^R_a+\left(\psi^{I}_{a}\right)_i\otimes e^I_a\right)\,,\qquad i=0,\ldots,3\,.
$$
By Proposition \ref{Prop:Inv}, the vector field $\partial_i$ is tangent to the submanifold ${\bm h}_i^{-1}(0,\ldots,0)$ for $i=0,\dotsc,3$. In any case, note that any other Lie symmetry of these equations should leave the Cartan distribution invariant. Since $\bm\eta$ is globally defined, these vector fields will be $\bm\eta$-Hamiltonian too.
\finish
\end{example}

\section{An additional \texorpdfstring{$k$}{}-symplectification}\label{Sec:kSym}

This section proposes an additional manner of extending a $k$-contact manifold to a $k$-symplectic one. In particular, a $k$-contact form  gives rise to $k$-symplectic forms on $\R^k_{\times}\times M$ in a natural and analogous manner, but with somehow different properties than in previous approaches. 

Let us first define the following projections
$$
\begin{gathered}
{\rm pr}_k:(z^1,\ldots,z^k;x)\in \mathbb{R}_\times^k\times M\longmapsto x\in M\,,\\
{\rm pr}^1_k:(z^1,\ldots,z^k;x)\in \mathbb{R}_\times^k\times M\longmapsto  (z^1,\ldots,z^k) \in \mathbb{R}^k_\times\,,\\
{\rm pr}:(z,x)\in \mathbb{R}_\times \times M\longmapsto x\in M\,.\\
\end{gathered}
$$

Every vector field $X$ on $M$ has a unique lift to
$\mathbb R_\times\times M$
(respectively, $\mathbb R_\times^k\times M$)
that projects onto $X$ and has zero component along
the first factor. These lifts will be intensively used hereafter.

\begin{proposition}\label{prop:symplectization}
    If $(M,{\bm\eta})$ is a $k$-contact manifold , then 
    \begin{equation}\label{eq:Ext}
    \left(\widetilde M = \R^k_\times\times M,\bm\omega = \sum_{\alpha=1}^{k}\d(z^{\alpha}{\rm pr}_k^*\eta^{\alpha})\otimes e_\alpha\right)
    \end{equation} 
is a $k$-symplectic manifold. Conversely, if \eqref{eq:Ext} is a $k$-symplectic manifold for some $\bm \eta \in \Omega^{1}(M,\mathbb{R}^{k})$, then $\ker \d{\bm\eta}\cap \ker {\bm\eta}=0$. If additionally  $\rank \ker \d{\bm\eta}=k$ and $\ker \bm \eta\neq 0$, then $(M,\bm\eta)$ is a $k$-contact manifold.
\end{proposition}

\begin{proof}Recall first that we intensively use the standard isomorphism $\T_{(z,x)}\widetilde{M}\simeq \T_z\mathbb{R}^k_\times\oplus \T_xM$ to understand elements of $\T_z\mathbb{R}^k_\times$ and $\T_xM$ as elements of $\T_{(z,x)}\widetilde{M}$ and conversely. This will simplify the notation of the proof.
    
Let us now prove the direct part. A tangent vector $X_{\tilde{x}}\in \T_{\tilde{x}}\widetilde{M}$, with $\tilde{x}=(z,x)\in \mathbb{R}^k_\times\times M$, can be decomposed in a unique manner as
$$
    X_{\tilde{x}} = \sum_{\alpha=1}^k\left(d^\alpha\parder{}{z^\alpha}\bigg|_{\tilde{x}} +c^\alpha \widetilde{R}_\alpha\bigg|_{\tilde{x}}\right)+Y_{\tilde{x}}\,,
$$
 for some constants $c^1,\ldots,c^k,d^1,\ldots,d^k\in \mathbb{R}$, where $Y_{\tilde{x}}\in \ker{\rm pr}_k^*\bm\eta_{x}\cap \ker \T_{\tilde{x}}{\rm pr}^1_k$ and $\widetilde{R}_{1},\ldots,\widetilde{R}_{k}$ are the Reeb vector fields of $\bm \eta$ on $M$ understood as vector fields in $\widetilde{M}$. Let us prove that, if $X_{\tilde{x}}\in \ker
 \bm \omega_{\tilde{x}}$, then $X_{\tilde{x}}=0$. In particular, if
$
\bomega_{\tilde{x}}(X_{\tilde{x}},\cdot)=0
$, and since
\[
\bm\omega=\sum^{k}_{\alpha=1}\d(z^\alpha{\rm pr}_k^{*} \eta^{\alpha}) \otimes e_{\alpha}=\sum^{k}_{\alpha=1}(z^\alpha\d ({\rm pr}_k^{*}\eta^{\alpha})+\d z^\alpha \wedge {\rm pr}_k^{*}\eta^{\alpha}) \otimes e_{\alpha},
\]
then
$$
0=\bomega_{\tilde{x}}(X_{\tilde{x}},(\widetilde{R}_{\alpha})_{\tilde{x}})=d^\alpha\otimes e_\alpha, \qquad 0=\bomega_{\tilde{x}}((\partial/\partial z^\alpha)_{\tilde{x}},X_{\tilde{x}})=c^\alpha\otimes e_\alpha\,, \qquad \alpha=1,\ldots,k\,,
$$
which leads to $X_{\tilde{x}}=Y_{\tilde{x}}$. Then, 
$$  
\inn{X_{\tilde{x}}}\bomega_{\tilde{x}}=\iota_{Y_{\tilde{x}}}\bomega_{\tilde{x}}=\sum_{\alpha=1}^{k}z^\alpha\iota_{Y_{\tilde{x}}}\d({\rm pr}_k^{*}\eta^\alpha)_{\tilde{x}}\otimes e_\alpha=0\,,
$$
and, since $(z^{1},\ldots,z^{k}) \in \mathbb{R}^{k}_{\times}$, the tangent vector $Y_{\tilde{x}}$ takes values in $\ker\d({\rm pr}^{*}_{k}\bm\eta)_{\tilde{x}}$, and amounts to a tangent vector in $\ker(\d\bm\eta)_x$ by definition. Since $\bm\eta$ is a $k$-contact form, then $Y_{\tilde{x}}=0$ and, then, $X_{\tilde{x}} = 0$.
Thus, $\ker\bm\omega = \{0\}$, and since $\bm\omega$ 
is closed, then  $\bm \omega $ is an $\mathbb{R}^k$-valued closed two-form defining a $k$-symplectic manifold on $\widetilde M$. 

Let us prove the converse. First, let us show that if $\bm\omega$ is $k$-symplectic, then $\ker \bm \eta\cap\ker \d\bm \eta=0$. To prove this, let us use the reciprocal statement: if there is a non-zero tangent vector at  $x \in M$ taking values $\ker \bm\eta_{x}\cap \ker\d\bm\eta_{x}\subset \T M$, then $\bm \omega$ has a non-trivial kernel.  

Consider the lift $Y_{\tilde{x}}$ to $\tilde{x} \in \mathbb{R}^{k}_{\times} \times M$ of a non-zero tangent vector $Y_x$ taking values in $\ker \bm\eta_{x} \cap \ker\d\bm\eta_{x}$. Then,
\begin{equation*}
    \sum_{\alpha=1}^k\iota_{Y_{\tilde{x}}}\omega^{\alpha}_{\tilde{x}}\otimes e_{\alpha}=\sum_{\alpha=1}^k\iota_{Y_{\tilde{x}}}(z^\alpha\d({\rm pr}_k^{*} \eta^{\alpha})_{\tilde{x}}+(\d z^\alpha)_{\tilde{x}} \wedge ({\rm pr}_k^{*}\eta^{\alpha})_{\tilde{x}}) \otimes e_{\alpha}=0,
\end{equation*}
but $\ker \bm \omega_{\tilde{x}}={0}$ and this is contradiction. Hence, $\ker\bm \eta_x\cap \ker \d\bm\eta_x=0$. 

Finally, $\bm\eta_x$ is injective on $\ker\d\bm\eta_x$, which has dimension $k$. Thus $\bm\eta_x$ has rank at least $k$, and hence exactly $k$ since it takes values in $\mathbb R^k$. Together with $\ker\bm\eta\neq0$, the defining conditions of a $k$-contact form follow. 

\end{proof}

It turns out that the condition $\rank\ker\d\bm\eta = k$ is necessary for $\bm\eta$ to give rise to a $k$-contact form on $M$. The following example illustrates this point.

\begin{example} We want to find a $k$-symplectic form on $\mathbb{R}^{k}_{\times} \times M$ for $k\geq 2$ given by
$$
\bm\omega=\sum_{\alpha=1}^k[z^\alpha \d({\rm pr}_k^{*}\eta^\alpha)+\d z^\alpha\wedge {\rm pr}_k^{*}\eta^\alpha]\otimes e_\alpha\,,
$$
where $\ker \d\bm\eta\cap \ker \bm\eta=0$, but $\bm\eta$ is not a $k$-contact form on $M$.  
Consider the manifold $\bigoplus_{\alpha=1}^k\cT\mathbb{R}^n$ with adapted coordinates $\{x^i,p^\alpha_i\}$ and define 
$$
{\bm \eta} =\left( \d p_1^{2}-\sum_{i=1}^{n}p^1_i\d x^i\right)\otimes e_1+\dotsb+\left(\d p_1^{k}-\sum_{i=1}^{n}p^{k-1}_i\d x^i\right)\otimes e_{k-1}+\left(\d p_1^{1}-\sum_{i=1}^{n}p^{k}_i\d x^i\right)\otimes e_{k}.
$$
Hence,
$$
\d\bm\eta= \sum_{\alpha=1}^k\sum_{i=1}^n\d x^i\wedge \d p^\alpha_i\otimes e_\alpha
$$
and $\ker \d\bm\eta=0$. Assume that $\bm\omega_{\tilde{x}}(X_{\tilde{x}},\cdot)=0$ for a tangent vector $X_{\tilde{x}}$ at  $\tilde{x}\in \mathbb{R}^k_\times\times \bigoplus_{\alpha=1}^k\cT\mathbb{R}^n$. Then, $0=\bm\omega_{\tilde{x}}(X_{\tilde{x}},\partial/\partial z_\alpha)=-\iota_{X_{\tilde{x}}}{\rm pr}_k^*\eta^\alpha \otimes\,e_{\alpha}$ for $\alpha=1,\ldots,k$ and $X_{\tilde{x}}$ takes values in  $\ker{\rm pr}^*_k\bm\eta_{\tilde{x}}$. We can write $X_{\tilde{x}}=\sum_{\alpha=1}^{k}f^\alpha (\partial/\partial z^\alpha)_{\tilde{x}}+Y_{\tilde{x}}$, where $Y_{\tilde{x}}$ is a tangent vector on $\ker({\rm pr}^*_k\bm \eta)_{\widetilde{x}}$ understood as a tangent vector in $\mathbb{R}_\times^k\times M$. Therefore,
$$
0=\iota_{X_{\tilde{x}}}\bm\omega=\sum_{\alpha=1}^{k}(f^\alpha{\rm pr}_k^*\eta^\alpha+z^\alpha\iota_{Y_{\widetilde{x}}}\d {\rm pr}_k^*\eta^\alpha)\otimes e_\alpha
$$
Since $\eta^\alpha$ never belongs to the image of $({\d\eta^\alpha})^\flat:\T M\rightarrow \T^*M$ and $z^{\alpha}\neq 0$ for $\alpha=1,\ldots,k$, one has $\iota_{Y_{\widetilde{x}}}\d{\rm pr}_k^*\bm\eta=0$ and $f^1=\ldots=f^k=0$. Hence, $\bm\omega_{\tilde{x}}$ is $k$-symplectic but $\bm \eta$ is not a $k$-contact form because $\ker \d\bm\eta$ has not  rank $k$.
\finish
\end{example}

It is immediate that the notion of a polarised co-oriented $k$-contact manifold $(M,\bm \eta,\mathcal{V})$, where  $\dim M=n+k+nk$ and  $\mathcal{V}$ is a polarisation of rank $nk$ gives rise to a polarised $k$-symplectic manifold $\mathbb{R}_\times\times M$ of dimension $(n+1)(k+1)$ with a polarisation of rank $(n+1)k$. In adapted Darboux coordinates $\{y^\alpha,x^i,p_i^\alpha\}$ for $M$, one has $\bm \eta=\sum_{\alpha=1}^k(\d y^\alpha-p_i^\alpha \d x^i )\otimes e_\alpha$, while in the $k$-symplectification  to $\mathbb{R}_\times\times M$, one has that $s,x^i$ can be considered as `base coordinates' and its associated momenta are the $y^\alpha$ and $sp_i^\alpha$, respectively,  with $\bm\omega=\d(s{\rm pr}^*\bm \eta)=\sum_{\alpha=1}^k[\d s\wedge \d y^\alpha+\d x^i\wedge \d (sp_i^\alpha)]\otimes e_\alpha$, respectively.

Apart from the fact that $\mathbb{R}_\times\times M$ is simpler than $\mathbb{R}^k_\times \times M$, the following proposition shows another reason to prefer the first to the second. In particular, the extension of a $k$-symplectic Hamiltonian vector field is only possible under very restrictive conditions of the Hamiltonian $k$-function of the vector field.

\begin{proposition}
\label{prop:k-symplectizated-hamiltonian}
Let $X$ be the $\bm\eta$-Hamiltonian vector field relative to  $(M,\bm \eta)$ with $\bm\eta$-Hamiltonian $k$-function $\bm h = \sum_{\alpha=1}^k h^\alpha\otimes e_\alpha$ such that $R_{\beta}h^{\alpha}=0$ for $\alpha\neq\beta$. Then, the vector field
\begin{equation}\label{eq:LiftkHam}
\widetilde X = \sum_{\alpha,\beta=1}^k z^\beta \pr_{k}^{*}(R_\beta h^\alpha)\parder{}{z^\alpha} + X
\end{equation}
is an $\bm\omega$-Hamiltonian vector field on $\widetilde{M}=\mathbb{R}_{\times }^{k} \times M$ that projects onto $X$ via ${\rm pr}_k:\widetilde M\to M$, with Hamiltonian $k$-function $\widetilde{\bm h}=\sum_{\alpha=1}^{k}z^\alpha {\rm pr}_k^*h^\alpha\otimes e_\alpha$. It is unique for this prescribed Hamiltonian $k$-function.
\end{proposition}
\begin{proof}
Under given assumptions, one has
{\small \begin{align*}
\iota_{\widetilde{X}}\bomega &= \sum^{k}_{\alpha=1}\iota_{\widetilde{X}}(z^\alpha \d (\pr_{k}^{*}\eta^{\alpha})+\d z^\alpha \wedge \pr_{k}^{*}\eta^{\alpha}) \otimes  e_{\alpha}  \\
&=\sum^{k}_{\alpha=1}\left( z^\alpha \iota_{\widetilde{X}}\d(\pr_{k}^{*}\eta^\alpha)+\sum_{\beta=1}^kz^\beta \pr_{k}^{*}(R_\beta h^\alpha) \pr_{k}^{*}\eta^\alpha+\pr_{k}^{*}h^\alpha \d z^\alpha \right)\otimes e_{\alpha}\\
&=\sum^{k}_{\alpha=1}\left( z^\alpha \left[\d(\pr_{k}^{*}h^\alpha)-\sum_{\beta=1}^k{\rm pr}^*_k((R_\beta h^\alpha)\eta^\beta)\right]+\sum_{\beta=1}^kz^\beta \pr_{k}^{*}(R_\beta h^\alpha) \pr_{k}^{*}\eta^\alpha+(\pr_{k}^{*}h^\alpha) \d z^\alpha \right)\otimes e_{\alpha}\\
&=\sum^{k}_{\alpha=1}\d(z^\alpha \pr_{k}^{*}h^{\alpha})\otimes e_\alpha.
\end{align*}}
It is very important to stress that we have used $R_\beta h^\alpha=0$ for $\alpha\neq \beta$ in the last step, as this will turn to be a key point in the possibility of making the lift of the $\bm\eta$-Hamiltonian vector field. Indeed, under this assumption,  $\widetilde{X}$ is a Hamiltonian vector field relative to the $k$-symplectic form $\bm\omega$. Moreover, it projects onto $X$ relative to ${\rm pr}_{k*}$, as required.

Uniqueness for the prescribed Hamiltonian follows from $\ker\bm\omega=0$: if two vector fields have Hamiltonian $\widetilde{\bm h}$, their difference belongs to $\ker\bm\omega$ and therefore vanishes.
\end{proof}
The condition $R_\beta h^\alpha=0$ for $\alpha\neq\beta$ ensures that the displayed lift has the prescribed Hamiltonian $\widetilde{\bm h}$. It disappears for $k=1$.


\section{Conclusions and outlook}
\label{Sec:Conclusions}

This work develops a distributional approach to $k$-contact geometry. The intrinsic characterisation uses the Levi bracket and commuting transverse Lie symmetries. Polarisations yield local normal forms and a local identification with first-order jet bundles. The examples also show how $k$-contact distributions differ from ordinary contact distributions, including their relation to Engel and Goursat distributions.

The compact examples distinguish the topology of Reeb leaves from the existence of closed leaves. The counterexample excludes a direct extension requiring a compact Reeb leaf for every co-oriented $k$-contact manifold, while the product formulation is equivalent to the classical Weinstein conjecture. The geometric covers relate the distributional formalism to $k$-symplectic and presymplectic geometry, and the Hamiltonian constructions give applications to Lie groups and characteristics of first-order PDE symmetries.

Related applications and extensions have been studied in \cite{SobczakFrelik_26,HLM_25,LRS_24,10.1007/978-3-032-03924-8_36,deLucas2026JetBundles,DLRS_26,IzquierdoLopezDeLeonSchiavoneSoto_2025,DLK_26}. These include thermodynamic field theories, integrable systems, higher-order jets and nonlocal symmetries, and coisotropic submanifolds.

\addcontentsline{toc}{section}{Acknowledgements}
\section*{Acknowledgements}

This work is devoted to the memory of our colleague and friend, Miguel C. Mu\~noz--Lecanda, who passed away on the Christmas's Eve of 2023. He was, and will always be, a source of inspiration for us.

We would like to extend our gratitude to S. Vilariño and B.M. Zawora for taking the time to read this paper and providing comments and suggestions that contributed to improving it. We also thank several comments during the conferences ``XXVI Encuentro de Invierno: Geometr\'{i}a, Din\'amica y Teoría de Campos" and ``RSME's 7th Congress of Young Researchers" by M. de Le\'on, E. Mart\'inez, A. Ibort, among others.

X. Rivas acknowledges partial financial support from a mikrogrant funded by the IDUB program of the Faculty of Physics (University of Warsaw) for his research stay. X. Rivas also acknowledges partial financial support from the Spanish Ministry of Science and Innovation, grants  PID2021-125515NB-C21, and RED2022-134301-T of AEI, and Ministry of Research and Universities of the Catalan Government, project 2021 SGR 00603.

We finally want to thank two anonymous referees for their careful reading of the work, comments, and deep insights that helped us to clarify the presentation and improve our initial manuscript.

\begingroup
\setlength{\emergencystretch}{2em}

\endgroup

\end{document}